\documentclass[final,3p,times]{elsarticlenew}
\usepackage{amsfonts}
\usepackage{amsmath,amssymb,amsthm,amsthm}
\usepackage{mathtools}
\usepackage{dsfont}
\usepackage{etoolbox}
\usepackage{color}

\newcommand{\dT}{\frac{\partial\phantom{\theta}}{\partial\theta}}
\newcommand{\dP}{\frac{\partial\phantom{\phi}}{\partial\phi}}

\newcommand{\sem}{\textnormal{sem}}

\newtheorem{theorem}{Theorem}[section]

\newtheorem{example}[theorem]{Example}
\newtheorem{lemma}[theorem]{Lemma}

\newtheorem{proposition}[theorem]{Proposition}

\numberwithin{equation}{section}

\newcommand{\dM}{b}

\def\RR{{\mathbb{R}}}
\def\M{\mathcal{M}}
\def\NN{{\mathbb{N}}}
\def\N{{\mathbb{N}}}

\newcommand{\norm}[2]{{\left\|#1\right\|}_{#2}}

\newcommand{\inv}{^{-1}}

\newcommand{\mc}[1]{\mathcal{#1}}
\newcommand{\mbb}[1]{\mathbb{#1}}

\newcommand{\ol}[1]{\overline{#1}}
\newcommand{\ti}[1]{\tilde{#1}}
\newcommand{\tn}[1]{\textnormal{#1}}

\renewcommand{\l}{\left}
\renewcommand{\r}{\right}

\newcommand{\dV}{ \ d\textup{vol}}
\newcommand{\pderiv}[2]{\frac{\partial #1}{\partial #2}}

\newcommand{\gdot}[1]{\left\langle #1\right\rangle_g}

\newcommand{\veps}{\epsilon}

\begin{document}

\title{Diffusion Maps for Embedded Manifolds with Boundary with Applications to PDEs}

\author[rvt]{Ryan Vaughn \corref{cor}}
\ead{rv2152@nyu.edu}
\author[rvt2]{Tyrus Berry}
\ead{tberry@gmu.edu}
\author[rvt2]{Harbir Antil}
\ead{hantil@gmu.edu}
\cortext[cor]{Corresponding author}
\address[rvt]{Courant Institute for the Mathematical Sciences, New York University, New York, NY 10012}
\address[rvt2]{Department of Mathematical Sciences, George Mason University, Fairfax, VA 22030, USA.}
\address[rvt2]{Department of Mathematical Sciences, George Mason University, Fairfax, VA 22030, USA.}
\date{\today}



\begin{abstract}
Given only a finite collection of points sampled from a Riemannian manifold embedded in a Euclidean space, in this paper 
we propose a new method to numerically solve elliptic {and parabolic} partial differential equations (PDEs) supplemented with boundary 
conditions.  Since the construction of triangulations on unknown manifolds can be both difficult and expensive,  
both in terms of computational and data requirements, our goal is to solve these problems without a triangluation. Instead, we rely only on using the sample points to define quadrature formulas on the unknown manifold. 
Our main tool is the diffusion maps algorithm. We re-analyze this well-known method in a variational
sense for manifolds with boundary.  Our main result is that the variational diffusion maps graph Laplacian is a consistent estimator of the Dirichlet energy on the manifold. This improves upon previous results and provides a rigorous justification of the well-known relationship 
between diffusion maps and the Neumann eigenvalue problem.  Moreover, using semigeodesic coordinates we derive the first uniform asymptotic expansion of the diffusion maps kernel integral operator for manifolds with boundary.  This expansion relies on a novel lemma which relates the extrinsic Euclidean distance to the coordinate norm in a normal collar of the boundary.  We then use a recently developed method of estimating the distance to boundary function (notice that the boundary location is assumed to be unknown) to construct a consistent estimator for boundary integrals.  Finally, by combining these various estimators, we illustrate how to impose Dirichlet {and Neumann} 
conditions for some common PDEs based on the Laplacian. Several numerical examples illustrate our theoretical findings. 
\end{abstract}
\maketitle
\section{Introduction}

The goal of this paper is to analyze the diffusion maps algorithm in a weak (variational) form and to introduce a completely 
rigorous method to solve elliptic {and parabolic} partial differential equations (PDEs) with boundary conditions. These PDEs are posed on an $m$-dimensional Riemannian manifold $(\mathcal{M},g)$ embedded in an ambient Euclidean space via $\iota:\mathcal{M}\to\RR^d$.  We assume that the Riemannian metric $g$ on $\M$ is inherited from the ambient space $\RR^d$ via the embedding.

Motivated by applications to machine learning or emergent structures in high-dimensional problems such as inertial manifolds, we will assume that we have no explicit description of the embedded Riemannian manifold.  Instead, we assume only that we have a collection of sample points, $\{x_i\}_{i=1}^N \subset \iota(\mathcal{M}) \subset \RR^d$ which, together with equal weights, form a consistent weighted quadrature rule.  Namely, for any square integrable function $f \in L^2(\mathcal{M},g)$, we assume that 
\begin{align}\label{quad} \lim_{N\to\infty} \frac{1}{N}\sum_{i=1}^N f(x_i) = \int_{\mathcal{M}} f(x)q(x) \, d\textup{vol} \end{align}
almost surely. In the statistical context the weight function $q$ is called the sampling density.  Ultimately our method will be independent of $q$, meaning that we do not require any specific density and if a grid of samples is used this grid is not required to be uniformly spaced.  This is critical when the nodes are data points (which are typically not sampled from the density that corresponds to the volume form), but is also an advantage for synthetic data sets where creating uniform grids on manifolds can be challenging (uniform in this context means that \eqref{quad} holds with a constant $q$).  There are several situations where this may arise:
\begin{itemize}
\item Random data on an unknown manifold, where $q$ is the sampling density.
\item Attractors and inertial manifolds for dynamical systems, where $q$ is the invariant measure.
\item Known but complex domains that are difficult to mesh, and difficult to sample uniformly.
\item Known but moderate dimensional manifolds if one cannot afford a mesh.
\end{itemize}

There is a wide literature starting with the Laplacian eigenmaps \cite{laplacianeigenmaps} and the diffusion maps \cite{diffusion} algorithms which give a method of approximating the intrinsic Laplacian operator on an unknown manifold.  In this manuscript we let $\Delta$ be the negative definite Laplacian, also know as the Laplace-Beltrami operator.  The basic strategy for estimating the Laplacian starts with a kernel function $K(\epsilon,x,y)$ which approximates the heat kernel on a manifold, for example $K(\epsilon,x,y) = e^{-\frac{|x-y|^2}{4\epsilon^2}}$, where we choose $\epsilon^2$ in the denominator so that $\epsilon$ has units of distance.  Then for any $f$ we can estimate the integral operator  
\begin{equation}\label{kernelintegral}\mathcal{I}f(x) := \int_{y\in \mathcal{M}} K(\epsilon,x,y)f(y)q(y) \, d\textup{vol} \end{equation}
by our quadrature formula
\begin{align}\label{kernelop} \mathcal{K}f(x) \equiv \frac{1}{N} \sum_{i=1}^N K(\epsilon,x,x_i)f(x_i)  
  = \mathcal{I}f(x) + \textup{Error}_{\textup{Quad}}(N,f,q) \end{align}
where the $\textup{Error}_{\textup{Quad}}$ term is assumed to go to zero as $N\to \infty$.  In fact, when the data, $x_i$, are independent identically distributed random variables it can be shown that $\textup{Error}_{\textup{Quad}} = \mathcal{O}(N^{-1/2})$ with high probability \cite{singer,variablebandwidth}.  The qualifier `with high probability' is required because the data set is random and there is a finite (but extremely small) probability of all the data points landing in for instance, a small ball on the manifold. This would clearly lead to a significant error in the quadrature formula in most cases.  However, the probability of all such high-error events can be made arbitrarily small as $N^{-1/2}$ approaches zero \cite{variablebandwidth, singer}.  Finally, notice that if we represent $f$ by a vector $\vec f_i = f(x_i)$ then we can represent $\mathcal{K}$ with the matrix {with entries} ${\bf K}_{ij} = K(\epsilon,x_i,x_j)$ so that $({\bf K}\vec f \,)_i = \mathcal{K}f(x_i)$.

The intuition behind the kernel function is that the exponential decay localizes the integral to an ${\epsilon}$-ball around $x$, and in this neighborhood the Euclidean distance $d_{\RR^d}(x,y)=\norm{x-y}{\RR^d}$ is close to the geodesic distance $d_g(x,y)$ (under appropriate assumptions on the manifold and embedding).  Thus, as $\epsilon \to 0$ the integral $\mathcal{I}f(x)$ can be shown \cite{laplacianeigenmaps} to converge to the semigroup $e^{\epsilon^2\Delta}$ associated to the intrinsic (negative definite) Laplace-Beltrami operator $\Delta$, so that
\[ \mathcal{K}f(x)  = m_0 \epsilon^{m} e^{\epsilon^2\Delta}(fq)(x) + \mathcal{O}(\epsilon^{m+2}) + \textup{Error}_{\textup{Quad}}(N,f,q). \]
In fact, a more detailed asymptotic analysis \cite{diffusion,hein2007convergence} reveals that for any kernel function of the form
\[ K(\epsilon,x,y) = k\left(\frac{\norm{x-y}{\RR^d}^2}{\epsilon^2}\right) \] 
where $k:[0,\infty)\to [0,\infty)$ is sufficiently regular and has exponential decay ($k(z)\leq a^{-bz}$ for some $a,b>0$), we have
\begin{align}\label{interiorexpansionintro} \epsilon^{-m}\mathcal{K}f(x)  
    &= m_0 f(x)q(x) + \epsilon^2 \frac{m_2}{2}\left(\omega(x)f(x)q(x) + \Delta(fq)(x)\right) \nonumber \\ 
    &\hspace{15pt}+ \mathcal{O}(\epsilon^4) + \epsilon^{-m}\textup{Error}_{\textup{Quad}}(N,f,q). 
\end{align}
for all $x$ with distance greater than $\epsilon$ from the boundary.  The constants $m_0$ and $m_2$ are the \textit{zeroth} and \textit{second moments} of the chosen kernel functions, namely,
\begin{align}\label{coeff}
m_0 &:= \int_{z \in \mathbb{R}^m} k(|z|^2) \, dz \hspace{20pt} \textup{and} \hspace{20pt} m_2 := \int_{z \in \mathbb{R}^m} z_i^2 k(|z|^2) \, dz
\end{align}
where the domain of integration is determined by the intrinsic dimension of the manifold, $m$, rather than the extrinsic dimension of the embedding space, $d$.

The expansion \eqref{interiorexpansionintro} is commonly used for estimating the density function \cite{manifoldKDE3,manifoldKDE2, manifoldKDE1} by applying the operator $\mathcal{K}$ to the constant function $f \equiv 1$ to find,
\[  \frac{\epsilon^{-m}}{m_0}\mathcal{K}1(x) = q(x) + \mathcal{O}(\epsilon^2) + \frac{\epsilon^{-m}}{m_0}\textup{Error}_{\textup{Quad}}(N,f,q). \]
However, it is well-known that for manifolds with boundary this does not hold. This leads to the well known bias of Kernel Density Estimators (KDEs) near the boundary \cite{boundary00,boundary93,boundary96,boundary05,boundary14,boundary99}.  In \cite{BERRY2017}, the authors developed a method to estimate the distance to the boundary of a manifold from data.  They then used this estimate to correct the bias of the KDE near the boundary.  A significant advance of the method of \cite{BERRY2017} is that the location of the boundary does not need to be known beforehand, and it is effectively learned from the data.  

A much more challenging and powerful use of \eqref{interiorexpansionintro} is for estimating the Laplace-Beltrami operator, $\Delta$, and this expansion is the key component of justifying the diffusion maps algorithm \cite{diffusion}.  The Laplace-Beltrami operator is ubiquitous, especially when a physical process is modeled using PDEs.  For such applications, it is critical that one be able to specify the appropriate boundary conditions.  Moreover, while it is widely observed that the diffusion maps algorithm produces Neumann eigenfunctions \cite{diffusion}, this empirical observation has not been adequately explained.  In this paper we will show that the estimator defined by the diffusion maps algorithm is consistent in the weak sense even for manifolds with boundary, and that Neumann eigenfunctions are observed because of the naturality of the Neumann boundary conditions for the eigenproblem.  Finally, allowing arbitrary boundary conditions to be specified requires us to introduce a new tool, namely a \emph{boundary integral estimator}, which may have uses beyond these applications, and the consistency of this estimator is one of our key results.

In order to solve diffusion type PDEs in the weak-sense and specify boundary conditions we need consistent discrete estimators of the following bilinear forms, 
\begin{align}\label{ops} 
    \mathcal{G}(\phi,f) = \int_{\mathcal{M}} \phi f \, d\textup{vol} \; , \hspace{30pt} 
    \mathcal{E}(\phi,f) = \int_{\mathcal{M}} \nabla \phi \cdot \nabla f \, d\textup{vol} \; , \hspace{30pt} 
    \mathcal{B}(\phi,f) = \int_{\partial\mathcal{M}} \phi f \, d\textup{vol}_\partial \; .  
\end{align} 
where $\mathcal{G}$ and $\mathcal{E}$ correspond to mass and stiffness matrices respectively. The $L^2$ inner product on the boundary, $\mathcal{B}$, arises in case of Neumann or Robin boundary conditions \cite{FJSayas_TSBrown_MEHassell_2019a}.
The first bilinear form, $\mathcal{G}$, is simply the $L^2$ inner product on the manifold, which can be represented by a diagonal matrix with entries $D_{ii}=(N q(x_i))^{-1}$ since by \eqref{quad} we have,
\[ \lim_{N\to\infty} \vec\phi^\top D \vec f = \lim_{N\to\infty} \sum_{i=1}^{N} D_{ii}\vec\phi_i \vec f_i = \lim_{N\to\infty}\frac{1}{N}\sum_{i=1}^N \frac{\phi(x_i)f(x_i)}{q(x_i)} = \int_{\mathcal{M}} \phi f \, \dV = \mathcal{G}(\phi,f) \]
where $q$ can be estimated as in \cite{BERRY2017} (the method is summarized in \ref{background}).  

In this paper, we show that the graph Laplacian (as constructed by the diffusion maps algorithm \cite{diffusion}) is a consistent estimator of the Dirichlet energy, $\mathcal{E}$, even for manifolds with boundary (Theorem \ref{thm3}).  Moreover, we introduce a novel consistent estimator for the boundary integral (Theorem \ref{thm2}). These results rely on the first uniform asymptotic expansion of kernel integral operators in a neighborhood of the boundary (Theorem \ref{thm1}).  This in turn requires a subtle new distance comparison (Lemma \ref{lemma2}) which expands the ambient space Euclidean distance between local points near the boundary of an embedded manifold with respect to semigeodesic coordinates.  This new distance comparison is the direct analog of the expansion of the ambient space distance with respect to geodesic normal coordinates \cite{smolyanov2007chernoff} which holds in the interior, but Lemma \ref{lemma2} carries out this expansion for points near the boundary using semigeodesic coordinates.

Recently, the diffusion maps estimate of the Laplacian has been used for solving PDEs such as $-\Delta u = f$ \cite{JohnPDE,JohnPDE2} where $f$ is now the data and one is solving for $u$.  We should note that in \cite{JohnPDE} a more general class of elliptic operators are considered using a more general class of kernel functions introduced in \cite{localkernels}.  In this paper we restrict our attention to the Laplace-Beltrami operator in order to focus on the boundary conditions, however the theory and methods introduced here can also be used to impose new boundary conditions on the operators considered in \cite{JohnPDE}.  We should also point out that \cite{JohnPDE} compared the diffusion maps approach to another popular meshless method based on radial-basis function (RBF) interpolation \cite{RBFmethod}.  The RBF method outperforms the diffusion maps when an appropriate global coordinate system is available in which to form the basis functions.  However, as pointed out in \cite{JohnPDE}, extending the RBF method to arbitrary manifolds would require extensive modifications, such as finding local coordinate systems and approximating the desired differential operators.  The diffusion maps approach provides a large class of operators directly with a global representation.  Thus, when more information about the manifold structure is known, approaches such as \cite{RBFmethod} may have superior results (just as mesh-based methods may have better results when a mesh is available), so it should be emphasized that our focus is on mesh-free methods on an unknown manifold as motivated above.

A key aspect of our approach is that the location of the boundary is also unknown and must be estimated as in \cite{BERRY2017}.  Since the boundary is a measure zero subset, we do not expect any data samples to lie exactly on the boundary, so rather than making a binary choice (`on' or `off' the boundary) we instead locate the boundary implicitly by estimating the distance to the boundary for each data point.  For completeness, we summarize the method of \cite{BERRY2017} for finding the distance to the boundary in \ref{background}. This distance-to-the-boundary function turns out to be the key to building the boundary integral estimator that converges to $\mathcal{B}$.  Together these results yield a collection of consistent estimators for $\mathcal{G,E},$ and $\mathcal{B}$ that can be used to solve the heat equation with various boundary conditions given only a set of points lying on the unknown manifold.

The paper is organized as follows: In Section \ref{prelims} we define the class of manifolds required for our theoretical results and establish bounds on the ratios of the intrinsic distance and the distances in coordinates near the boundary that will be required later. In Section \ref{localization} we use these bounds to show that kernel integral operators of the form \eqref{kernelintegral} with kernels having fast decay can be localized to a small neighborhood of $x$ (up to an error that is small with respect to the bandwidth, $\epsilon$). In Section \ref{expansion} we present the first asymptotic expansion of the kernel integral operator \eqref{kernelintegral} that holds uniformly in a neighborhood of the boundary (meaning $\epsilon$ is independent of $x$).  In Section \ref{boundaryint} we introduce the first boundary integral estimator (making use of the distance to the boundary function) which gives a consistent estimator of $\mathcal{B}$.  In Section \ref{estimators} we use the asymptotic expansion from Section \ref{expansion} together with the boundary integral estimator in Section \ref{boundaryint} to construct consistent estimators of the operators in \eqref{ops}.  By analyzing the existing Laplacian estimators in light of our new results we show why these standard constructions result in Neumann boundary conditions, as observed empirically going back to \cite{diffusion}.  Finally in Section \ref{PDEapps} we show how to impose standard boundary conditions using these operator estimators. 

\section{Preliminaries}\label{prelims}
The contents of this section are devoted to establishing some coordinate computations which will be used several times in subsequent sections. We begin by setting some notation and recalling some fundamental properties from Riemannian geometry. 

In what follows,  we let $\M$ be a $C^3$  compact manifold with nonempty boundary smoothly and properly embedded into $\RR^d$ via the map $\iota: \M\rightarrow \RR^d$. We endow $\M$ with the pullback metric $g=\iota^*g^{\RR^n}$ so that $\iota$ is an isometric embedding. We let $\dV$ be the Riemannian volume element defined by this metric and we let $q:\M\rightarrow \RR$ denote a $C^3$ probability density function that is absolutely continuous with respect to $\dV$.

We also recall that in any local coordinates $(s^1,...,s^m)$ on $\M$, the pullback metric evaluated on vector fields $X=x^i\partial_i ,Y=y^j\partial_j$ is given by
\[
\gdot{x^i\partial_i, y^j \partial_j} = \delta_{\alpha \beta} \pderiv{\iota^\alpha}{s^i} \pderiv{\iota^\beta}{s^j}x^iy^j
\]
where we are using Einstein notation so that indices appearing in both a superscript and subscript are implied to be summed over a common index. For convenience, we let Greek characters such as $\alpha, \beta$ range from $1$ to $d$ and Roman characters such as $i,j$ range from $1$ to $m = \dim(\M)$.

We recall that the Riemannian metric on $\M$ induces a metric space structure on $\M$ with the metric by letting $d_g(x,y)$ denote the infimum of all piecewise smooth regular curves connecting $x$ to $y$ in $\M$. If $x$ and $y$ are not in the same component, we define $d_g(x,y) = +\infty$.

Recall that the Laplace-Beltrami operator $\Delta$ on $\M$ is defined by $\Delta f  = \tn{Div}(\tn{grad } f)$ and has coordinate expression
\[
\Delta f = \frac{1}{\sqrt{|\det{g}|}} \partial_i\l(\sqrt{|\det g|}g^{ij}\partial_j f\r)
\]
in any coordinate system. In particular, when the metric is flat, $g_{ij} = \delta_{ij}$ this corresponds to the standard Euclidean Laplacian 
\[
\Delta_{\RR^m} = \sum_{i=1}^m \pderiv{^2}{x^i\partial x^i}.
\]

\subsection{Kernel Regularity}
We let $k:\RR_{\geq 0} \rightarrow \RR_{\geq 0}$ be a $C^2$ real valued function. We furthermore assume for the remainder of the paper that $k, \l|k'\r|$, and $\l|k''\r|$ have exponential decay and $k(0) = 0$. We also assume that $q:\M\rightarrow \RR$ is a strictly positive $C^3$ probability distribution function which is absolutely continuous with respect to the Riemannian volume element. These assumptions are identical to those found in \cite{hein2007convergence}.

We define a family of kernel averaging operators $\mc{I}_\veps$ indexed by parameter $\veps \in (0,\infty)$ by:
\[
\mc{I}_\veps f = \int_{\M} k\l(\frac{\norm{x-y}{\RR^n}^2}{\veps^2}\r)f(y)q(y) \dV.
\]

We remark that this choice of kernel operator differs from the one originally used in \cite{diffusion} in that we choose to use squared distance over $\veps^2.$ This is simply to avoid issues of smoothness when $x=y$. For simplicity of exposition, we will state the results of the following section assuming that $q\equiv 1$. This presents no loss of generality, since one may make the substitution $f \mapsto fq$ and obtain the needed results. 

\subsection{Normal Coordinates}
Normal coordinates are a set of coordinates which are used extensively in the asymptotic analysis of manifold learning algorithms on Riemannian manifolds. In this section, we review several of the important properties used later in the paper.

Recall that the exponential map based at a point $x$ is a mapping $\exp_x:U\subseteq T_x\M\rightarrow \M$ which maps a tangent vector $v$ to the endpoint of the geodesic based at $x$ with initial velocity $v$. On a small star-shaped neighborhood of $T_x\M$, $\exp_x$ is a diffeomorphism onto its image. The smallest value $\tn{inj}(x)>0$ such that $B_{\tn{inj}(x)}(0)\subseteq T_x\M$ is a diffeomorphism is called the \textit{injectivity radius} of $\M$ at $x$. The infimum of injectivity radii over all $x\in \M$ is called the \textit{injectivity radius of} $\M$  and it can be shown that for a compact manifold without boundary, $\textnormal{inj}(\M)$ is positive.

By identifying $T_x\M$ with $\RR^m$ using an orthonormal basis, one can use the exponential map to construct Riemannian normal coordinate charts centered at $x\in \M$, which are a system of coordinates $(s^1,...,s^m)$ with the following nice properties:
\begin{proposition}
Let $x\in \M$ and let $(s^1,..,s^m)$ denote a system of normal coordinates centered about $x$. Then
\begin{enumerate}
\item The coordinates of $x$ are $(0,..,0)$.
\item The components of the metric at $x$ are $g_{ij}(x) = \delta_{ij}.$
\item For every vector $v = v^i \partial_i$ at $x$, the radial geodesic with initial velocity $v$ is represented in coordinates by:
\[
\gamma_v(t) = t(v^1,...,v^m).
\]
\item The Christoffel symbols and first partial derivatives of $g_{ij}$ vanish at $x$.
\end{enumerate}
\end{proposition} 

In particular, the geodesic distance between $x$ and a point $y$ in normal coordinates corresponds to the 2-norm of the coordinate representative $s$ of $y$ in normal coordinates:
\[
d_g(x,y) = \norm{s}{\RR^m}.
\] 
This expression for the intrinsic distance is the foundation of the asymptotic expansion \eqref{interiorexpansion} in \cite{diffusion} as well as many related papers \cite{variablebandwidth,BERRY2017,hein2007convergence,wu2018locally,singer2012vector,gao2021diffusion}. 

In the case of manifolds with boundary, we remark that the exponential map may be less well-behaved near the boundary than in the non boundary case. Since geodesics may intersect the boundary, there is no longer a nice one-one relationship between initial velocity vectors and geodesics through a given point. 

As an illustrative example of this behavior, consider the closed two dimensional annulus $A \subseteq\RR^2$ and let $x$ be an interior point of $A$, $y$ be a boundary point for which the straight line segment $\ol{xy}$ in $\RR^2$ is not contained in $A$.

\subsection{The Normal Collar and Semigeodesic Coordinates}
We instead use \textit{semigeodesic coordinates}, which will be more amenable to calculations for points near the boundary.  We outline the needed results here, for more details on semigeodesic coordinates see \cite{LeeRM}. Since $\M$ is compact, it admits a \textit{normal collar} \cite{LeeRM}, which is a mapping $\phi: \partial \M \times [0,r_C)\rightarrow \M$ defined by:
\[
\phi(x,t) = \exp_x(-t\eta_x)
\]
where $r_C>0$, and $-\eta_x$ is the inward-facing unit normal vector field at $x$. Such a mapping is a diffeomorphism onto its image, which we will denote as $\mc{N}$. For each $t\in [0,r_C)$, we note that the set $\partial \M_t := \phi\inv(\partial \M \times \{t\})$ is the hypersurface of points distance $t$ from the boundary. Such $\partial \M_t$ are embedded submanifolds of $\M$ for each $t\in [0,r_C)$.

Inside of the normal collar, we can now construct semigeodesic coordinates centered at $x$. To do so, one first fixes a point $x$ in the normal collar, and constructs normal coordinates $(u^1,...,u^{m-1})$ in the $(m-1)$-dimensional hypersurface parallel to $\partial M$ which intersects $x$. We shall refer to such a hypersurface as $\partial \M(x)$. One then uses the $m$-th coordinate $u^m$ to parameterize the geodesic distance in $\M$ from $\partial \M(x)$. Thus semigeodesic coordinates are formed through a composition of the exponential map of $\partial \M$ and the inward-facing exponential map $\exp_x(-u^m\eta)$ of $\M$. We now list a few of their properties in contrast to the previous section.

\begin{proposition}
Let $x\in \M$ and let $(u^1,..,u^m)$ denote a system of semigeodesic coordinates centered at $x$. Then
\begin{enumerate}
\item The coordinates of $x$ are $(0,..,0)$.
\item The components of the metric at $x$ are $g_{ij}(x) = \delta_{ij}.$
\item For every vector $v = v^i \partial_i$ at $x$, the radial geodesic in $\partial \M_t$ with initial velocity $\sum_{i=1}^{m-1}v^i$ is represented in coordinates by:
\[
\gamma(t) = t(v^1,...,v^{m-1},0).
\]
The geodesic starting at $x$ which intersects each $\partial \M_t$ orthogonally is represented in coordinates as:
\[
\gamma(t) = t(0,...,0,v^m).
\]
\end{enumerate}
\end{proposition}

We remark that in contrast to the case in normal coordinates, the norm in semigeodesic coordinates no longer measures a well-defined distance. The first $m-1$ coordinates parameterize geodesic distance in the parallel hypersurface, while the last coordinate entry parameterizes geodesic distance in $\M$ in the direction orthogonal to the hypersurface.  Moreover, whereas the geodesic ball is a sphere, a semigeodesic chart may be viewed as a hypercylinder which intersects the boundary orthogonally.  The cylinder is symmetric in the coordinates $u^1,...,u^{m-1}$ with respect to any $(m-1)$-dimensional rotation but the symmetry does not extend the $u^m$ which is the `height' of the cylinder.  Moreover, since $u^m$ parametrizes the geodesic toward the boundary, the cylinder is truncated to $u^m \geq b_x = -d(x,\partial \M)$.  Finally, we note that following \cite{LeeRM}, $u^m$ is oriented along an inward-facing pointing normal, so for any vector $v$ we have $v^m = -v\cdot \eta_x$ since $\eta_x$ is outward-pointing.




 
\subsection{Coordinate Estimates}
Although the Euclidean norm in semigeodesic coordinates no longer corresponds to geodesic distance from $x$, the goal of this section is to relate the norm in such coordinates to the geodesic distance in $\M$.

We first note that compactness of $\M$ implies an upper and lower bound on the sectional curvature of the interior of $\M$ as well as bounds on the inward-facing sectional curvature of $\partial \M$. From the Gauss equation, it also follows that each parallel hypersurface $\partial \M_t$ for $t\in [0,\frac{1}{2}r_C]$ has inward-facing sectional curvature bounded above and below. We let $K$ be an upper bound on the sectional curvature and inward facing sectional curvatures of such parallel hypersurfaces such that $-K$ is also a lower bound on such curvatures.  By the results of \cite{ehrlich1974continuity}, one also obtains a single lower bound $r_\partial>0$ on the injectivity radius of each parallel hypersurface $\partial \M_t$ for all $t\in [0,\frac{1}{2}r_C]$.

The following lemma may be proven by an application of the Rauch comparison theorem (for normal coordinates) and Warner's generalization of the Rauch comparison theorem \cite{warner1966extension} (for semigeodesic coordinates. Such an argument was done in Proposition 2.6 of \cite{schick2001manifolds} and needs only slight modification for semigeodesic coordinates.

\begin{lemma}\label{metrictensorbounds}\cite{schick2001manifolds}
Suppose $\M$ is a compact Riemannian manifold, then there exists constants $\ti{C}_0>0, \ti{C}_1>0$ such that in any geodesic chart in $\M$,
\[
|g_{ij}|\leq \ti{C}_0 \tn{ and } |g^{ij}|\leq \ti{C}_1
\]
\end{lemma}

Using Lemma \ref{metrictensorbounds}, we can now show that the norm in geodesic coordinates approximates the extrinsic distance induced by $\iota: \M\rightarrow \RR^d$.
\begin{proposition}\label{prop1}
There exists a $C_0,C_1>0$ such that in any geodesic chart centered at $x$, and any point $y$ in that chart,
\[
C_0 \vert u \vert^2 \leq d^2_{\RR^d}(\iota(x),\iota(y)) \leq C_1 \vert u \vert^2
\]
where $u$ is the coordinate representative of $y$ in either normal or semigeodesic coordinates.
\end{proposition}
\begin{proof}
From Lemma \ref{prop1}, there exists positive constants $\ti{C}_0$ and $\ti{C}_1$ such that
\[
|g_{ij}|\leq \ti{C}_0 \tn{ and } |g^{ij}| \leq \ti{C}_1
\]
for any geodesic coordinate chart. Since the matrices with entrees $g_{ij}$ and $g^{ij}$ are symmetric and positive definite, this implies that there exists  positive bounds $C_1$ and $C_0$ on the largest eigenvalue of $(g_{ij})$ and $(g^{ij})$ in any geodesic coordinate chart.

We now let $x$ be a point in $\M$, and choose either normal or semigeodesic coordinate charts for $\M$ centered at $x$, depending on whether $x$ is in the normal collar. We then choose a normal coordinate chart for $\iota(x)$ in $\RR^d$, which is simply centering $\iota(x)$ at zero. In these coordinates, we have that
\[
d_{\RR^d}(\iota(x),\iota(y)) = g^{\RR^d}_{\alpha\beta}(0)\iota^\alpha(u)\iota^\beta(u).
\]
where $u$ is the coordinate representative of $y$ in these coordinates. We then perform a Taylor expansion of $\iota(u)$, recalling that in these coordinates $\iota(0)= 0$. Therefore there exists a point $\ti{u}$ in the domain such that:
\begin{align*}
g^{\RR^d}_{\alpha\beta}(0)\iota^\alpha(u)\iota^\beta(u) &= g^{\RR^d}_{\alpha\beta}(0)\l(\iota^\alpha(0) +\pderiv{\iota^\alpha}{u^i}\Big|_{\ti{s}}u^i\r) \l(\iota^\beta(0) +\pderiv{\iota^\beta}{u^j}\Big|_{\ti{u}}u^j\r)\\
&=g^{\RR^d}_{\alpha\beta}(0)\pderiv{\iota^\alpha}{u^i}\Big|_{\ti{u}}\pderiv{\iota^\beta}{u^j}\Big|_{\ti{u}}u^iu^j.
\end{align*}
Since $g^{\RR^d}_{\alpha\beta}(0)= g^{\RR^d}_{\alpha\beta}(\iota(\ti{u})) = \delta_{\alpha\beta}$, we have:
\begin{align*}
g^{\RR^d}_{\alpha\beta}(0)\iota^\alpha(u)\iota^\beta(u) &= g^{\RR^d}_{\alpha\beta}(\iota(\ti{u}))\pderiv{\iota^\alpha}{u^i}\Big|_{\ti{u}}\pderiv{\iota^\beta}{u^j}\Big|_{\ti{u}}u^iu^j\\
&= g^\M_{ij}(\ti{u}) u^i u^j.
\end{align*}
We now note that the expression $g^\M_{ij}(\ti{u}) u^i u^j$ is maximized by the maximum eigenvalue of the matrix $g^\M_{ij}(\ti{u})$ and minimized by the maximum eigenvalue of $(g^\M)^{ij}(\ti{u})$. Since we have previously show that these are bounded by $C_1$ and $C_0$ regardless of choice of geodesic chart, we have that 
\begin{align*}
C_0 |u|^2\leq g^{\RR^d}_{\alpha\beta}(0) \iota^\alpha(s)\iota^\beta(s)\leq C_1 |u|^2
\end{align*}
and therefore 
\[
C_0 \vert u \vert^2 \leq d^2_{\RR^d}(\iota(x),\iota(y)) \leq C_1 \vert u \vert^2.
\]
\end{proof}
Thus, we have established a relationship between the norm in semigeodesic coordinates an the extrinsic distance defined by the embedding $\iota: \M\rightarrow \RR^d$.

\section{Localization of asymptotic expansions for manifolds with boundary}\label{localization}

We begin by rigorously defining two regions of $\M$ which are ``close'' and ``far" from the boundary respectively. For each $\veps>0$, we let $\M_\veps$ be the set of all points $x\in M$ such that $d(x,\partial \M)>\veps$. We then define $\mc{N}_\veps$ to be the points such that $d(x,\partial M)\leq\veps$. We refer to $\M_\veps$ as the \textit{interior region} and $\mc{N}_\veps$ as the \textit{closed collar region}. Due to the generalized Gauss lemma \cite{LeeRM}, one has that if $\veps$ is less than the normal collar width $r_C$, the topological boundary $\partial \M_\veps$ of $\M_\veps$ is a hypersurface in $\M$ parallel to $\partial \M$. 

We first show that for each sufficiently small $\veps,$ the manifold $\M$ admits an atlas of charts which are each ``large enough'' to contain a metric ball of radius $\veps$. This will be proven in Proposition \ref{prop:atlaslocal}, which we list here.

{
\renewcommand{\thetheorem}{\ref{prop:atlaslocal}}
\begin{proposition}
There exists a $C_\M>0$ such that for all $0<\veps < C_\M$, the preimage of an extrinsic ball $\iota\inv(B_\veps^{\RR^d}(\iota(x)))$ centered at a point $x\in \M$ is contained in a normal coordinate chart if $x\in \M_\veps$ or in a semigeodesic coordinate chart if $x\in \N_\veps$.
\end{proposition}
\addtocounter{theorem}{-1}
}

 We let $\partial \M_t$ refer to the hypersurface of points distance $t$ from the boundary. Such hypersurfaces are parallel to $\partial \M$ in the sense that geodesics with initial velocity normal to $\partial \M$ intersect the surfaces $\partial \M_t$ orthogonally. Using this fact, each tangent space $T_x\mc{N}$ for $x\in \mc{N}_\veps$ admits a decomposition:
\[
T_x \mc{N}_\veps = T^{\top}_x \mc{N}_\veps \oplus T^{\perp}_x \mc{N}_\veps.
\]
where $T^{\top}_x \mc{N} = T_x \partial \M_{d(x,\partial \M)}$ is the tangent space of the parallel hypersurface intersecting $x$ and $T^{\perp}_x \mc{N}_\veps$ is the space spanned by the unit vector normal to the hypersurface in $\M$.

For each $x\in \M$ and each $\veps>0$, we define the \textit{semigeodesic hypercylinder} $\ti{B}_\veps (x)$ of radius and height $\veps$ as the set of all vectors $v\in T_x\mc{N}$ such that:
\begin{enumerate}
    \item $\norm{v^\top}{g} < \veps$
    \item $\norm{v^\perp}{g} < \veps$ if $v^\perp$ is inward-facing
    \item $\norm{v^\perp}{g} < d_g(x,\partial \M)$ if $v^\perp$ is outward-facing.
\end{enumerate}

By this construction, we have that if $\veps$ is less than the injectivity radius of the parallel hypersurface $\partial \M(x) = \partial \M_{d_g(x,\partial \M})$  through $x$, then $exp^{\partial \M(x)}_x(v^\top)$ is well-defined for all $v\in \ti{B}_\veps(x)$. Similarly, if $\veps < \frac{r_C}{2}$ then $\exp_q(T^{\partial \M(x)}_{pq} v^\perp)$ is well-defined for any $q$ in a normal neighborhood of $x$ in $\partial \M(x)$, where here $T^{\partial \M(x)}_{pq}v^\perp$ denotes the parallel translate of $v^\perp$ in $\partial \M(x)$ to the point $q\in \partial \M(x)$ and $r_C$ is the normal collar width. Thus we see that if $\veps$ is sufficiently small, the semigeodesic cylinder of radius and height $\veps$ may be identified with hypercylinder which is a submanifold of $\M$ through the exponential map of $\partial \M(x)$.

The next proposition shows that there is a constant $r_\tn{sem}>0$ such that if one chooses a radius smaller than some constant $C_\M > 0$, then a metric ball of radius $C_\M$ in $\M$ is small enough to fit inside any semigeodesic cylinder of radius and height $r_\tn{sem}.$ This is an essential step to show that one can uniformly localize the operator $\mc{I}_\veps$ to semigeodesic charts in the same way as normal coordinate charts.

\begin{proposition}\label{prop:semilocalization}
Let $r_C>0$ denote the normal collar width and let $r_\partial>0$ be a lower bound on the injectivity radii of all parallel hypersurfaces $\partial \M_\veps$ with $\veps< \frac{1}{2}r_C$. Let $K\in \RR$ be an upper bound on the sectional curvature of $\M$ and the inward-facing sectional curvature of $\partial \M$. Let $r_\sem = \min\{\frac{1}{2}r_C,r_\partial, \frac{\pi}{2\sqrt{K}}\}$ where $\sqrt{K}\inv$ is defined to be infinite if $K\leq 0$ Then there exists a $\ti{C}_\M>0$ such that for all $x\in \mc{N}_{r_\sem}$,
\[
B^\M_{\ti{C}_\M}(x)\subseteq \ti{B}_{r_\sem}(x).
\]
\end{proposition}
\begin{proof}
We first let $X$ denote
\[
X =\coprod_{x\in \mc{N}_{r_\sem}} \partial \ti{B}_{r_{\sem}}(x).
\]
In other words, consider the disjoint union of the boundary of all semigeodesic hypercylinders centered about all points in $x\in \mc{N}_{r_\sem}$. It can be shown that such a set is a fiber bundle over $\mc{N}_{r_\sem}$ with model fiber diffeomorphic to a hypercylinder with the interior of the ''bottom'' face removed. Such a model fiber is compact and since $\mc{N}_{r_\sem}$ is also compact, it follows that the fiber bundle $X$ is compact. Note that proof of existence of this fiber bundle follows in analogy to the construction of the unit tangent bundle on a Riemannian manifold without boundary.

We now consider the function $d_\partial:X\rightarrow \RR$ which assigns to each $(x,v)$ the distance from $x$ to the geometric realization of $v$ in $\M$. Such a map is clearly continuous on $x$, and thus obtains a minimum value $C_\M$ on $X$ by compactness. It can be easily argued using properties of the exponential map that for each $x\in \mc{N}_{r_\sem},$ we have that $x\notin \partial B_{r_\sem}(x)$ and thus $d(x,v)> 0$ for all $(x,v)\in X$. Hence $\ti{C}_\M>0.$

Since we assume $r_{\sem}<\frac{\pi}{4\sqrt{K}}$, we  have that $\norm{v}{g}<\frac{\pi}{4\sqrt{K}}$ for each $(x,v) \in X$. It follows that the distance from $x$ to the geometric realization of $v$ is less than $\frac{\pi}{2\sqrt{K}}$. From Corollary 2 of \cite{alexander1993geometric}, it follows that every pair of points $(x,y)$ in a metric ball of radius $\frac{\pi}{2\sqrt{K}}$ has a length-minimizing curve $\gamma$ connecting $x$ and $y$ in $\M$.

Now, fix $x\in \mc{N}_{r_\sem}$ and suppose that $y\in B_{\ti{C}_\M}(x)$ but $y\notin \ti{B}_{r_\sem}(x)$. Denote the length-minimizing curve connecting $x$ and $y$ by $\gamma_{xy}(t)$. A simple topological argument can be used to show that $\gamma_{xy}$ eventually intersects the boundary of $\ti{B}_{r_\sem}(p)$. Namely, there exists a $t_0$ such that $\gamma_{xy}(t_0)\in \partial \ti{B}_{r_\sem}(p)$. It follows that $d_g(p,\gamma_{xy}(t_0))\geq \ti{C}_\M$ and hence the length of $\gamma_{xy}>C_\M$. This contradicts the fact that $y \in B_{\ti{C}_\M}(x)$ and implies the result.
\end{proof}

Since in general the kernel function $k_\veps(x,y)$ is defined using extrinsic distance in $\RR^d$ instead of distance in $\M$, we need to improve the result in Proposition \ref{prop:semilocalization} to account for extrinsic distance. We also wish to improve the result to hold for both semigeodesic and normal coordinates and hold for all $\veps$ sufficiently small. Most of this can be done by simply observing that the embedding map $\iota$ has a uniformly continuous inverse. The main remaining obstacle is that the region $\M_\veps$ grows as $\veps$ approaches zero. We therefore need to show that the injectivity radius of the region $\M_\veps$ does not shrink too fast as $\veps$ approaches zero.

\begin{proposition}\label{prop:atlaslocal}
There exists a $C_\M>0$ such that for all $0<\veps < C_\M$, the preimage of an extrinsic ball $\iota\inv(B_\veps^{\RR^d}(\iota(x)))$ centered at a point $x\in \M$ is contained in a normal coordinate chart if $x\in \M_\veps$ or in a semigeodesic coordinate chart if $x\in \mc{N}_\veps$.
\end{proposition}
\begin{proof}
First we consider the value 
\[
\tn{inj}(\M_\veps):=\inf_{x\in \M_\veps}\tn{inj}(x).
\]

We first show that for sufficiently small $\veps$, the value of $\tn{inj}(\M_\veps)\geq\veps$. Consider the double $D(\M)$ of $\M$ formed by gluing identical copies of $\M$ along the boundary of $\M$. It follows that $D(\M)$ is compact and one may extend the metric on $\M$ arbitrarily to a metric on $D(\M)$. In such a case, geodesic balls of radius $\veps$ or less on $\M_\veps$ coincide with geodesic balls of radius $\veps$ or less on $D(\M)$. Since $D(\M)$ is compact, it has positive injectivity radius and thus if $\veps < \tn{inj}(D(\M)),$ then $exp_x$ is bijective on $B_\veps(x)$ in $\M_\veps.$ Hence, if $\veps < \tn{inj}(D(\M)),$ $\M_\veps$ may be covered by normal coordinate charts which contain a metric ball of radius $\veps.$ By Proposition \ref{prop:atlaslocal} if $\veps<\ti{C}_\M$, then $\mc{N}_\veps$ may be covered in semigeodesic charts each of which contain a metric ball in $\M$ of radius $\veps.$

We now let $C_\M' = \min\{\ti{C}_\M, \tn{inj}(D(\M))\}$. Since the embedding $\iota:\M\rightarrow \RR^d$ is continuous on a compact set, its inverse $\iota\inv$ is uniformly continuous on its domain. Thus, there exists a $C_\M>0$ which does not depend on $x\in \M$ for which $\iota\inv(B_{C_\M}^{\RR^d}(\iota(x))\subseteq B_{C_\M'}(x)$ for all $x\in \M$. Therefore for all $0<\veps < C_\M$ and all $x\in \M$, we have $\iota\inv(B_\veps^{\RR^d}(\iota(x))$ is contained in a normal coordinate chart if $x\in \M_\veps$ and a semigeodesic chart if $x\in \mc{N}_\veps$.
\end{proof}

We now put together Proposition \ref{prop:atlaslocal} and Proposition \ref{prop1} to show that for sufficiently small $\veps$, one may localize the kernel integral operator to a geodesic coordinate chart up to order $\veps^z$ for arbitrarily large $z\in \NN$.

\begin{lemma}[Localization to a Geodesic Neighborhood]\label{lemma1}
Let $0< \gamma < 1$. For any $\veps>0$ such that $\veps^\gamma <\min\{ \frac{r_\M}{C_1}, C_\M\}$, 
\[
\l|\int_{\M\setminus \ti{B}^\M_{\veps^\gamma}(x)} k\l(\frac{d^2_{\RR^d}(\iota(x),\iota(y))}{\veps^2}\r) f(y)q(y)\dV\r| \in \mc{O}(\veps^z)
\]
where $z$ may be chosen arbitrarily large in $\mbb{N}$.
\end{lemma}
\begin{proof}
If $\veps^\gamma < C_\M$, we have that the preimage of an $\veps^\gamma$ ball in $\RR^d$ centered about $\iota(x)$ is contained in a geodesic coordinate chart centered at $x$. If $\veps^\gamma  < \frac{C_\M'}{C_1}$,  then by Proposition \ref{prop1}, $\ti{B}^\M_{C_1\veps^\gamma}(x)$ contains this preimage, and is also contained in the geodesic chart. Hence, any point in $\M$ outside of $\ti{B}^\M_{C_1\veps^\gamma}(x)$ has extrinsic distance no less than $\veps^\gamma$ from $x$.

Using exponential decay of the kernel, this implies that
\begin{align*}
\int_{\ti{B}^\M_{C_1\veps^\gamma}(x)}k\l(\frac{d_{\RR^d}(\iota(x),\iota(y))}{\veps^2}\r)q\dV &\leq \int_{\ti{B}^\M_{C_1\veps^\gamma}(x)}\alpha e^{-\beta\frac{d^2_{\RR^d}(\iota(x),\iota(y))}{\veps^2}}q \dV\\
&\leq \alpha e^{-\beta\frac{\veps^{2\gamma}}{\veps^2}}\\
&=\alpha e^{-\beta e^{2(\gamma-1)}}
\end{align*}
 We then apply Cauchy-Schwarz inequality in $q$-weighted $L^2(\M)$:
\begin{align*}
\l\langle k\l(\frac{d_{\RR^d}(\iota(x),\iota(y))}{\veps^2}\r) , f\r\rangle^2  &\leq \l\langle k\l(\frac{d_{\RR^d}(\iota(x),\iota(y))}{\veps^2}\r) , k\l(\frac{d_{\RR^d}(\iota(x),\iota(y))}{\veps^2}\r)\r\rangle^2\langle f, f\rangle\\
&\leq\langle f, f\rangle \alpha e^{-2\beta \veps^{2(\gamma-1)}}\int_{\M\setminus \ti{B}^\M_{C_1\veps^\gamma}(x)} q \dV\\
&\leq\langle f, f\rangle \alpha e^{-2\beta \veps^{2(\gamma-1)}}
\end{align*}

We see that the term $\langle f, f\rangle \alpha e^{-2\beta \veps^{2(\gamma-1)}}$ is asymptotically bounded by any polynomial $\veps^z$ with $z\geq 1$.
\end{proof}

We have now shown that the value of $\mc{I}_\veps f(x)$ depends only on the behavior of $f$ inside a single chart for small enough values of the parameter $\veps.$ If $x$ is in the closed collar region $\mc{N}_\veps$ for small enough $\veps$, this chart must be taken as a semigeodesic coordinate chart, while if $x$ contained in the interior region $\M_\veps$ one may use a normal coordinate chart. 

\section{Uniform Asymptotic Expansion for Manifolds with Boundary}\label{expansion}
The results of the previous section show that the asymptotic analysis of $\mc{I}_\veps f(x)$ can be subdivided into two cases depending on whether $x$ is in the interior region $\M_\veps$ or closed collar region $\mc{N}_\veps$. In this section, we derive new asymptotic expansions of $\mc{I}_\veps f$ in semigeodesic coordinates. When taken together with existing expansions in normal coordinates from \cite{diffusion, hein2007convergence}, this yields an asymptotic expansion of $\mc{I}_\veps$ that is uniform in $\veps$, meaning that for sufficiently small $\veps,$ the expansion holds for each $x\in M$. This uniformity is necessary for our later proof of convergence.

We begin by deriving asymptotic expansions of $\mc{I}_\veps f$ in semigeodesic coordinates. The following lemmas are used to show that the value of the constant in the leading order error term is related to the mean curvature of the boundary of $\M$. They are largely technical, but this specific value will give us some cancellation in the final expansion and is thus important.

We begin by observing coordinate expressions for the Levi-Civita connection on $\M$.

\begin{lemma}\label{lemma:computation1}
Let $U$ be the vector field such that
\begin{enumerate}
\item $U_x \in T_x\M$ maps to the point $u$ in semigeodesic coordinates centered at $x$.
\item The coordinate representation $U= u^i\partial_i$ has constant component functions $u^i$.
\end{enumerate}
Then at the point p:
\[
2\gdot{\nabla_U U ,U} = \delta_{\alpha\beta} \l(\pderiv{^2\iota^\alpha}{u^a\partial u^c}\pderiv{\iota^\beta}{u^b} + \pderiv{\iota^\beta}{u^a}\pderiv{^2\iota^\alpha}{u^b\partial u^c}\r).
\]
\end{lemma}
\begin{proof}
Since $\iota: \M \rightarrow \RR^d$ is an isometric embedding, we may relate the components of the metric in $\M$  to those  in $\RR^d$ via:
\[
g^\M_{ab}(u) = g^{\RR^d}_{\alpha\beta}(\ti{s}) \pderiv{\iota^\alpha}{u^a}\pderiv{\iota^\beta}{u^b}= \delta_{\alpha\beta}\pderiv{\iota^\alpha}{u^a}\pderiv{\iota^\beta}{u^b}
\]
We then take the partial derivative of both sides, noting that we are in normal coordinates in $\RR^d$ and therefore all partial derivatives of the metric components vanish at $0$. This yields:
\[
\pderiv{g^\M_{ab}(0)}{u^c} =\pderiv{g^{\RR^d}_{\alpha\beta}(0)}{u^c} \pderiv{\iota^\alpha}{u^\rho}\pderiv{\iota^\rho}{u^c}\pderiv{\iota^\alpha}{u^a}\pderiv{\iota^\beta}{u^b}  + g^{\RR^d}_{\alpha\beta}(0) \pderiv{ }{u^c}\l(\pderiv{\iota^\alpha}{u^a}\pderiv{\iota^\alpha}{u^b}\r) =\delta_{\alpha\beta} \l(\pderiv{^2\iota^\alpha}{u^a\partial u^c}\pderiv{\iota^\beta}{u^b} + \pderiv{\iota^\beta}{u^a}\pderiv{^2\iota^\alpha}{u^b\partial u^c}\r).
\]
Given any $u \in T_x\M$, we can extend $u$ to the vector field $U= u^i\partial_i$ on the coordinate chart where $u^i$ are constant functions and $\partial_i$ are the coordinate vector fields. Using that the Levi-Civita connection is compatible with the metric, we obtain:
\[
\pderiv{g^\M_{ab}(0)}{u^c}u^au^bu^c = U\gdot{U,U} = 2\gdot{\nabla_U U, U}
\]
as desired.
\end{proof}
Next we relate the Levi-Civita connection to the second fundamental form $\Pi_{\partial \M_t}$ of the hypersurfaces $\partial \M_t$ as a submanifolds of $\M$.
\begin{lemma}\label{lemma:computation2}
With $u$ and $U$ having the same conditions as above, decompose $U$ into the vector field $U^\top = \sum_{i=1}^{m-1} u^i \partial_i$ tangential to the hypersurface $\partial \M_t$ and the normal vector field $U^\perp = u^m\partial_m = -u^m\eta_x$. Then we have:
\[
\gdot{\nabla_U U,U} = -\gdot{\Pi_{\partial \M_t}(U^\top, U^\top),U^\perp}.
\]
\end{lemma}
\begin{proof}
We decompose $\gdot{\nabla_U,U}$ into
\[
\gdot{\nabla_UU,U} = \gdot{\nabla_{(U^\top +U^\perp)}(U^\top +U^\perp),(U^\top +U^\perp)}.
\]
Since the connection is linear in both components over $\RR$, and the component functions of $U$ are constant, we may simply bilinearly expand the above term. We also note that 
\[\nabla_{u^j\partial_j} u^i\partial_i = u^iu^j \nabla_{\partial_j}\partial_i= u^iu^j\Gamma_{ij}^k\partial_k.
\] 
Since many of the Christoffel symbols in semigeodesic coordinates are zero, we are left with:
\begin{align*}
\gdot{\nabla_UU,U} = \gdot{\nabla_{U^\top} U^\top, U^\top} +\gdot{\nabla_{U^\top} U^\top, U^\perp} + \gdot{\nabla_{U^\top} U^\perp, U^\top} + \gdot{\nabla_{U^\perp} U^\top, U^\top}
\end{align*}

Using the Gauss equation for the hypersurface embedded in $\M$, we get that $\nabla_{U^\top} U^\top = \Pi(U^\top,U^\top)$. This implies that the first term is zero and the second term is $\gdot{\Pi_{\partial \M_t}(U^\top, U^\top),U^\perp}$.

For the next two terms, we first note that since $\nabla$ is a symmetric connection,
\[
\nabla_{u^j\partial_j} u^i\partial_i = u^iu^j\Gamma_{ij}^k\partial_k=u^iu^j\Gamma_{ji}^k\partial_k=\nabla_{u^i\partial_i} u^j\partial_j
\]
and so $\nabla$ is a symmetric tensor over $\RR$. Thus, both of the remaining terms are equal. The Weingarten equation implies that:
\[
\gdot{\nabla_{U^\top} U^\perp, U^\top}  =-\gdot{ \Pi(U^\top,U^\top),U^\perp}.
\]
Putting this all together, we are left with:
\[
\gdot{\nabla_UU,U} = \gdot{\Pi_{\partial \M_t}(U^\top, U^\top),U^\perp} -2\gdot{ \Pi(U^\top,U^\top),U^\perp}= -\gdot{ \Pi(U^\top,U^\top),U^\perp}
\]
\end{proof}
Putting the above three lemmas together, we now asymptotically compare the extrinsic distance in $\RR^d$ of two nearby points $\iota(x),\iota(y)$ to the norm of the coordinate expression of $y$ in semigeodesic coordinates centered at $x$. This result is analogous to Proposition 6 of \cite{smolyanov2007chernoff} which makes a similar comparison, except in their case the comparison was between extrinisic and intrinsic distance instead of semigeodesic norm.

\begin{lemma}\label{lemma2}
Let $x\in \M$ and let $y\in \M$ be such that $\norm{\iota(x)-\iota(y)}{\RR^d}< C_\M$. Let $u$ denote the coordinate representative of $y$ in semigeodesic coordinates and let $\norm{u}{}$ denote the norm of $u$ in semigeodesic coordinates. Then
\[\lim_{|u|^3\rightarrow 0}\frac{\Vert \iota(x)-\iota(y) \Vert^2_{\RR^d} - \Vert u\Vert^2}{\Vert u\Vert^3} =  -\gdot{\Pi_{\partial \M_t}(U^\top, U^\top),U^\perp}\]
\end{lemma}
\begin{proof} 
Since $\norm{\iota(x)-\iota(y)}{\RR^d}< C_\M$, it follows from \ref{prop:atlaslocal} that $y$ has a coordinate representative $u$ in semigeodesic coordinates centered at $x$. We thus choose such a coordinate system and apply Taylor's theorem and apply Lemmas \ref{lemma:computation1} and \ref{lemma:computation2}.
\begin{align*}
\Vert \iota(u) \Vert^2_{\RR^d} &= g_{\alpha\beta}(0) \pderiv{\iota^\alpha}{u^a}\pderiv{\iota^\beta}{u^b}u^au^b + \frac{1}{2}g_{\alpha\beta}(0)\pderiv{^2\iota^\alpha}{u^au^c}\pderiv{\iota^\beta}{u^b} u^au^bu^c + \frac{1}{2} g_{\alpha\beta}(0)\pderiv{\iota^\beta}{u^b}\pderiv{^2\iota^\alpha}{u^bu^c} u^au^bu^c + \mc{O}(|u|^4)\\
&=\Vert u \Vert^2_{\M} +\frac{1}{2} \pderiv{g^\M_{ab}(0)}{u^c}u^au^bu^c + \mc{O}(\Vert u\Vert^4)\\
&= \Vert u \Vert^2_{\M} +\gdot{\nabla_U U,U} + \mc{O}(\Vert u\Vert^4)\\
&=\Vert u \Vert^2_{\M}-\gdot{\Pi_{\partial \M_t}(U^\top, U^\top),U^\perp}+\mc{O}(\Vert u\Vert^4)
\end{align*}
where the final equations follow from the previous two lemmas.
\end{proof}

We note that the analogous expansion done in Proposition 6 of \cite{smolyanov2007chernoff} has order $d_\M(x,y)^4$ instead of $\norm{u}{}^3$ and depends on the second fundamental form of the embedding of $\mathcal{M}$ into the ambient space.  Thus, one of the tradeoffs for using semigeodesic coordinate charts is a lower order error term.

Next, we expand the volume form $\dV$ in semigeodesic coordinates.  In contrast, we first note that in normal coordinates the volume form has the expansion
\[
\dV = 1 + R_{ij} s^is^j +\mc{O}(|s|^3)
\]
where $R_{ij}$ are the components of the Ricci curvature tensor \cite{LeeRM}. 
In semigeodesic coordinates, an analogous result follows from the first variation formula for the area of hypersurfaces.
\begin{theorem}[First Variation of Area, \cite{chow2006hamilton,gray2012tubes}]\label{volume}
Let $\partial \M_t$ be a hypersurface in $\M$ with $x\in \partial \M$ and outward facing normal $\eta_x$. Let $\sigma(t):(-\epsilon , \epsilon)\rightarrow \M$ be a geodesic with initial velocity $\dot{\sigma}(0) = -\eta_x$  Then for $y=\sigma(t)$ we have
\[ \dV(y) = 1 - (m-1)H(x)t +\mc{O}(t^2) \]
where $H$ is the mean curvature of $\partial \M_t$.  
\end{theorem}
We note that \cite{chow2006hamilton} defines the mean curvature as simply the summation, whereas we follow \cite{LeeRM} in including the factor $\frac{1}{m-1}$.  For points $y$ that are not along the geodesic $\sigma$ the first order term will be the same since semigeodesic coordinates are the same as normal coordinates on the submanifold $\partial \mathcal{M}^t$ which contains no first order term.  Thus for $y$ not along the geodesic we have,
\[ \dV(y) = 1 - (m-1)H(x)u^m + \omega_3(x)_{ij}u^i u^j + \mathcal{O}(|u|^3) \]
for some smooth tensor $\omega_3$.  The next result will be required to simplify some expressions in the theorem which involve the derivative of the kernel.

\begin{lemma}\label{lemma3} Integrating over a cylinder $B = \{u \, | \, \sum_{i=1}^{m-1}(u^i)^2<\epsilon^2, u^m \in [-b_x/\epsilon,\epsilon] \}$ which is symmetric in coordinates $u^i$ for $1\leq i\leq m-1$ we have
\begin{align}\label{m1coeff}
\int_B k'(|u|^2)\gdot{\Pi_{\partial \M_t}(U^\top, U^\top),U^\perp} du &= -\frac{(m-1)}{2}H(x)\int_B k(|u|^2)u^m \, du \nonumber \\ &= \frac{(m-1)}{2}m_1^\partial(x)H(x) + \mathcal{O}(\epsilon^z)
\end{align}
for any $z\geq 1$, where $H(x)$ is the mean curvature.
\end{lemma}
\begin{proof}
Linear expansion of $\gdot{\Pi_{\partial \M_t}(U^\top, U^\top),U^\perp}$ in terms of the coordinate basis at $x$ yields:
\[
\gdot{\Pi_{\partial \M_t}(U^\top, U^\top),U^\perp} = \gdot{\Pi_{\partial \M}(\partial_i,\partial_j),\partial_m}u^iu^ju^m.
\]
since the domain $B$ is symmetric in the coordinates $u^i$ for $1 \leq i \leq m-1$, all of the terms $u^i u^j$ with $i \neq j$ will integrate to zero.  Thus, we have 
\[ \int_B k'(|u|^2)\gdot{\Pi_{\partial \M_t}(U^\top, U^\top),U^\perp} du = \gdot{\Pi_{\partial \M}(\partial_i,\partial_i),\partial_m}\int_B k'(|u|^2)u^i u^i u^m du \]
and by the symmetry of the kernel, the integrals are equal for all $1 \leq i \leq m-1$, so we only need to compute 
\[ \int_B k'(|u|^2)u^1 u^1 u^m ds = \int_B \frac{1}{2}\left(\frac{\partial}{\partial u^1}k(|u|^2)\right) u^1 u^m du^1 du^2 \cdots du^m = -\frac{1}{2}\int_B k(|u|^2)u^m du \]
where the last equality follows from integration by parts with respect to $u^1$. Finally, pulling the integral out of the sum, we have,
\[ \int_B k'(|u|^2)\gdot{\Pi_{\partial \M_t}(U^\top, U^\top),U^\perp} du = -\frac{1}{2}\int_B k(|u|^2)u^m du \sum_{i=1}^{m-1} \gdot{\Pi_{\partial \M}(\partial_i,\partial_i),\partial_m} \] 
and since the mean curvature is defined as $H(x) = \frac{1}{m-1}\sum_{i=1}^{m-1} \gdot{\Pi_{\partial \M}(\partial_i,\partial_i),\partial_m}$ the first equality in \eqref{m1coeff} follows.  Finally, substituting $u^m = -u\cdot \eta_x$ and extending the integral to all of $\{u \, | \, u^m > -b_x\} = \{u \, | \, u\cdot \eta_x < b_x\}$ by Lemma \ref{lemma1} we obtain the second equality of \eqref{m1coeff}.
\end{proof}


We now compute the asymptotic expansion for points inside of $\mc{N}_\veps$. We introduce the following definition from \cite{BERRY2017} for the moments of a kernel function near the boundary,
\begin{align}\label{moments} 
m^{\partial}_\ell(x) = \int_{\{z  \in \mathbb{R}^m \, | \, z\cdot \eta_x < b_x/\epsilon\}} (z\cdot \eta_x)^\ell k(|z|^2)\, dz =  \int_{\mathbb{R}^{m-1}}\int_{-\infty}^{\dM_x/\epsilon} z_m^\ell k\left(|z|^2 \right) \, dz_m dz_1 \cdots dz_{m-1} \;  
\end{align}
where $\eta_x$ is a smooth extension of the boundary normal vector field into the normal collar of the boundary.  For more information on these moments, see \ref{background}. We can now state and prove the following theorem which includes a uniform asymptotic expansion for points near the boundary.
\begin{theorem}\label{thm1}(expansion near the boundary) Let $\mathcal{M} \subset \mathbb{R}^n$ be a compact $m$-dimensional $C^3$ Riemannian manifold with a 
$C^3$ boundary.  Let $k:\mathbb{R}\to\mathbb{R}$ have exponential decay. Suppose that for fixed $\gamma \in (0,1),$ that $\veps^\gamma <\min\{\frac{r_\M}{C_1}, C_\M\}$ Then for all $x\in \mc{N}_\veps:= \{x\in \M : d(x,\partial \M)\leq \veps\}$, we have
\begin{align}\label{fullexpansion} \epsilon^{-m}\int_{y\in\mathcal{M}} k\left(\frac{|x-y|^2}{\epsilon^2}\right)f(y)\dV &= m_0^{\partial}(x) f(x) + \epsilon m_1^{\partial}(x) \left(\eta_x \cdot \nabla f(x) + \frac{m-1}{2}H(x)f(x) \right) + \mathcal{O}(\epsilon^2)  \end{align}
where the moments $m_\ell^\partial(x)$ are defined in \eqref{moments} and $H(x)$ is the mean curvature of the  hypersurface parallel to $\partial \M$ intersecting $x$ (which depends on the second fundamental form of $\partial\mathcal{M}\subset\mathcal{M}$). 
\end{theorem} 
\begin{proof} First, by Lemma \ref{lemma1} we localize the integral to a semigeodesic $\epsilon$-ball, $B$ making an error of higher order than $\mathcal{O}(\epsilon^2)$.  Note that $B$ is exactly the domain of the integral defining the coefficients $m_{\ell}^\partial(x)$ in \eqref{moments}.  We then multiply three expansion.  First, the kernel expansion,
\begin{align*} k\left(\frac{||x-y||_{\RR^d}^2}{\epsilon^2}\right) &= k\left(\frac{||u||_\M^2 - \gdot{\Pi_{\partial \M_t}(U^\top, U^\top),U^\perp} + \omega_1(x,u) + \mathcal{O}(|u|^5)}{\epsilon^2}\right) \\ 
&= k\left(\frac{||u||_\M^2}{\epsilon^2}\right) - k'\left(\frac{||u||_\M^2}{\epsilon^2}\right)\frac{1}{\epsilon^2}\gdot{\Pi_{\partial \M_t}(U^\top, U^\top),U^\perp} + \mathcal{O}(\epsilon^{-2}|u|^4) \end{align*}
which follows from Lemma \ref{lemma2}.  Second, the Taylor expansion of $f$,
\[ f(y) = f(x) + \pderiv{f}{u^i}u_i + \mathcal{O}(|u|^2) \]
and finally, by Theorem \ref{volume} we have the following expansion of the volume form
\[ \dV(y) = 1 - (m-1)H(x) u^m +\mc{O}(|u|^2).\]
The product of these three terms appears inside the integral, so multiplying the three expansions and making the change of variables $u \mapsto \epsilon u$, we find the order-$\epsilon^0$ term is $k\left(|u|^2 \right)f(x)$ which integrates to $m_0^\partial(x)f(x)$.  The order-$\epsilon^1$ term is,
\begin{align} \epsilon \int_B &k\left(|u|^2\right)\left(\frac{\partial f}{\partial u^i}u^i  - f(x)(m-1)H(x) u^m\right) - k'\left(|u|^2\right)\gdot{\Pi_{\partial \M_t}(U^\top, U^\top),U^\perp}f(x) \, du  \nonumber \\ &= \epsilon m_1^\partial(x) \nabla f(x) \cdot \eta_x  + \epsilon m_1^\partial(x)(m-1)H(x)f(x) - \epsilon m_1^\partial(x)\frac{m-1}{2}H(x)f(x) \nonumber \\ &= \epsilon m_1^\partial(x)\left(\nabla f(x) \cdot \eta_x +\frac{m-1}{2}H(x)f(x)\right) \nonumber \end{align}
where the first equality comes from noting that $u^i$ integrates to zero by symmetry for $1\leq i \leq m-1$ and then applying Lemma \ref{lemma3}.  
\end{proof}

Having proven an asymptotic expansion for the points in the boundary region $\mc{N}_\veps,$ we now turn our attention to the interior region $\M_\veps.$ The expansion in this region was computed previously in \cite{diffusion,hein2007convergence}. Combining these previous results together with the bounds on $\veps$ proven in Lemma  \ref{lemma1} yields the following uniform expansion.
\begin{theorem}\label{thm:hein}(expansion in the interior due to \cite{diffusion,hein2007convergence}) Let $\mathcal{M} \subset \mathbb{R}^n$ be a compact $m$-dimensional $C^3$ Riemannian manifold with a $C^3$ boundary.  Let $k:\mathbb{R}\to\mathbb{R}$ have exponential decay. Suppose that for fixed $\gamma \in (0,1),$ that $\veps^\gamma <\min\{\frac{r_\M}{C_1}, C_\M\}$ Then for all $x\in \M_\veps:= \{x\in \M : d(x,\partial \M)> \veps$, we have
\begin{align}\label{interiorexpansion} \epsilon^{-m}\int_{y\in\mathcal{M}} k\left(\frac{|x-y|^2}{\epsilon^2}\right)f(y) \, \dV &= m_0 f(x) + \frac{m_2}{2} \epsilon^2 \left(S(x)f(x) + \Delta f(x) \right) + \mathcal{O}(\epsilon^3)  \end{align}
where $m_0 = \int_{\mathbb{R}^m} k(|u|)\, du$ and $m_2 = \int_{\mathbb{R}^m} u_1^2 k(|u|)\, du$ are the zeroth and second moments of the kernel and $S(x) = \frac{1}{2}(-R(x) + \frac{1}{2}||\sum_a \Pi(\partial_a,\partial_a)||^2)$ depends on the scalar curvature $R$ and the second fundamental form $\Pi$ at $x$.
\end{theorem}

The results of Theorems \ref{thm1} and \ref{thm:hein}, when taken together, provide a uniform asymptotic treatment of the operator $\mc{I}_\veps$. That is, show that for $\veps$ sufficiently small, the asymptotic behavior of $\mc{I}_\veps f(x)$ can be computed for all points $x\in \M$. This subtle but important notion of uniformity comes from Lemma \ref{lemma1}.


We remark that for small enough values of $\veps$, the both expansions \eqref{interiorexpansion} and \eqref{fullexpansion} hold in normal collar $\mc{N}_{r_C}$ but outside of $\mc{N}_\veps,$.  To reconcile these two expansions, notice that for $\dM_x \gg \epsilon$ we have $m_0^\partial = m_0$ and $m_1^\partial = 0$ up to higher order terms in $\epsilon$.  Thus, outside of the $\epsilon$ neighborhood of the boundary \eqref{fullexpansion} reduces to $m_0 f(x) + \mathcal{O}(\epsilon^2)$ which is consistent with \eqref{interiorexpansion}.

\subsection{Uniformity and Compactness}
Before presenting numerical experiments supporting the expansions of the previous section, we will briefly comment on the role of compactness of $\M$ and uniformity of the expansions in the variable $\veps.$ The authors of this paper speculate that the uniform expansions of this section should be obtainable in the noncompact case assuming some mild conditions such as sectional curvature bounds on $\M$, uniform normal collar, and positive injectivity radius lower bounds.

The expansion \ref{interiorexpansion} in Theorem \ref{thm:hein} was proven in \cite{hein2007convergence} for noncompact manifolds with boundary assuming bounds on the sectional curvature. However, the expansions hold nonuniformly, that is, for every $x\in \M \setminus \partial \M$, there exists an $\veps_x>0$ depending on $x$ such that for all $\veps< \veps_x$ the expansion holds. The only other requirement in their proof is that the embedding map $\iota$ have uniformly continuous inverse. This condition is equivalent to assuming a lower bound on minimum radius of curvature, which is the condition that appears in \cite{hein2007convergence}.

The fact that expansion (\ref{interiorexpansion}) is uniform in the variable $\veps$ is a consequence of Lemma \ref{lemma1} in this paper, which requires compactness of $\M$ in two different parts, both of which could potentially be generalized to the noncompact case. 
First, one requires a universal lower bound on the injectivity radii of $\M_\veps$ for all $\veps$ as shown in the proof of Proposition \ref{prop:atlaslocal}. A sufficient condition for this to occur in the noncompact case is to assume that the double $D(\M)$ admits a Riemannian metric which smoothly extends the metric on $\M$ and has a positive injectivity radius. This is manifestly true in the compact case.

Second, the argument in the proof of Proposition \ref{prop:semilocalization} computing a positive lower bound on the $d_\partial$ function uses compactness of the closed collar region $\mc{N}_{r_{\tn{sem}}}$. However, the authors conjecture that through using a triangle comparison argument using techniques from geometry of CAT$(\kappa)$ spaces  \cite{alexander1993geometric,alex2019alexandrov}, one may be able to derive such a positive lower bound using only bounds on the curvature of $\M$ and inward-facing sectional curvature of $\partial \M$.

\subsection{Examples}

We now provide some simple numerical examples which verify the new boundary expansion in \eqref{fullexpansion}.  We start with the interval, which is a flat manifold with a zero dimensional boundary, so the mean curvature $H(x)=0$.  We then consider a filled ellipse, so that the boundary has nontrivial curvature, but the manifold is still flat in the Riemannian sense.

\begin{example}[Interval] In Fig.~\ref{figure1} we verify \eqref{fullexpansion} using a uniform grid of $N=5000$ data points on the interval $[-1,1]$ and the function $f(x) = x^4$.  Since the grid is uniform, the density is $q(x) = 1/\textup{vol}(\mathcal{M}) = 1/2$ so in this simple example we can correct for the density by multiplying $\mathcal{K}$ by 2.  After computing $2\mathcal{K}f$ we subtract the analytical value of $m_0^\partial(x)f(x)$ and divide by $\epsilon^2 m_2^\partial(x)/2$, which will agree with $\Delta f$ in the interior of the manifold, but blows up like $\epsilon^{-1}$ near the boundary as shown by the solid black curves in Fig.~\ref{figure1}.  In order to obtain a consistent estimator we must also subtract the normal derivative term $m_1^\partial \eta_x \cdot \nabla f(x)$ as shown by the dashed blue curves.  

\begin{figure}[h]
\centering
\includegraphics[width=0.45\textwidth]{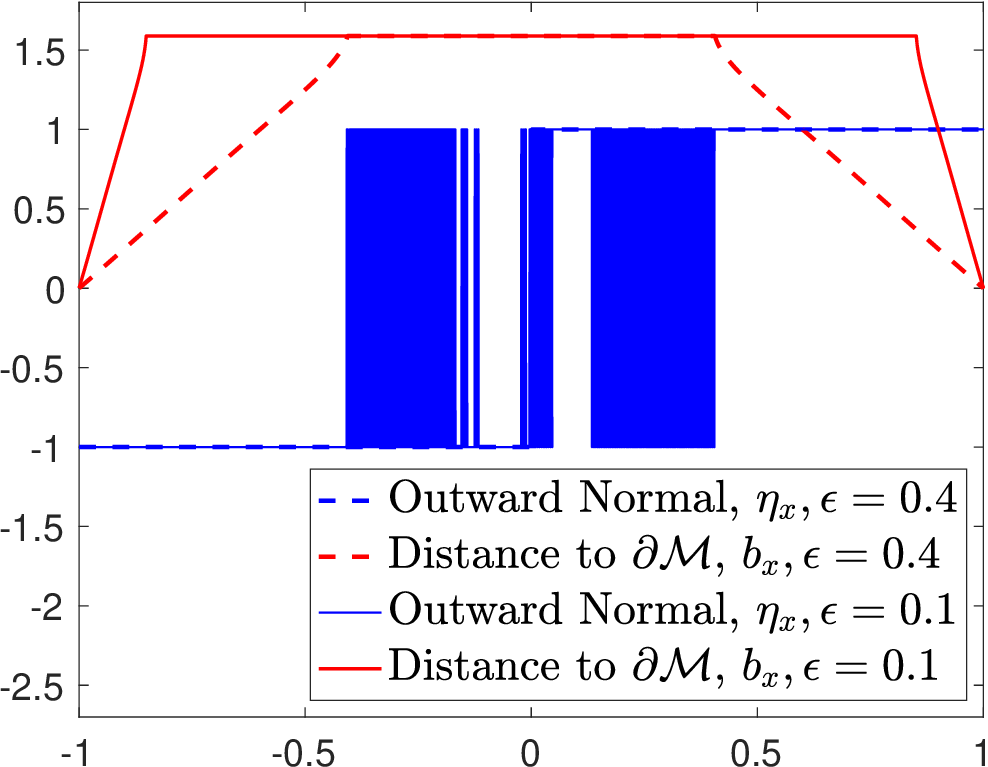}
\includegraphics[width=0.45\textwidth]{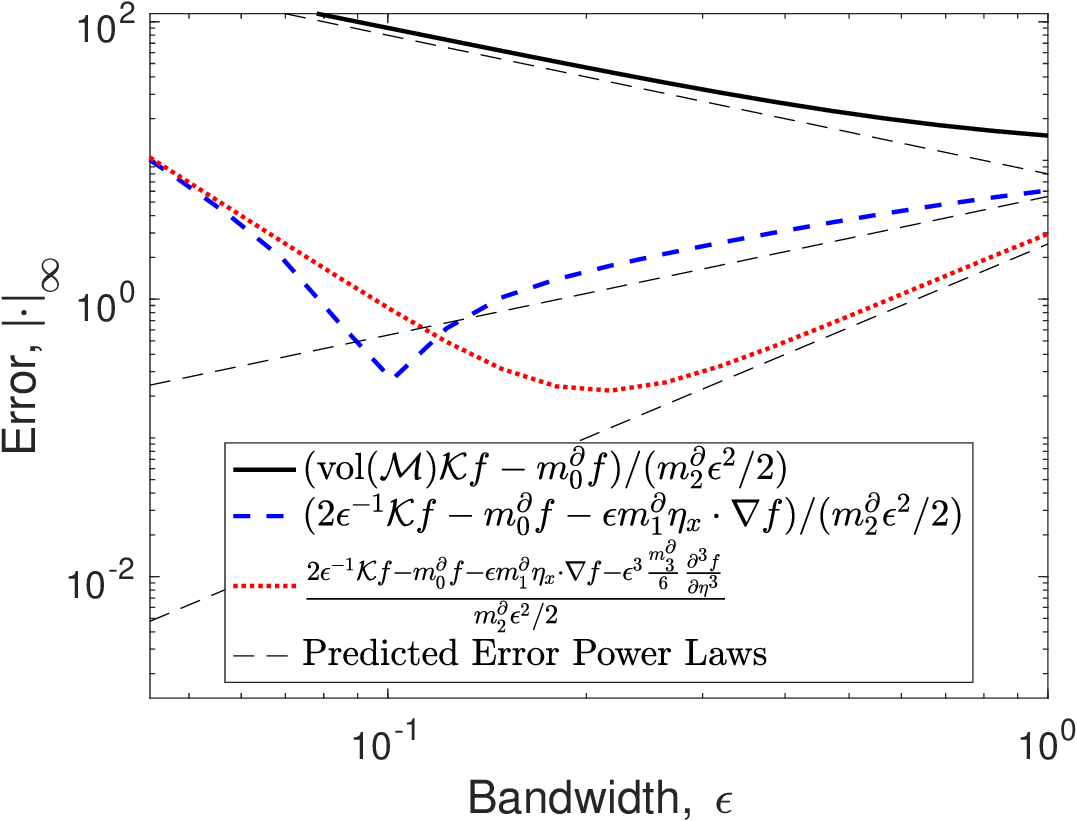} \\
\includegraphics[width=0.45\textwidth]{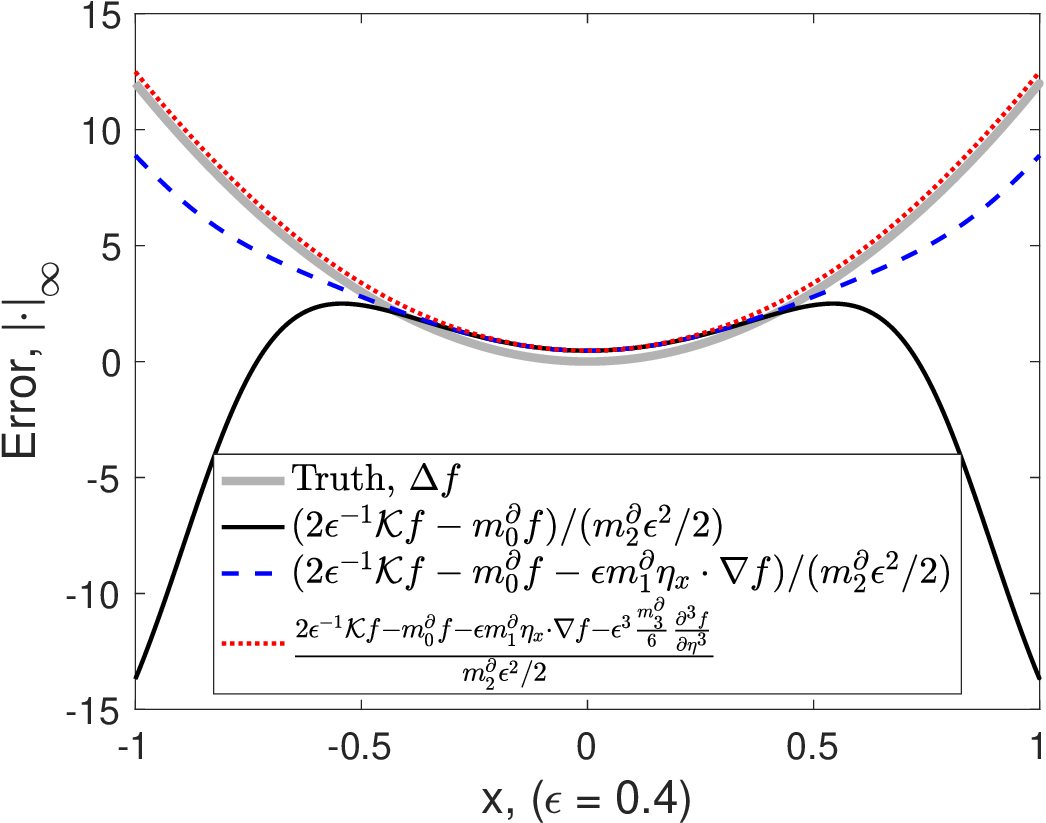}
\includegraphics[width=0.45\textwidth]{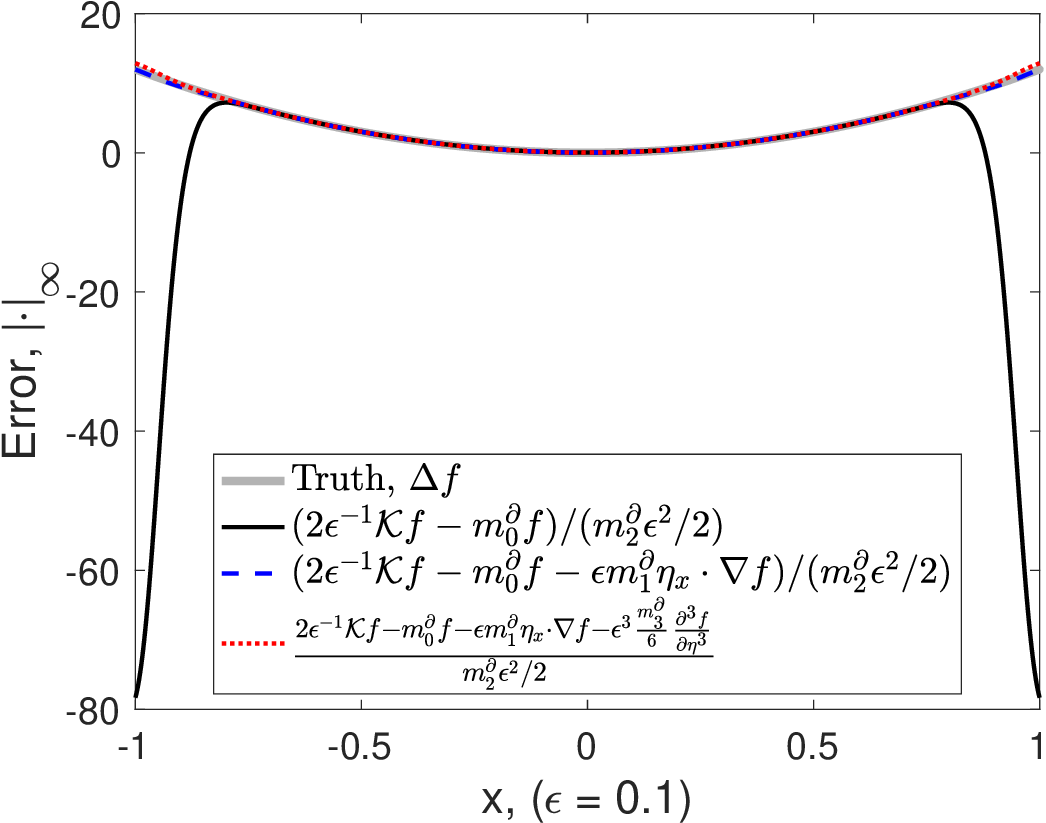}
\caption{Verifying \eqref{fullexpansion} by extracting the Laplacian on the interval $[-1,1]$ applied to the function $f(x) = x^4$. Top, left: We show the estimate of $\dM_x$ and $\eta_x$ from the previous section, the $\dM_x$ estimate saturates when $\dM_x \gg h$ and the $\eta_x$ estimate is very noisy far from the boundary.  Top, right: Error rates for various estimators of $\Delta f$ extracted from $\mathcal{K}f$, for very small $\epsilon$ the quadrature error dominates.  Bottom: True $\Delta f = 12x^2$ compared to various estimates for $\epsilon=0.4$ (left) and $\epsilon=0.1$ (right).}
\label{figure1}
\end{figure} 

\end{example} 

In the next example, in order to eliminate the variance of a single random sample, we generated a uniform grid $\{x_i\}$ on $\M$ and then a very large set of uniform sampled random data points $\{y_j\}$ and computed the kernel $k\left(\frac{|x_i-y_j|^2}{\epsilon^2}\right)$ and the function $f(y_j)$ and then estimated the expected value by 
\[ \int_\M k\left(\frac{|x_i -y|^2}{\epsilon^2}\right) f(y)\, dV(y) = \mathbb{E}\left[k\left(\frac{|x_i -y|^2}{\epsilon^2}\right) f(y)\right] = \frac{1}{N}\sum_{j=1}^N k\left(\frac{|x_i -y_j|^2}{\epsilon^2}\right) f(y_j) + \mathcal{O}(N^{-1/2}). \]
Since the average can be computed iteratively, this strategy allows us to compute the average over $N = 5 \times 10^7$ points and eliminate any variance (quadrature) error.

\begin{example}[Ellipse]  In this example we consider $\M = \{(\tilde x,\tilde y) \, | \, \tilde x^2/a^2 + \tilde y^2/b^2 \leq 1\}$ with $a=1,b=2/3$ where we use $\tilde x,\tilde y$ to denote the coordinates since $x,y \in \RR^d$ denote vectors.  We note that this example is easy to sample uniformly by simply sampling points uniformly in $[0,1]^2$ and then selecting only the points that satisfy the inequality.  
 We start by extracting the mean curvature term using the function $f\equiv 1$ so that $\nabla f \equiv 0 \equiv \Delta f$ and \eqref{fullexpansion} becomes,
\begin{equation}\label{norm1} \epsilon^{-m}\int_{y\in\mathcal{M}} k\left(\frac{|x-y|^2}{\epsilon^2}\right) \, dV = m_0^{\partial}(x) + \epsilon m_1^{\partial}(x)\frac{m-1}{2}H(x) + \frac{\epsilon^2}{2} m_2 \tilde\omega(x) + \mathcal{O}(\epsilon^3) \end{equation}
Using $\epsilon=0.1$ (results were robust for $\epsilon \in [0.5,0.15]$) we estimated the integral as described above and extracted the mean curvature term by subtracting $m_0^\partial$ and dividing by $\epsilon m_1^\partial (m-1)/2$.  In Fig.~\ref{ellipse}(left) we compare the extracted mean curvature with the following analytic derivation.  Note that the boundary of the ellipse can be parameterized as $\iota(\theta) = (a\cos\theta,b\sin\theta)$ with first derivative (and tangent vector) $V^\top = D\iota(\theta) = (-a\sin\theta,b\cos\theta)$ so that the normal vector is $V^\perp = (b\cos\theta,a\sin\theta)$ and second derivative $D^2\iota(\theta) = (-a\cos\theta,-b\sin\theta)$.  Thus, the projection of the second derivative onto the normal direction is $\iota''(\theta)\cdot \frac{V^\perp}{||V^\perp||} = \frac{-ab}{\sqrt{b^2\cos^2\theta + a^2\sin^2\theta}}$. However since we did not use a unit tangent vector we also need to divide by the norm-squared of the tangent vector which yields a mean curvature of 
\begin{align} (m-1)H(x) &= {\rm trace}\left(\left<\Pi(U^\top,U^\top),U^\perp\right>\right) = {\rm trace}\left(\left<\Pi\left(\frac{V^\top}{||V^\top||},\frac{V^\top}{||V^\top||}\right),\frac{V^\perp}{||V^\perp||}\right>\right) \nonumber \\
&= \frac{{\rm trace}\left(\left<\Pi(V^\top,V^\top),V^\perp\right>\right)}{||V^\top||^2 ||V^\perp||}  =\frac{D^2\iota \cdot \frac{V^\perp}{||V^\perp||}}{||V^\top||^2} = \frac{ab}{(b^2\cos^2\theta + a^2\sin^2\theta)^{3/2}} \end{align}
which is simply the standard (extrinsic) curvature of the parameterized curve.  This function is shown as the solid grey curve in Fig.~\ref{ellipse}(left) and compared to the empirically extracted curvature shown as red dots.  This comparison is only valid for points near the boundary, and in Fig.~\ref{ellipse}(left) we only show points with distance to the boundary less than $\epsilon/4$.

Next we verify the derivative terms in \eqref{fullexpansion} by defining a function on the ellipse by $f(\tilde x,\tilde y) = R^3$ where $R \equiv \sqrt{\tilde x^2/a^2 + \tilde y^2/b^2}$ so that $(\tilde x,\tilde y) = (aR\cos\theta,bR\sin\theta)$.  The gradient $\nabla f = (\frac{\partial f}{\partial \tilde x},\frac{\partial f}{\partial \tilde y})$ in the normal direction is
\[ \nabla f \cdot \eta_x = 3R^2(\cos(\theta)/a,\sin(\theta)/b) \cdot \frac{V^\perp}{||V^\perp||} =  3R^2 \frac{(b/a)\cos^2\theta + (a/b)\sin^2\theta}{\sqrt{b^2\cos^2\theta + a^2\sin^2\theta}}  \]
and the Laplacian is
\[ \Delta f = \frac{\partial^2 f}{\partial \tilde x^2}+\frac{\partial^2 f}{\partial \tilde y^2} = 3R((\cos^2(\theta)+1)/a^2 + (\sin^2(\theta) + 1)/b^2). \]
In order to eliminate the curvature terms, we note that multiplying \eqref{norm1} by $f(x)$ matches many of the terms from \eqref{fullexpansion}, so subtracting this from \eqref{fullexpansion} we isolate the terms
\begin{align} \left({\bf K}\vec f\right)_i - f(x_i) \left({\bf K}\vec 1 \right)_i &\to \epsilon^{-m}\int_{y\in\mathcal{M}} k\left(\frac{|x-y|^2}{\epsilon^2}\right)f(y) \, dV - f(x)\epsilon^{-m}\int_{y\in\mathcal{M}} k\left(\frac{|x-y|^2}{\epsilon^2}\right) \, dV \nonumber \\ &= \epsilon m_1^\partial(x) \nabla f(x) \cdot \eta_x + \epsilon^2 \frac{m_2^\partial(x)}{2}\Delta f(x)  \end{align}
where the convergence is as the number of data points, $N\to\infty$.
Using the averaging strategy described above to reduce variance, we estimate $\left({\bf K}\vec f\right)_i - f(x_i) \left({\bf K}\vec 1 \right)_i$ and compare the analytic expressions derived above in Fig.~\ref{ellipse}(right).  This validates the derivative terms in \eqref{fullexpansion}.

\begin{figure}[h]
\centering
\includegraphics[width=0.48\textwidth]{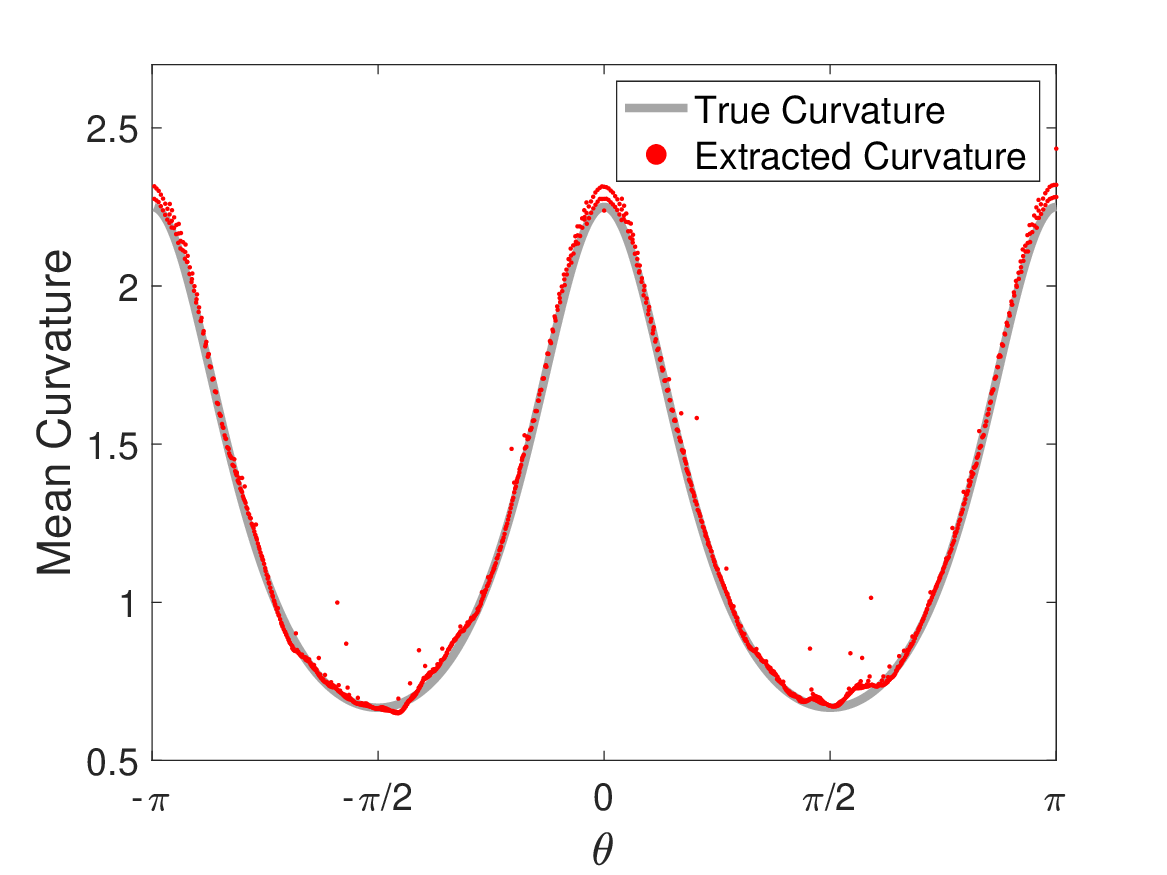}
\includegraphics[width=0.48\textwidth]{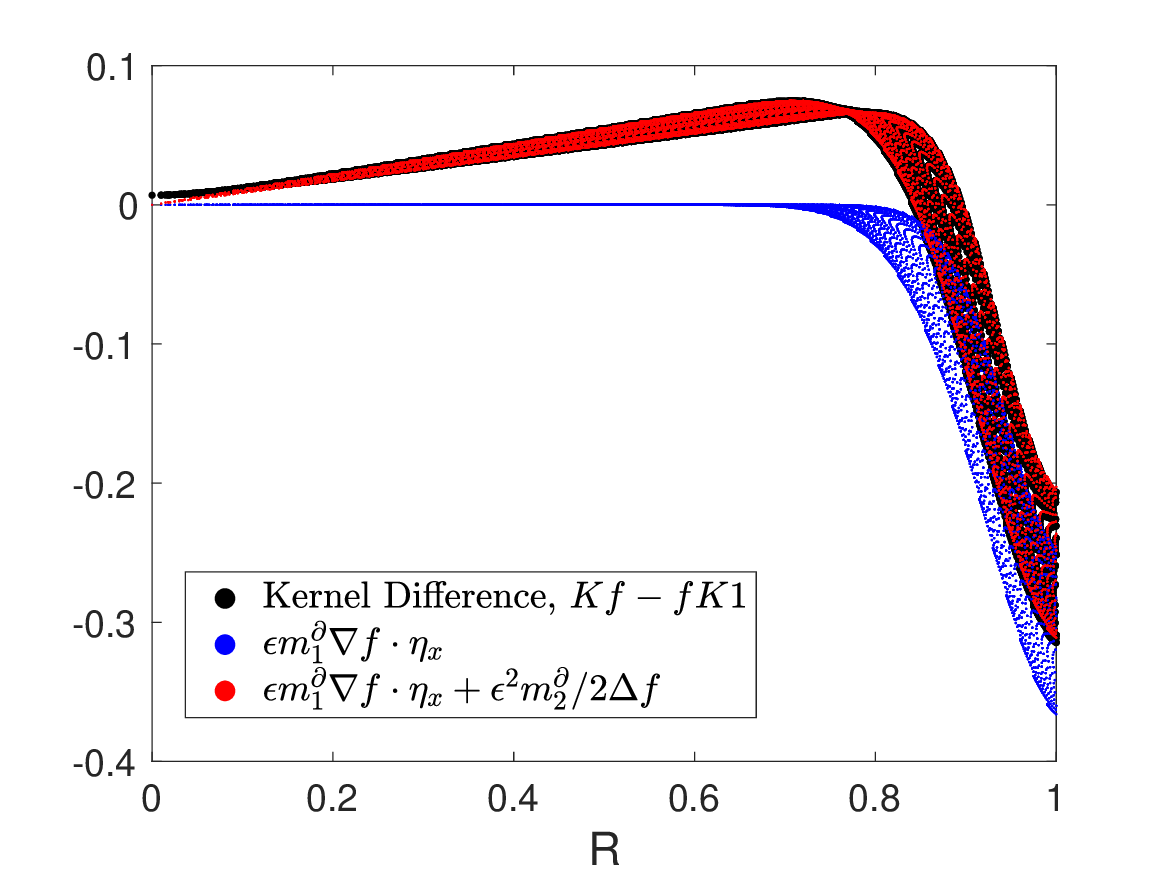} 
\caption{{\bf Left:} Verifying the mean curvature $H(x)$ in the order-$\epsilon$ term of \eqref{fullexpansion} on the ellipse.  {\bf Right:} Verifying the derivative terms in the expansion \eqref{fullexpansion}.}
\label{ellipse}
\end{figure} 

\end{example}

\section{Estimating boundary integrals}\label{boundaryint}
The purpose of this section is to use the distance to the boundary to construct a consistent estimator of the boundary integral $\mathcal{B}(\phi,f) = \int_{\partial\mathcal{M}} \phi(x)f(x)\, d\textup{vol}^\partial$
introduced in equation \eqref{ops} as one of the key bilinear forms required for solving boundary value problems.  It will turn out that the key result of this section will also be required in the next section to prove consistency of the graph Laplacian as an estimator of the Dirichlet energy, $\mathcal{E}$, in the weak sense.  

We saw in the previous subsection that the standard graph Laplacian estimate of the Laplacian on a manifold is not consistent near the boundary.  We now consider the weak form of the operators that the kernel matrix and graph Laplacian are estimating which requires a new result connecting the normal derivative term $m_1^\partial \eta_x \cdot \nabla f$ to a boundary integral.  We first define a boundary integral estimator by
\[ \mathcal{J}(f) = \frac{1}{\epsilon N} \sum_{i=1}^N k\left(\frac{\dM_{x_i}^2}{\epsilon^2}\right) f(x_i) \hspace{40pt} \mathbb{E}[\mathcal{J}(f)] = \frac{1}{\epsilon}\int_{x\in \mathcal{M}} k\left(\frac{\dM_{x}^2}{\epsilon^2}\right) f(x)q(x) \, d\textup{vol}  \]
where $K$ is a kernel with exponential decay as above, for instance $K(z) = e^{-z^2}$ is the prototypical example.  The expectation of the $\mathcal{J}$ functional is the integral over the entire manifold since we assume that the samples $x_i$ yield a weighted quadrature on the manifold.  However, the functional $\mathcal{J}$ uses the distance to the boundary $\dM_x$ to weight the data points,  so that only points near the boundary contribute significantly to the integral.  In practice, in order to compute $\mathcal{J}$, we use the method described in Section \ref{background} to estimate the distance to the boundary.  We first show that $\mathcal{J}$ is a consistent estimator of a boundary integral.

For this result, it will now be convenient to use \textit{boundary normal coordinates}, which are the special case of semigeodesic coordinates constructed on $\partial \M$. In this special case, we also will only need to parameterize the ``height'' of such charts, and so we will let $\epsilon$ parameterize only the $n$-th coordinate $u^n$ in these charts.

\begin{theorem}\label{thm2} In the same context as Theorem \ref{thm1}, let $d\textup{vol}_{\partial}$ be the natural volume element on the boundary inherited from $d\textup{vol}$, then we have
\begin{align}\label{boundaryIntExpansion} \mathbb{E}[\mathcal{J}(f)] = \overline m_0\int_{y \in \partial \mathcal{M}} f(y)q(y) d\textup{vol}_\partial + \epsilon \overline m_1 \int_{y \in \partial \mathcal{M}} f(y) q(y) H(y) - \eta_x\cdot \nabla (fq)(y) d\textup{vol}_{\partial} + \mathcal{O}(\epsilon^2) \end{align}
where $\overline m_0 = \int_0^{\infty} k(u)\, du$ and $\overline m_1 = \int_0^{\infty} u k(u)\, du$ and $H(y)$ is the mean curvature of $\partial \M$ at $y\in \partial \M$.
\end{theorem}

\begin{proof}
Let $0<\gamma< 1$ and $\epsilon>0$ be such that $\epsilon^\gamma$ is less than the normal collar width $R_C$. In addition, let $\mc{N}_{\epsilon^\gamma} = \{y\in \M:b_y < \epsilon^\gamma\}$ denote the normal collar of width $\epsilon^\gamma$. By using an identical argument as Lemma \ref{lemma1}, one can localize the integral over $\mc{N}_{\epsilon^\gamma}$  so that
\begin{align*}
\frac{1}{\veps}  \int_\M k\l(\frac{b_y^2}{\veps^2}\r) f(y) q(y) d\textup{vol} &= \frac{1}{\veps} \int_{\mc{N}_\veps}k\l(\frac{b_y^2}{\veps^2}\r) f(y) q(y) d\textup{vol} +\mc{O}(\veps^z)
\end{align*}
for any choice of $z\in \mbb{N}$.

Now let $\mc{U} = \{U_i\}_{i\in I}$ be a covering of $\partial \M$ in boundary normal coordinate coordinate charts. By taking intersections and complements, we can then generate a covering of $\M$ by measurable sets $\{V_j\}_{j\in J}$  of $\partial \M$ such that $V_j$ are disjoint, and each contained in a single $U_i$. We then extend each $V_j$ to a measurable subset of $\M$ by letting $\ti{V}_j = \phi(V_j\times [0,\veps^\gamma))$. Therefore each $V_j$ can be extended to a measurable set $\ti{V}_j$ of $\M$ which is contained in a boundary normal coordinate chart map $\phi_{U_i}$. The integral over the normal collar can then be parameterized as:
\begin{align}
\int_{\mc{N}_{\veps^\gamma}}k\l(\frac{b_y^2}{\veps^2}\r)f(y)q(y)\dV = \sum_{j\in J} \int_{\ti{V_j}}k\l(\frac{b_y^2}{\veps^2}\r)f(y)q(y)\dV\label{parameter}.
\end{align}

In each of these charts, the coordinate representation of $b_y$ is $u^n$. We then perform an order 1 Taylor expansion about $u^n=0$ of $fq$ as well as a Taylor expansion of $\dV$ about $u^n = 0$ using the first variation of area:

\begin{align*}
\frac{1}{\veps}  \int_{\ti{V}_j} K\l(\frac{b_y^2}{\veps^2}\r) f(y) q(y) d\textup{vol} &= \frac{1}{\veps}  \int_{\phi_{U_i}(\ti{V}_j)} k\l(\frac{(u^n)^2}{\veps^2}\r) f(u^\top,u^n) q(u^\top,u^n) \dV(u^\top, u^n)\\
&= \frac{1}{\veps}\Big(  \int_{\phi_{U_i}(\ti{V}_j)} k\l(\frac{(u^n)^2}{\veps^2}\r)f(u^\top,0) q(u^\top,0)\dV(u^\top, 0)\\
& +  \int_{\phi_{U_i}(\ti{V}_j)} k\l(\frac{(u^n)^2}{\veps^2}\r)f q(u^\top,0) H(u^\top,0) u^n \dV(u^\top, 0)\\
& +  \int_{\phi_{U_i}(\ti{V}_j)} k\l(\frac{(u^n)^2}{\veps^2}\r)\pderiv{fq}{u^n}(u^\top ,0)u^n \dV(u^\top, 0)\\
& +  \int_{\phi_{U_i}(\ti{V}_j)} k\l(\frac{(u^n)^2}{\veps^2}\r) \pderiv{fq}{u^n}(u^\top ,0)H(u^\top,0) (u^n)^2\dV(u^\top, 0)\\
&+  \int_{\phi_{U_i}(\ti{V}_j)} k\l(\frac{(u^n)^2}{\veps^2}\r) \omega(u^\top,\ti{u}^n))(u^n)^2\dV(u^\top, \ti{u}^n) \Big )
\end{align*}
where $\omega(u^\top,\ti{u}^n)$ is the sum of the second-order terms in both expansions with $0\leq \ti{u}^n<\veps^\gamma$. Since $\dV = \sqrt{|g|}$ in coordinates, and the $n$-th coordinate vector field $\partial_n=-\eta_y$ is orthogonal to each of the other coordinate vector fields, cofactor expansion of $\sqrt{|g|}$ implies that $\dV(u^\top,0)= \dV_\partial(u^\top)$. We can then separate terms involving $u^n$ to obtain:
\begin{align*}
\frac{1}{\veps}  \int_{\ti{V}_j} K\l(\frac{b_y^2}{\veps^2}\r) f(y) q(y) d\textup{vol} &= \frac{1}{\veps}\int_{u^n=0}^{u^n=\veps^\gamma} k\l(\frac{u^n}{\veps}\r) du^n \int_{y \in V_j} f(y)q(y) \dV_\partial\\
&+  \frac{1}{\veps}\int_{u^n=0}^{u^n=\veps^\gamma} k\l(\frac{u^n}{\veps}\r)u^n du^n \int_{V_j} f(y) q(y) H(y)\dV_\partial\\
& +  \frac{1}{\veps}\int_{u^n=0}^{u^n=\veps^\gamma} k\l(\frac{u^n}{\veps}\r) u^n du^n \int_{V_j}-\eta_y \cdot \nabla fq(y)\dV_\partial\\
&+  \frac{1}{\veps} \int_{u^n=0}^{u^n=\veps^\gamma} k\l(\frac{(u^n)^2}{\veps^2}\r)(u^n)^2 du^n \int_{y\in V_j} -\eta_y \cdot \nabla fq(y)H(y) \dV_\partial \\
&+ \frac{1}{\veps} \int_{u^n=0}^{u^n=\veps^\gamma} k\l(\frac{(u^n)^2}{\veps^2}\r)(u^n)^2 du^n \int_{y\in W_j} \omega(u^\top,\ti{u}^n) \dV_{\partial \M^{\ti{u}^n}} \\
\end{align*}
Where $W_j$ is the coordinate image of $\partial \M^{\ti{u}^n}$ in these coordinates (recall that $\partial \M_t$ indicates the hypersurface of points distance $t$ away from $\partial \M$.)

Since the integral over $V_j$ does not depend on $\veps$, we may use exponential decay of the kernel to extend the integral over $u^n$ to infinity. 
 By then making a substitution $u^n\mapsto \veps u$, and letting   $\overline m_0 = \int_0^{\infty} k(u)\, du$ and $\overline m_1 = \int_0^{\infty} k(u)u \, du$ we are left with:
\begin{align*}
\frac{1}{\veps}  \int_{\ti{V_j}} k\l(\frac{b_y}{\veps}\r) f(y) q(y) d\textup{vol}&= \overline{m}_0 \int_{\phi_{U_i}(V_j)} f(y)q(y) \dV_\partial\\
&+  \veps \overline{m}_1  \int_{\phi_{U_i}(V_j)} f(y) q(y) H(y)-\eta_y\cdot \nabla fq (y)\dV_\partial\\
&+\mc{O}(\veps^2)
\end{align*}
We remark that 
To compute the integral over the entire normal collar, we return to the parameterization of the integral in \eqref{parameter}:
\[
\int_{\mc{N}_{\veps^\gamma}}k\l(\frac{b_y^2}{\veps^2}\r)f(y)q(y)\dV = \sum_{j\in J} \int_{\ti{V_j}}k\l(\frac{b_y^2}{\veps^2}\r)f(y)q(y)\dV.
\]

Summation over all $\ti{V}_j$ in the manner above, we are left with:
\begin{align*}
\frac{1}{\veps}  \int_{N_{\veps^\gamma}} k\l(\frac{b_y^2}{\veps^2}\r) f(y) q(y) d\textup{vol}&= \overline{m}_0 \int_{\partial \M} f(y)q(y) \dV_\partial\\
&+  \veps\overline{m}_1  \int_{\partial \M} f(y) q(y) H(y)
-\eta_y \cdot \nabla fq(y)\dV_\partial+\mc{O}(\veps^2)
\end{align*}
from which the result follows.

\end{proof}

Theorem \ref{thm2} allows us to reinterpret integrals weighted by functions such as $m_1^\partial \propto e^{-\dM_x^2/\epsilon^2}$ (for which $\overline m_0=\sqrt{\pi}/2$ and $\overline m_1 = 1/2$) as boundary integrals up to higher order terms.  As in the previous section, this estimator is influenced by the sampling density $q$ and we will correct this in the next section.  

Returning to our example of the function $f(x) = x^4$ on the interval $[-1,1]$, in this case the boundary is the set $\{-1,1\}$ so the boundary integral is simply $\int_{\partial \mathcal{M}} x^4 \, d\textup{vol}_{\partial} = 1^4 + (-1)^4 = 2$.  Using the estimator from Theorem \ref{thm2} we can estimate this boundary integral as $\frac{\textup{vol}(\mathcal{M})}{\epsilon m_0}\mathcal{J}(f)$ which will have error of order-$\epsilon$ as shown in Fig.~\ref{figure2}.  The next order term in the expansion \eqref{boundaryIntExpansion} is  $- \epsilon \overline m_1 \int_{x \in \partial \mathcal{M}} \eta_x\cdot \nabla (fq)(x) d\textup{vol}_{\partial} = - 4\epsilon$ so the error in $\frac{\textup{vol}(\mathcal{M})}{\epsilon m_0}\mathcal{J}(f)+4\epsilon$ should be order-$\epsilon^2$ as shown in Fig.~\ref{figure2}.  Again, we emphasize that this example is purely for verification of Theorem \ref{thm2}, we will return to practical computation methods in Section \ref{PDEapps}.

\begin{figure}[h]
\centering
\includegraphics[width=0.45\textwidth]{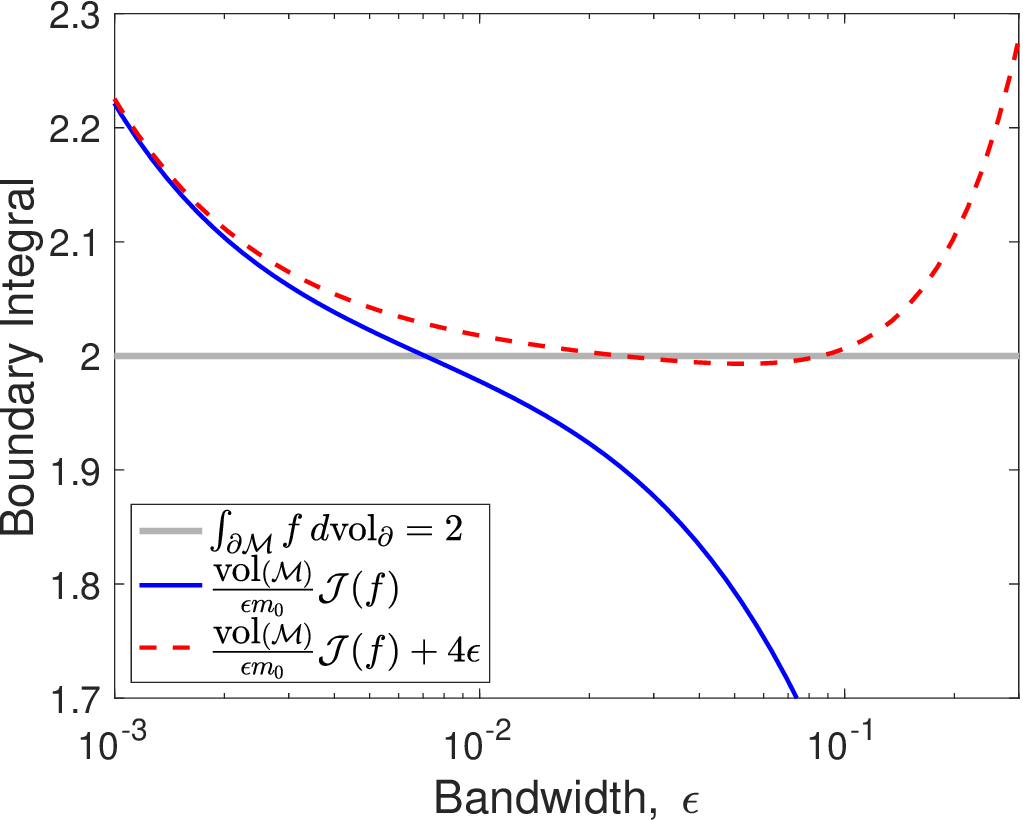}
\includegraphics[width=0.45\textwidth]{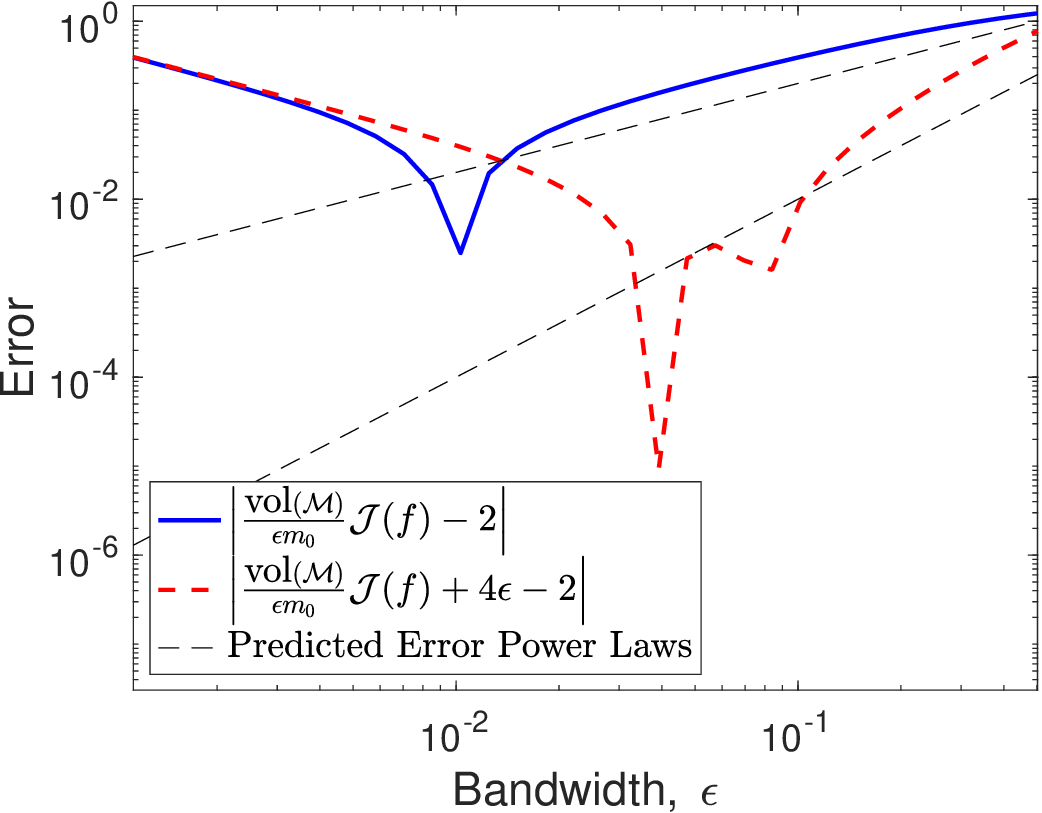}
\caption{Verifying \eqref{boundaryIntExpansion} by extracting the boundary integral on the interval $[-1,1]$ applied to the function $f(x) = x^4$. Left: We show the estimate of the boundary integral compared to the true value, $2$.  Right: Error vs. bandwidth for various estimators of the boundary integral.}
\label{figure2}
\end{figure}

\section{Estimating the Weak Laplacian}\label{estimators}

 The standard method of estimating the Laplace operator is with a graph Laplacian, \[ {\bf L} = c({\bf D}-{\bf K}) \] where $c=\epsilon^{-(m+2)}N^{-1}$ and ${\bf K}_{ij} = K(\epsilon,x_i,x_j)$ is the kernel matrix so that $({\bf K}\vec f)_i \propto \mathcal{K}f(x_i)$ where $\vec f_j = f(x_j)$.  The matrix ${\bf D}$ is diagonal with ${\bf D}_{ii} = ({\bf K}\vec 1)_i \propto \mathcal{K}1(x_i)$.  By Theorem \ref{thm1}, for $x$ near the boundary we have
 \[ \epsilon^{-m}\mathbb{E}[\mathcal{K}(f)(x)] = \mathcal{I}(f)(x) = m_0^\partial(x) f(x)q(x) + m_1^\partial(x) \eta_x \cdot \nabla(fq)(x) + \mathcal{O}(\epsilon^2) \]
and the expected value of ${\bf L}\vec f$ near the boundary is,
 \begin{align*}
 \mathbb{E}\left[({\bf L}\vec f)_i\right] &= \frac{\epsilon^{-m}\mathbb{E}[{\bf D}\vec f] - \epsilon^{-m}\mathbb{E}[{\bf K}\vec f]}{N\epsilon^2} \\
 &= \frac{\epsilon^{-m}\mathbb{E}[f(x_i)\mathcal{K}(1)(x_i)] - \epsilon^{-m}\mathbb{E}[\mathcal{K}(f)(x_i)]}{N\epsilon^2} \\
 &= -\epsilon^{-1}m_1^\partial(x_i)q(x_i) \eta_{x_i} \cdot \nabla f(x_i) + \mathcal{O}(1).
  \end{align*}
Notice that the estimator blows up like $\epsilon^{-1}$ near the boundary.  This was first pointed out in \cite{diffusion} who derived the first two terms of \eqref{fullexpansion}.  They argued that the graph Laplacian is a consistent estimator for Neumann functions (where $\eta_{x_i}\cdot \nabla f(x_i)=0$), which is true but does not explain the empirically observed fact that the eigenvectors of ${\bf L}$ approximate the Neumann eigenfunctions of the Laplacian.  While this pointwise blow-up seems discouraging, notice that it is blowing up with rate $\epsilon^{-1}$ in a neighborhood of the boundary that has volume of order $\epsilon$.  In this section we show that in fact, considered in the weak sense, the graph Laplacian ${\bf L}$ does have a well defined limit, namely the Dirichlet energy or weak Laplacian.

We now use \eqref{interiorexpansion} and \eqref{fullexpansion} together with \eqref{boundaryIntExpansion} to understand the weak form of the kernel operator and the graph Laplacian.  In order to make this connection, we will require the following surprising result connecting the integral of the first moment near the boundary to the second moment in the interior.
\begin{lemma}\label{momentslemma}
For any $z\in \NN$,
\[
\int_0^\infty m^\partial_1(b_x) \ db_x = \frac{m_2}{2} +\mc{O}(\veps^z)
\]
\end{lemma}
The proof of lemma \ref{momentslemma} is included in \ref{thm3proof}.

When we apply the kernel operator to a function $f$ and evaluate at all the data points we obtain a vector $\frac{\epsilon^{-m}}{N}({\bf K}\vec f)_i = \mathcal{K}f(x_i)$ which is simply the kernel matrix multiplied by the vector representation of the function $\vec f_i = f(x_i)$.  If we then take another function $\vec \phi_i= \phi(x_i)$ and compute the inner product, the expectation will be a second integral,
\begin{align}\label{weakform} \frac{\epsilon^{-m}}{N^2} \mathbb{E}\left[\vec \phi^{\top}\,{\bf K}\vec f \right] &= \epsilon^{-m} \int_{x\in \mathcal{M}} \int_{y\in \mathcal{M}} \phi(x)q(x)K\left(\frac{|x-y|}{\epsilon}\right)f(y)q(y) \, d\textup{vol} \, d\textup{vol}.
\end{align}
By expanding the inner integral using \eqref{fullexpansion} and applying Theorem \ref{thm2} we derive the following result.
\begin{theorem}\label{thm3} In the same context as Theorem \ref{thm1} for all $\phi,f \in C^3(\mathcal{M})$ we have,
\begin{align}\label{weakexpansionLaplacian} \frac{2}{m_2 N} \mathbb{E}\left[\vec \phi^{\top}\,{\bf L}\vec f \right] &= \int_{\mathcal{M}} (\nabla\phi \cdot \nabla f) \, q^2 d\textup{vol} + \mathcal{O}(\epsilon).
\end{align}
\end{theorem}
\begin{proof} Recall that ${\bf L} = \epsilon^{-(m+2)}N^{-1}({\bf D} - {\bf K})$ and ${\bf D}_{ii} = \sum_{j=1}^N {\bf K}_{ij}$ so that,
\begin{align*}
     \frac{1}{N}\mathbb{E}\left[\vec \phi^{\top}\,{\bf L}\vec f \right] &= \mathbb{E}\left[\frac{\epsilon^{-m-2}}{N^2}\sum_{i,j=1}^N \phi(x_i)K(\epsilon,x_i,x_j)(f(x_i) - f(x_j)) \right] \\
     &= \epsilon^{-m-2}\int_{x\in\mathcal{M}}\int_{y\in\mathcal{M}} q(x)\phi(x) K(\epsilon,x,y)(f(x)-f(y))q(y) \dV \dV  \\
    &= \epsilon^{-m-2}\int_{x\in\mathcal{M}} q(x)\phi(x)(f(x)\mathcal{K}(1)(x)-\mathcal{K}(f)(x)) \dV 
     \end{align*}
We now segment the manifold into the disjoint union $\mathcal{M} = M_{\epsilon^\gamma}\cup \mc{N}_{\epsilon^\gamma}$ where $N_{\epsilon^\gamma} = B_{\epsilon^{\gamma}}(\partial\mathcal{M})$ is a neighborhood of the boundary and $M_{\epsilon^{\gamma}} = \mathcal{M}\backslash \mc{N}_{\epsilon^{\gamma}}$ is the interior region.  By Theorem \ref{interiorexpansion}, we see that the integral over the interior region is:
\begin{align*}
\epsilon^{-m-2} \int_{x\in \M_{\veps^\gamma}} &\phi(x)q(x) \l(f(x) \mc{I}_\veps (1)(x) - \mc{I}_\veps (f)(x)
\r) \dV  \\ &= \epsilon^{-2}\int_{x\in \M_{\veps^\gamma}} \phi(x)q(x) f(x)\l(m_0 q(x)+ \frac{\veps^2m_2}{2}(S(x)q(x) +\Delta q(x)) +\mc{O}(\veps^3) \r)  \\ &\hspace{15pt} -\phi(x)q(x) \l( m_0f(x)q(x) + \frac{\veps^2m_2}{2} \Big(f(x)q(x)S(x)+\Delta(fq)(x)\Big)  +\mc{O}(\veps^3)\r)    \dV \\
&=-\frac{m_2}{2}\int_{x\in \M_{\veps^\gamma}}q(x)^2\phi(x)\Delta f(x) \dV - \frac{m_2}{2} \int_{x\in \M_{\veps^\gamma}} 2q(x)\phi(x) \nabla f \cdot \nabla q \dV +\mc{O}(\veps).
\end{align*}
Moreover, applying the divergence theorem (for the negative definite Laplacian $\Delta = \textup{div}\circ \nabla$, so that $\int_{\partial} h\eta\cdot \nabla f = \int \textup{div}(h \nabla f) = \int h\Delta f + \nabla h\cdot \nabla f \dV$) we have, 
\begin{align*} -\int_{\M_{\veps^\gamma}}q^2\phi \Delta f \dV &= \int_{\M_{\veps^\gamma}}\nabla(q^2\phi)\cdot \nabla f \dV - \int_{\partial \M_{\veps^{\gamma}}} q^2 \phi \eta \cdot \nabla f \dV_\partial \\
&=\int_{\M_{\veps^\gamma}} q^2\nabla \phi \cdot \nabla f + 2q\phi \nabla q \cdot \nabla f \dV - \int_{\partial \M_{\veps^{\gamma}}} q^2 \phi \eta \cdot \nabla f \dV_\partial \end{align*}
and combining this with the previous equation we have,
\begin{align}
\epsilon^{-m-2} \int_{x\in \M_{\veps^\gamma}} &\phi(x)q(x) \l(f(x) \mc{I}_\veps (1)(x) - \mc{I}_\veps (f)(x)
\r) \dV \nonumber \\ &= \frac{m_2}{2} \int_{\M_{\epsilon^\gamma}} q^2 \nabla \phi\cdot \nabla f \dV -  \frac{m_2}{2}\int_{\partial \M_{\epsilon^\gamma}} q^2 \phi \eta \cdot \nabla f \dV_\partial + \mathcal{O}(\epsilon) \label{interiorpart} 
\end{align}
By Theorem \ref{fullexpansion}, the integral over the boundary region is:
\begin{align*}
\epsilon^{-m-2} \int_{x\in \mc{N}_{\veps^\gamma}} &\phi(x)q(x) \l(f(x) \mc{I}_\veps (1)(x) - \mc{I}_\veps (f)(x)
\r) \dV  \\
& =\epsilon^{-2} \int_{x\in \mc{N}_{\veps^\gamma}} q(x)\phi(x) \Big( m_0^\partial(x)q(x) f(x) + \veps m_1^\partial(x)f(x)\eta_x \cdot \nabla q(x) - \frac{m-1}{2}H(x)f(x)q(x)  \\
&\hspace{15pt}-m_0^\partial(x)q(x) f(x) - \veps m_1^\partial(x)\eta_x \cdot \nabla(fq)(x) + \frac{m-1}{2}H(x)f(x)q(x)+\mc{O}(\veps^2)\Big) \dV  \\
&=-\veps^{-1} \int_{\mc{N}_{\veps^\gamma}}m_1^\partial(x)q(x)^2\phi(x)\eta_x \cdot \nabla f(x) +\mc{O}(\epsilon) \dV. 
\end{align*}
We now notice that the function $m_1^\partial(x)$ is purely a function of the distance to the boundary $b_x$ and that it decays to zero exponentially as $b_x \rightarrow \infty$. Hence, we may apply the result of Theorem \ref{thm2}. We then have that,
\begin{align}
\epsilon^{-m-2} \int_{x\in \mc{N}_{\veps^\gamma}} &\phi(x)q(x) \l(f(x) \mc{I}_\veps (1)(x) - \mc{I}_\veps (f)(x)
\r) \dV \nonumber \\ &=-\veps^{-1} \int_{\mc{N}_{\veps^\gamma}}m_1^\partial(x)q(x)^2\phi(x)\eta_x \cdot \nabla f(x) + \mc{O}(\veps) \dV \nonumber \\
&= -\frac{m_2}{2} \int_{\partial \mathcal{M}} q(x)^2 \phi(x) \eta_x \cdot \nabla f(x)\dV_\partial +\mc{O}(\veps). \label{boundarypart} 
\end{align}
where the integral of $m_1^\partial(x)$ is exactly $\frac{m_2}{2}$ by Lemma \ref{momentslemma}. 
Adding \eqref{interiorpart} and \eqref{boundarypart} we have,
\begin{align}
    \frac{1}{N}\mathbb{E}\left[\vec \phi^{\top}\,{\bf L}\vec f \right] &= \frac{m_2}{2} \int_{\M_{\epsilon^\gamma}} q^2 \nabla \phi\cdot \nabla f \dV -  \frac{m_2}{2}\int_{\partial \overline{\mc{N}}_{\epsilon^\gamma}} q^2 \phi \eta \cdot \nabla f \dV_\partial + \mathcal{O}(\epsilon)
\end{align}
where $\overline{\mc{N}}_{\epsilon^{\gamma}}$ is the closure of the boundary region, and its boundary is the union of the boundary of the interior region with the boundary of the manifold, $\partial \overline{\mc{N}}_{\epsilon^{\gamma}} = \partial \M_{\epsilon^{\gamma}}\cup \partial \mathcal{M}$.  We then apply the Divergence theorem on the closure of the boundary region $\overline{\mc{N}}_{\veps^\gamma}$.
\begin{align*}
\int_{\partial \overline{\mc{N}}_{\veps^\gamma}} q^2 \phi \eta \cdot \nabla f \dV_\partial &=  \int_{\overline{\mc{N}}_{\veps^\gamma}} \textnormal{div}(q^2\phi \nabla f) \dV_\partial \\
&=\int_{x\in \mathcal{M}} 1_{B_{\epsilon^\gamma}(\partial\mathcal{M})}(x) \textnormal{div}(q^2\phi \nabla f) \dV_\partial \\ 
&= \mathcal{O}(\epsilon)
\end{align*}
by Theorem \ref{thm2}, since the indicator function $1_{B_{\epsilon^\gamma}(\partial\mathcal{M})}(x)$ has exponential decay in the distance to the boundary as required by Theorem \ref{thm2}. Gathering the order-$\epsilon$ terms and dividing by $\frac{m_2}{2}$ we have,
\[ \frac{2}{m_2 N}\mathbb{E}\left[\vec \phi^{\top}\,{\bf L}\vec f \right] = \int_{\M_{\epsilon^\gamma}} q^2 \nabla \phi\cdot \nabla f \dV + \mathcal{O}(\epsilon) \]
Similarly, we can add and subtract the integral of the same integrand over the boundary region (since this is order-$\epsilon$) and we obtain,
\[ \frac{2}{m_2 N}\mathbb{E}\left[\vec \phi^{\top}\,{\bf L}\vec f \right] =  \int_{\mathcal{M}} q^2  \nabla \phi\cdot \nabla f  \dV + \mathcal{O}(\epsilon) \]
as desired.
\end{proof}
 
Returning to our simple example of the uniform grid on the interval $[-1,1]$ with $f(x)=x^4$, we verify \eqref{weakexpansionLaplacian} by estimating $\int_{\mathcal{M}} |\nabla f|^2 \, d\textup{vol} = 32/7$ using the graph Laplacian ${\bf L}$.  Notice that the optimal bandwidth for the weak sense estimator is much smaller than the optimal bandwidth for the pointwise estimator which results from the double summation being a lower variance estimator as shown in \cite{berrysauer19}.   

\begin{figure}[h]
\centering
\includegraphics[width=0.45\textwidth]{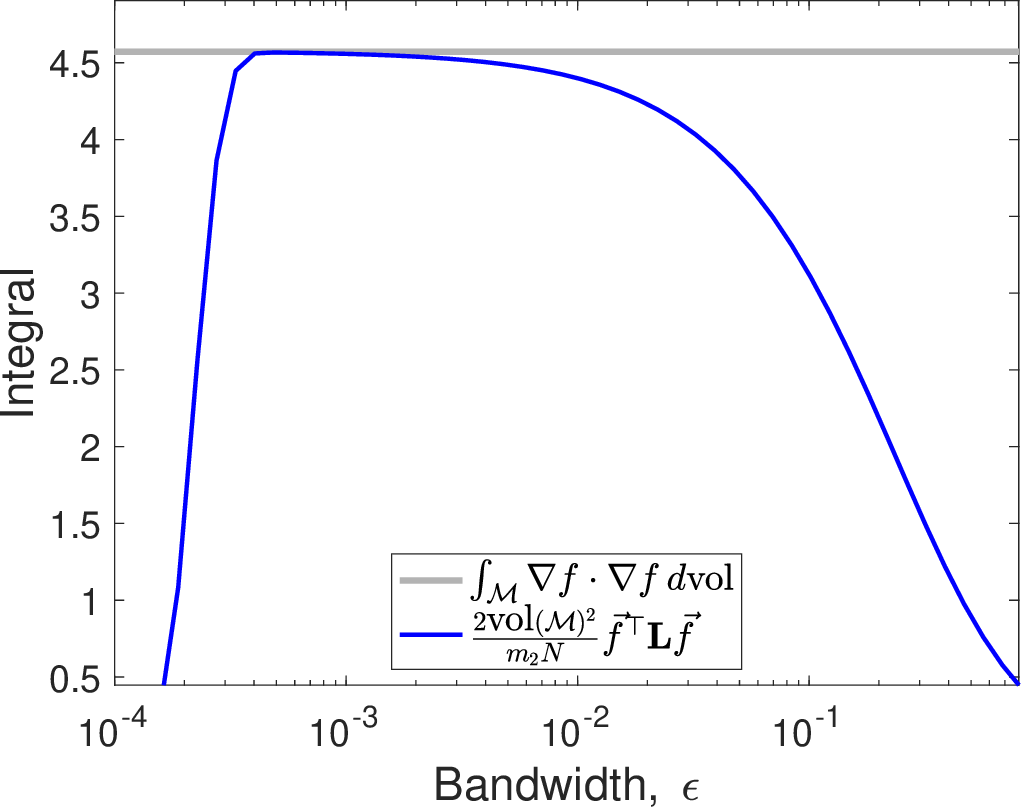}
\includegraphics[width=0.45\textwidth]{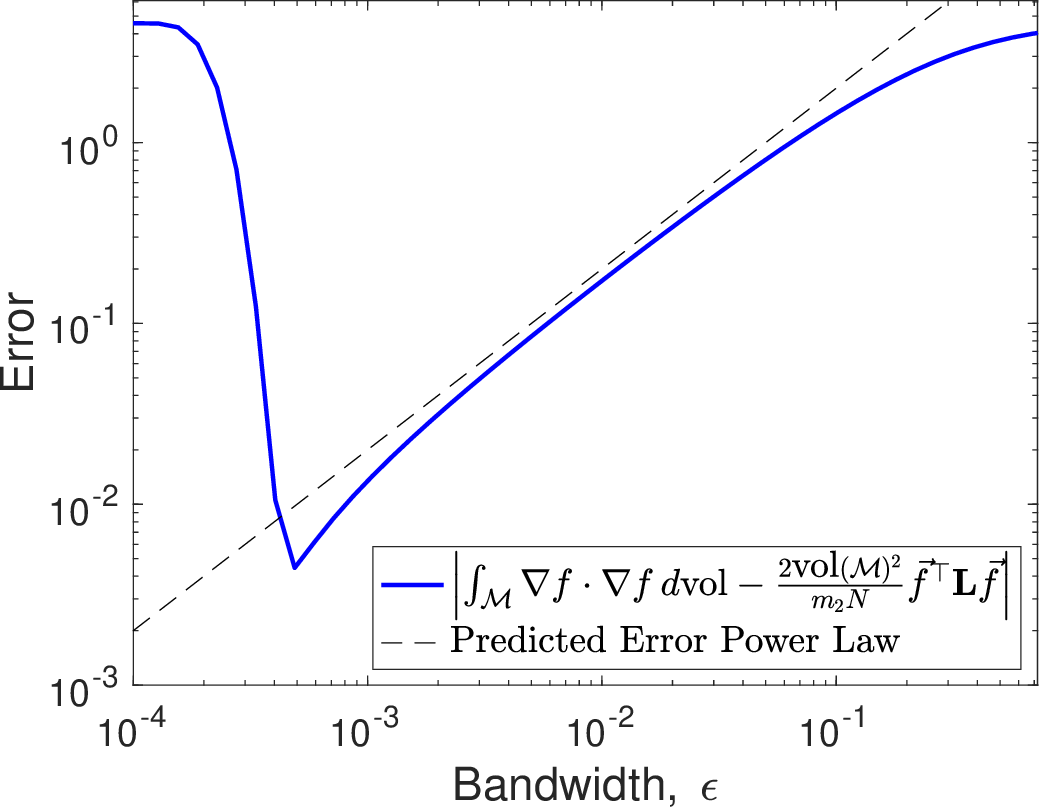}
\caption{Verifying \eqref{weakexpansionLaplacian} computing the exact integral and the graph Laplacian estimator on the interval $[-1,1]$ applied to the function $\phi(x) = f(x) = x^4$. Left: We compare the estimate to the true integral value as a function of the bandwidth parameter, $\epsilon$.  Right: Error vs. bandwidth for the graph Laplacian as a weak-sense estimator.}
\label{figure3}
\end{figure}  

The perhaps surprising conclusion of this section is that, even though the graph Laplacian is not a consistent pointwise estimator of the Laplacian for manifolds with boundary, it is a consistent weak-sense estimator.  We should note that, if one simply removes the boundary and considers the interior, the graph Laplacian will be consistent pointwise at each point of the interior, however the rates of convergence will not be uniform, and the the bandwidth required for pointwise consistency will decrease to zero as you approach the boundary.  

\subsection{Correcting for the sampling density}\label{density}

In \eqref{interiorexpansion}, \eqref{fullexpansion} and \eqref{weakexpansionLaplacian} all the terms are influenced by the density $q$, so to remove this influence we apply the `right-normalization' introduced by \cite{diffusion}.  The idea is to apply \eqref{fullexpansion} to the function $f \equiv 1$ in order to extract a density estimate.  Computationally this means multiplying the kernel matrix by a vector $\vec 1$ of all ones, or equivalently summing the rows of the kernel matrix, which are the diagonal entries of the diagonal matrix ${\bf D}_{ii} = ({\bf K}\vec 1)_i$.  We then normalize the kernel matrix ${\bf K}$ to form the matrices ${\bf \hat K, \hat D, \hat L}$ given by, 
\[ {\bf \hat K} = {\bf D^{-1}KD^{-1}}  \hspace{15pt}\textup{ and }\hspace{15pt} {\bf \hat D}_{ii} = ({\bf \hat K}\vec 1)_i \hspace{15pt}\textup{ and }\hspace{15pt} {\bf \hat L} = \epsilon^{m-2}N^{-1}({\bf\hat D}-{\bf \hat K}). \]
Notice that since ${\bf D}_{ii}$ is proportional to $q(x_i)$ we are essentially pre-dividing by a consistent estimator of $q(x)^2$.  The next theorem shows that this normalization, introduced by \cite{diffusion}, produces a consistent estimator of the Dirichlet energy that is independent of the sampling density of the points $x_i$.

\begin{theorem}\label{thm4} In the same context as Theorem \ref{thm1} for all $\phi,f \in C^3(\mathcal{M})$ we have,
\begin{align}\label{nonuniform} \frac{2 m_0^2}{m_2 N} \mathbb{E}\left[\vec \phi^{\top}\,{\bf \hat L}\vec f \right] &= \int_{\mathcal{M}} \nabla\phi \cdot \nabla f \, d\textup{vol} + \mathcal{O}(\epsilon).
\end{align}
\end{theorem}
\begin{proof} Set $\hat q(x) = \epsilon^{-m}\mathcal{K}(1)(x) = \epsilon^{-m}\sum_{j=1}^N K(\epsilon,x,x_j)$ so that $\hat q(x_i) = \epsilon^{-m}{\bf D}_{ii}$ and we have
\begin{align*}
     \frac{1}{N}\mathbb{E}\left[\vec \phi^{\top}\,{\bf \hat L}\vec f \right] &= \mathbb{E}\left[\frac{\epsilon^{-m-2}}{N^2}\sum_{i,j=1}^N \phi(x_i)\frac{K(\epsilon,x_i,x_j)}{\hat q(x_i)\hat q(x_j)}(f(x_i) - f(x_j)) \right] \\
     &= \epsilon^{-m-2}\int_{x\in\mathcal{M}}\int_{y\in\mathcal{M}} \frac{q(x)}{\hat q(x)}\phi(x) K(\epsilon,x,y)(f(x)-f(y))\frac{q(y)}{\hat q(y)} \dV \dV  \\
    &= \epsilon^{-m-2}\int_{x\in\mathcal{M}} \frac{q(x)}{\hat q(x)}\phi(x)(f(x)\mathcal{K}(1/\hat q)(x)-\mathcal{K}(f/\hat q)(x)) \dV 
     \end{align*}
     Note that by Theorems \ref{interiorexpansion} and \ref{boundaryexpansion}, on the interior of the manifold we have $\hat q(x) = m_0 q(x) + \mathcal{O}(\epsilon)$ and on the boundary we have $\hat q(x) = m_0^\partial(x)q(x) + \mathcal{O}(\epsilon)$.  So for $x$ in the interior region, the above reduces to 
     \[ \epsilon^{-m-2}\int_{x\in\mathcal{M}} \phi(x)(f(x)\mathcal{K}(1/\hat q)(x)-\mathcal{K}(f/\hat q)(x)) \dV = -\frac{m_2}{2 m_0^2}\int_{x\in \M_{\veps^\gamma}}\phi(x)\Delta f(x) \dV\]
     and for $x$ in the boundary region we have,
     \[ \frac{\epsilon^{-m-2}}{m_0^2}\int_{x\in\mathcal{M}} \phi(x)(f(x)\mathcal{K}(1/\hat q)(x)-\mathcal{K}(f/\hat q)(x)) \dV = -\veps^{-1} \int_{\mc{N}_{\veps^\gamma}}\frac{m_1^\partial(x)}{m_0^\partial(x)}\phi(x)\eta_x \cdot \nabla f(x) +\mc{O}(\epsilon) \dV \]
     and since $\frac{m_1^\partial}{m_0^\partial}$ has the necessary decay away from the boundary, it satisfies the hypotheses if Theorem \ref{thm2} so that,
     \[ \int_{\mc{N}_{\veps^\gamma}}\frac{m_1^\partial(x)}{m_0^\partial(x)}\phi(x)\eta_x \cdot \nabla f(x) +\mc{O}(\epsilon) \dV = \frac{m_2}{2 m_0^2}\int_{\partial\mathcal{M}}\phi(x) \eta_x \cdot \nabla f(x) \, d\textup{vol}_\partial \]
     since the integral of $m_1^\partial$ is $\frac{m_2}{2}$ by Lemma \ref{momentslemma}.  Recombining the integral over the interior region and the integral of the boundary region we have,
     \begin{align*}
          \frac{2m_0^2}{N m_2}\mathbb{E}\left[\vec \phi^{\top}\,{\bf \hat L}\vec f \right] &= -\int_{x\in \M_{\veps^\gamma}}\phi(x)\Delta f(x) \dV +  \int_{\partial\mathcal{M}}\phi(x) \eta_x \cdot \nabla f(x) \, d\textup{vol}_\partial + \mathcal{O}(\epsilon)  \\
          &= \int_{\M_{\epsilon^\gamma}}\nabla\phi\cdot\nabla f \, d\textup{vol} + \int_{\partial \M_{\epsilon^{\gamma}} \cup \partial\mathcal{M}} \phi \eta_x\cdot \nabla f \, d\textup{vol}_\partial + \mathcal{O}(\epsilon)
          \end{align*} 
    and as in the proof of Theorem \ref{thm3}, the boundary integral term above can be converted into an integral of the boundary region $\mc{N}_{\epsilon^{\gamma}}$ by the divergence theorem, and this integral is order-$\epsilon$ by Theorem \ref{thm2}.  Similarly, extending the integral $\int_{\M_{\epsilon^\gamma}}\nabla\phi\cdot\nabla f \, d\textup{vol}$ to all of $\mathcal{M}$ also creates an error of order-$\epsilon$ so we obtain \eqref{nonuniform}.
     \end{proof}

To verify \eqref{nonuniform} in our simple example on $[-1,1]$ we started with a uniform grid of $N=5000$ points and then applied the nonlinear transformation $(x+.05)^{1.2}$ to each point and then shift and scale the resulting grid back to $[-1,1]$.  The result is a nonuniform grid with higher density near $-1$ and lower density near $1$.  In Fig.~\ref{figure4} we show that the normalized graph Laplacian ${\bf \hat L}$ constructed in this section recovers the weak-sense Laplacian for the same example as in the previous section on this nonuniform grid. 

\begin{figure}[h]
\centering
\includegraphics[width=0.45\textwidth]{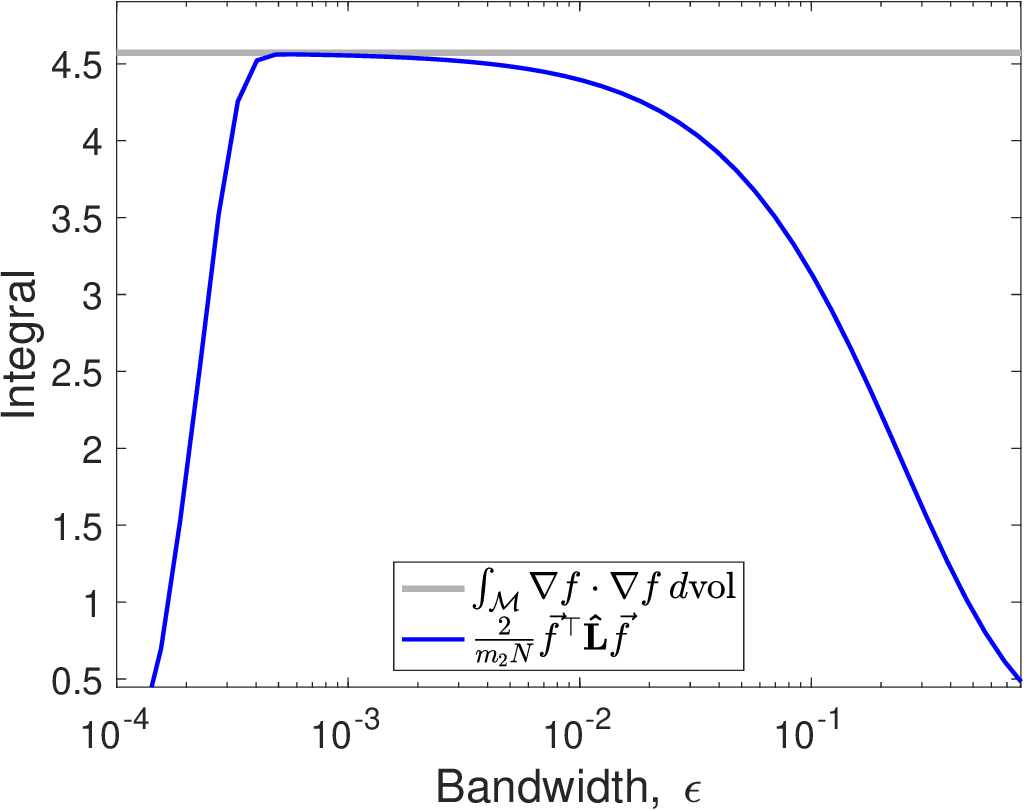}
\includegraphics[width=0.45\textwidth]{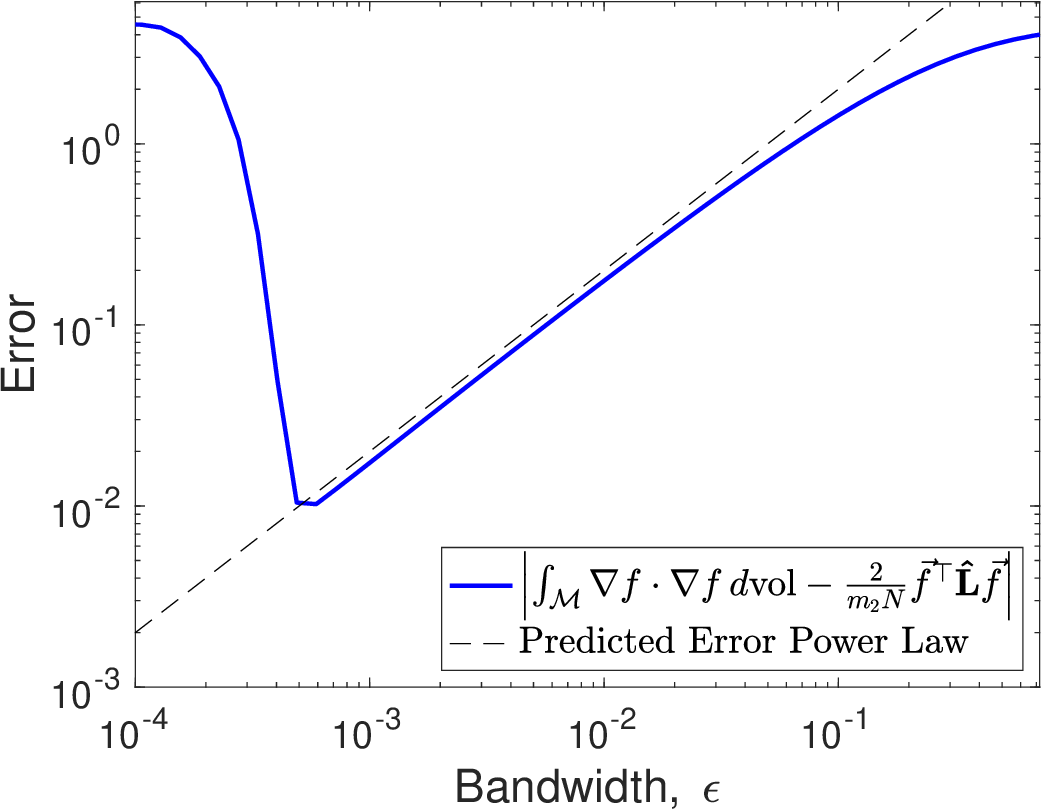}
\caption{Verifying \eqref{nonuniform} computing the exact integral and the graph Laplacian estimator on the interval $[-1,1]$ applied to the function $\phi(x) = f(x) = x^4$. Left: We compare the estimate to the true integral value as a function of the bandwidth parameter, $\epsilon$.  Right: Error vs. bandwidth for the graph Laplacian as a weak-sense estimator.}
\label{figure4}
\end{figure}  


\section{Applications to Elliptic and Parabolic PDEs}\label{PDEapps}
Now that we have shown consistency of the estimator of the weak Laplacian, we provide  numerical experiments to demonstrate the validity of this method. For simplicity of presentation, we shall write  our formulation for problems with Dirichlet and Neumann boundary conditions. However, the approach could be applied to problems with Robin or mixed boundary conditions.

It should be noted that while the previous results of this paper prove consistency of the method, they make no mention of explicit error bounds in terms of the number of data points used in the estimator. In order to provide such bounds, an analysis of the variance of the estimator. Similar variance analyses have been done in the non boundary case in for instance \cite{hein2007convergence,singer,singer2012vector} as well as many others. Although such a variance analysis likely would rely on the results of this paper, it would require a different set of techniques and thus will be explored in a future work.

\subsection{Continuous PDEs}

Let $\Omega \subset \mathbb{R}^n$ be an open bounded set with Lipschitz 
boundary $\partial\Omega$. If $L^2(\Omega)$ denotes the set of square integrable 
functions then we define the Sobolev space 
$
H^1(\Omega) := \{ v \in L^2(\Omega) \, : \, \nabla v \in L^2(\Omega) \}
$
where $\nabla v = (\partial_{x_i}v)_{i=1}^n$ denotes the weak gradient.
We also define a closed subspace of $H^1(\Omega)$ as $H^1_0(\Omega)$
which is the set functions in $H^1(\Omega)$ which are zero on 
$\partial\Omega$ in the trace sense. 
We shall denote the dual space of $H^1(\Omega)$ and $H^1_0(\Omega)$
by $H^1(\Omega)^*$ and $H^{-1}(\Omega)$, respectively and  
the duality pairing between these spaces as $\langle \cdot,\cdot \rangle$. 
Finally $H^{\frac12}(\partial\Omega)$ denotes the standard fractional  
order space. 

%
We start with the elliptic Dirichlet problem: Given 
$f \in H^{-1}(\Omega)$, $g \in H^{\frac12}(\partial\Omega)$ we are interested in solving
    \begin{align}\label{eq:dir}
        \begin{cases}
                 -\Delta u &= f \quad \mbox{in } \Omega \\
                         u &= g \quad \mbox{on } \partial\Omega .
        \end{cases}
    \end{align}
We impose the nonzero boundary condition using the classical lifting 
argument: 
Let $\tilde{g} \in H^1(\Omega)$ denote an extension of $g$ to $\Omega$. 
Notice that due to trace theorem, such an extension exists. 
We then write $u = w + \tilde{g}$ where $w|_{\partial\Omega} = 0$ in the 
trace sense. Then by Lax-Milgram Theorem, under the stated assumptions on 
the data $f,g$, there exists a unique weak solution $w \in H^1_0(\Omega)$ 
to \eqref{eq:dir} in the following sense
    \begin{equation}\label{eq:dir_weak}
        \int_\Omega \nabla w \cdot \nabla v \,dx 
        = \langle f,v \rangle - \int_\Omega \nabla \tilde{g} \cdot \nabla v \,dx  
        \quad \forall v \in H^1_0(\Omega) . 
    \end{equation}
    
Next we turn to the Neumann boundary value problem. Given $f\in H^1(\Omega)^*$,
$g \in L^2(\partial\Omega)$, we consider 
    \begin{align}\label{eq:neu}
        \begin{cases}
                 -\Delta u + u &= f \quad \mbox{in } \Omega \\
           \nabla u \cdot \eta  &= g \quad \mbox{on } \partial\Omega 
        \end{cases}
    \end{align}
where $\eta$ denotes the outward unit normal to $\partial\Omega$. 
Again by Lax-Milgram Theorem it is not difficult to see that under the stated 
assumptions on the data $f,g$ there exists a unique weak solution $u \in H^1(\Omega)$ 
to \eqref{eq:neu} in the following sense: 
    \begin{equation}\label{eq:neu_weak}
        \int_\Omega \nabla u \cdot \nabla v + u v \,dx 
        = \langle f,v \rangle + \int_{\partial\Omega} g v \,ds  
        \quad \forall v \in H^1(\Omega) . 
    \end{equation}
We also state the parabolic homogeneous Dirichlet problem, the Neumann problem is similar and is 
omitted for brevity: Given $f \in L^2(0,T;H^{-1}(\Omega))$ and $u_0 \in L^2(\Omega)$, we consider 
    \begin{align}\label{eq:neu_par}
        \begin{cases}
                 \partial_t u -\Delta u + u &= f \quad \mbox{in } \Omega \times (0,T) \\
          					 u  &= 0             \quad \mbox{on } \partial\Omega \times (0,T) \\
                             u  &= u_0           \quad \mbox{in } \Omega  \; .
        \end{cases}
    \end{align}
The notion of weak solution to \eqref{eq:neu_par} is: find $u \in L^2(0,T;H^1_0(\Omega)) \cap 
H^1(0,T;H^{-1}(\Omega))$ solving 
        \begin{equation}\label{eq:neu_par_weak}
            \langle \partial_t u, v \rangle + 
            \int_\Omega \nabla u \cdot \nabla v  \,dx 
            = \langle f,v \rangle 
            \quad \forall v \in H^1_0(\Omega)  
        \end{equation}
and almost every $t \in (0,T)$.

\subsection{Discrete System}    

Next we describe the linear algebraic systems we obtain after discretization of \eqref{eq:dir_weak}, 
\eqref{eq:neu_weak}, and \eqref{eq:neu_par_weak}.  Assume that we have $N$ nodes sampled on the manifold, and recall that the $N$-by-$N$ matrix ${\bf \hat L}$ and diagonal $N$-by-$N$ matrix ${\bf D}$ 
denote the discrete form of the Laplacian (in weak form) and the discretization of the integral over 
the entire manifold $\mathcal{M}$, respectively. Namely, 
	\[
		\int_\mathcal{M} \nabla u \cdot \nabla v \approx {\bf v}^\top {\bf \hat L} {\bf u}  \quad 
		\mbox{and} \quad 
		\int_\mathcal{M} f v \approx {\bf v}^\top {\bf D} {\bf f} .
	\]
We indicate the discrete boundary integral $\int_{\partial\mathcal{M}}$ as 
	\[
		\int_{\partial\mathcal{M}} g v = {\bf v}^\top {\bf B} {\bf g} 
	\]
where ${\bf B}$ is a diagonal $N$-by-$N$ matrix, whose diagonal entries are very close to zero for nodes that are far from the boundary, and ${\bf v}, {\bf g}$ specify the values of functions on all the nodes in the data set, but can be set to zero on interior nodes for the purposes of estimating the boundary integral since the corresponding diagonal entries of $B$ will be very close to zero. Let $N_{\rm int}$ denote the number of interior degrees of freedom. We note that the interior degrees of freedom can be identified using the estimated distance to boundary function as the nodes, $x_i$, such that $b_{x_i}>\epsilon/2$.  Choosing $\epsilon/2$ means that the boundary layer will be half the width of the $\epsilon$ tube around the boundary, which we found empirically to be an effective width.  The 
discrete form of \eqref{eq:dir_weak} is given by 
	\[
		{\bf \hat L}({\rm f_{dof}},{\rm f_{dof}}) {\bf w} 
			= {\bf D}({\rm f_{dof}},:) {\bf f} - {\bf \hat L}({\rm f_{dof}},:) {\bf \tilde{g}}
	\]
where $f_{\rm dof}$ indicates the interior degrees of freedom.  Then the discrete solution on the 
interior nodes is ${\bf u}({\rm f_{dof}}) = {\bf w}$. We note that $\rm f_{dof}$ refers to the indices of the interior nodes and the colon denotes the inclusion of all the nodes.  So the ${\bf \hat L}({\rm f_{dof}},{\rm f_{dof}})$ matrix on the left hand side of the above equation is $N_{\rm int}$-by-$N_{\rm int}$, whereas the matrices on the right hand side are $N_{\rm int}$-by-$N$.  Note that while these matrices can be quite large (depending on the number of nodes used to sample the domain) both ${\bf D}$ and ${\bf B}$ are diagonal, and ${\bf \hat L}$ is well approximated by a sparse matrix due to the fast decay of the kernel function used to construct it.

The discrete form of the the system \eqref{eq:neu_weak} is
	\[
		({\bf \hat L} + {\bf D}) {\bf u} = {\bf D} {\bf f} + {\bf B} {\bf g}
	\]
	in this case, since we are solving the Neumann problem and $\hat L$ is consistent on the whole domain, we do not have to restrict to the interior nodes, so here all the matrices are $N$-by-$N$.  
	
Finally, we describe the discretization of \eqref{eq:neu_par_weak}. In addition to the spatial 
discretization, we use Backward Euler to discretize in time. Let the number of time sub-intervals
be $K$ and the time step size is $\tau = T/K$. Then given ${\bf u}^0 = {\bf u}_0$, for $k = 1,\dots,K$,
we solve 
	\[
		\left( {\bf D}({\rm f_{dof}},{\rm f_{dof}}) + \tau {\bf \hat L}({\rm f_{dof}},{\rm f_{dof}}) \right) {\bf u}^k 
			= \tau {\bf D}({\rm f_{dof}},:) {\bf f}^k  + {\bf D}({\rm f_{dof}},{\rm f_{dof}}) {\bf u}^{k-1}
	\]
where this last equation is again restricted to interior nodes so that all the matrices are $N_{\rm int}$-by-$N_{\rm int}$ except for ${\bf D}({\rm f_{dof}},:)$ which is $N_{\rm int}$-by-$N$.

\subsection{Numerical Examples}

With the help of several examples, next we show that the approach introduced in this paper, 
can help solve the boundary value problems \eqref{eq:dir_weak}, \eqref{eq:neu_weak}, and
\eqref{eq:neu_par_weak}. In the first 5 examples, we let $\Omega = (0,1)^2$. We first 
consider elliptic problems with both Dirichlet and Neumann boundary conditions. Afterwards,
we illustrate the applicability of our approach on time-dependent PDE with Dirichlet boundary
conditions. 
For numerical approximation in these 5 examples, we partition $\Omega$ into 100 uniform 
cells in each direction. 
Our final example is a semi-sphere with Dirichlet boundary conditions.
\begin{example}[Elliptic homogeneous Dirichlet]
{\rm 
In \eqref{eq:dir} we set $g \equiv 0$, therefore $\tilde{g} \equiv 0$ (cf. \eqref{eq:dir_weak}). 
Consider the exact solution $u(x,y) = \sin(2\pi x) \sin(2\pi y)$, then 
$f(x,y) = 8\pi^2 \sin(2\pi x) \sin(2\pi y)$. The error between the exact solution $u$ and it's
approximation $u_h$ using our proposed method in $L^2$-norm is: $\|u-u_h\|_{L^2(\Omega)}
= 1.541989$e-02. Figure~\ref{f:dhom} shows a visual comparison between the computed
and the exact solution. 
\begin{figure}[h!]
	\centering
	\includegraphics[width=0.7\textwidth]{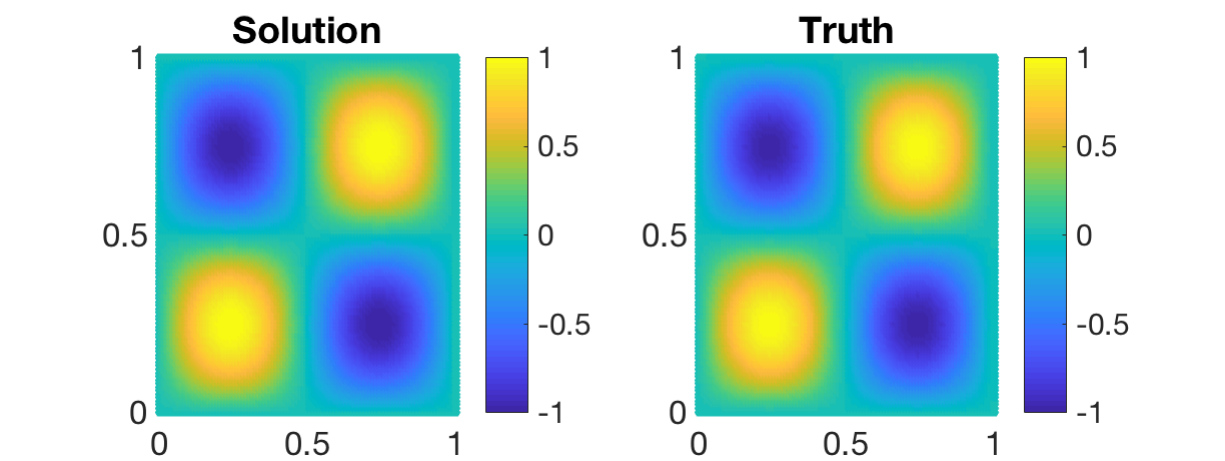}
	\caption{Left: Solution computed using our approach. Right: Exact solution.}
	\label{f:dhom}
\end{figure}
}
\end{example}

\begin{example}[Elliptic nonhomogeneous Dirichlet]
{\rm 
Let the exact solution $u = x^2+y^2$. We set $g = u|_{\partial\Omega}$. The error between the exact solution $u$ and it's
approximation $u_h$ using our proposed method in $L^2$-norm is: $\|u-u_h\|_{L^2(\Omega)}
= 6.378652$e-03. Figure~\ref{f:dnhom} shows a visual comparison between the computed
and the exact solution. 
\begin{figure}[h!]
	\centering
	\includegraphics[width=0.7\textwidth]{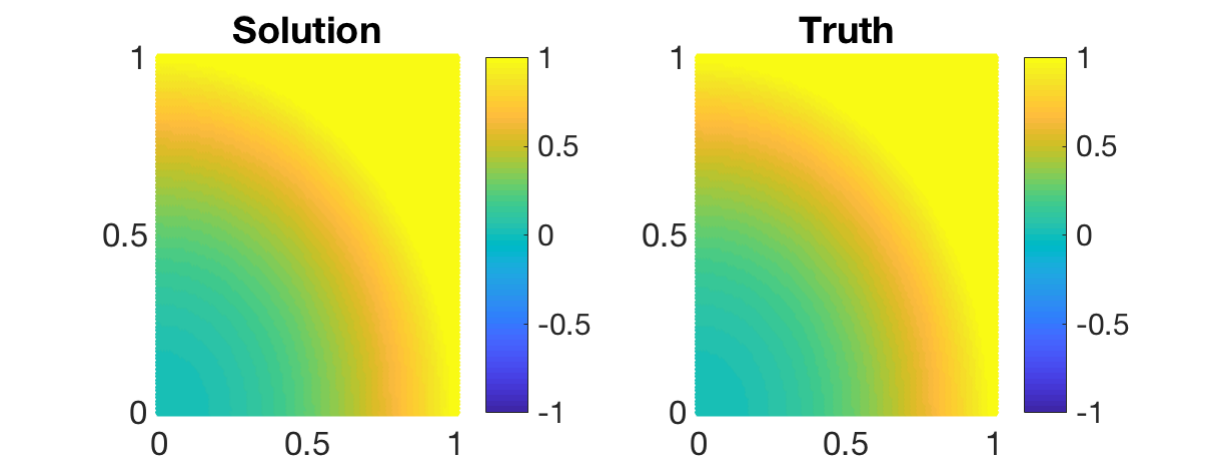}
	\caption{Left: Solution computed using our approach. Right: Exact solution.}
	\label{f:dnhom}
\end{figure}
}
\end{example}

\begin{example}[Elliptic homogeneous Neumann]
{\rm 
In \eqref{eq:neu_weak} we set $g \equiv 0$. Consider the exact solution $u(x,y) = \cos(2\pi x) \cos(2\pi y)$, 
then $f(x,y) = (8\pi^2+1) \cos(2\pi x) \cos(2\pi y)$. The error between the exact solution $u$ and it's
approximation $u_h$ using our proposed method in $L^2$-norm is: $\|u-u_h\|_{L^2(\Omega)}
= 2.125979$e-02. Figure~\ref{f:nhom} shows a visual comparison between the computed
and the exact solution. 
\begin{figure}[h!]
	\centering
	\includegraphics[width=0.7\textwidth]{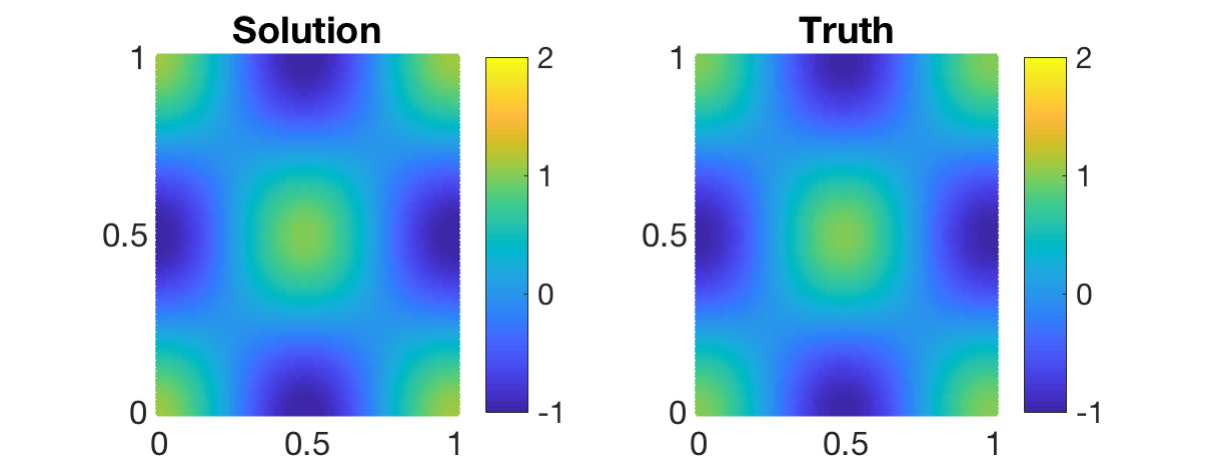}
	\caption{Left: Solution computed using our approach. Right: Exact solution.}
	\label{f:nhom}
\end{figure}
}
\end{example}

\begin{example}[Elliptic nonhomogeneous Neumann]
{\rm 
Let the exact solution $u = x^2+y^2$. We set $g = \nabla u \cdot \eta$. The error between the exact solution $u$ and it's
approximation $u_h$ using our proposed method in $L^2$-norm is: $\|u-u_h\|_{L^2(\Omega)}
= 8.303406$e-02 . Figure~\ref{f:nnhom} shows a visual comparison between the computed
and the exact solution. 
\begin{figure}[h!]
	\centering
	\includegraphics[width=0.7\textwidth]{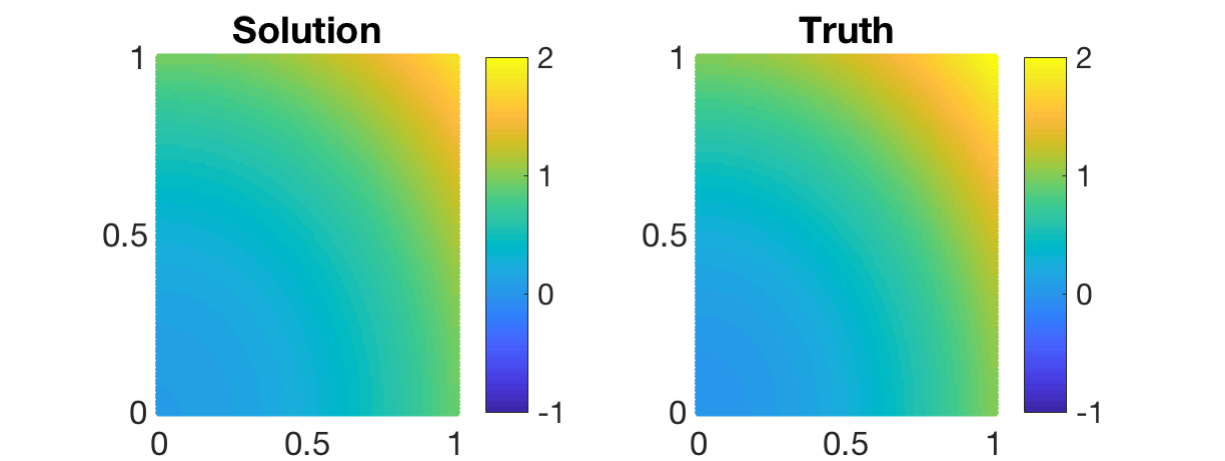}
	\caption{Left: Solution computed using our approach. Right: Exact solution.}
	\label{f:nnhom}
\end{figure}
}
\end{example}

\begin{example}[Parabolic homogeneous Dirichlet]
{\rm 
In \eqref{eq:neu_par_weak} we set $T = 1$. Consider the exact solution 
$u(x,y) = \sin(2\pi x) \sin(2\pi y) e^{-t}$, then $f(x,y) = (8\pi^2-1) u(x,y)$. 
We apply Backward-Euler scheme to do the time discretization with number of time steps equal to 50.
The error between the exact solution $u$ and it's approximation $u_h$ using our proposed 
method in $L^2$-norm is: $\|u-u_h\|_{L^2(0,T;L^2(\Omega))} = 9.902258$e-03.
}
\end{example}

\begin{example}[Dirichlet on the hemisphere]
{\rm 
We now consider \eqref{eq:dir} on the hemisphere $\mathcal{M} = \{(x,y,z) \in \mathbb{R}^3 \, : \, x^2+y^2+z^2 = 1, z\geq 0\}$.  This two dimensional manifold with boundary can be seen as the image of $(\theta,\phi) \in [0,\pi/2]\times [0,2\pi]$ under the embedding function, $\iota:[0,\pi]\times [0,2\pi] \to \mathbb{R}^3$ defined by
$$\iota(\theta,\phi)=
\left[
\begin{array}{c}
\sin\theta \cos\phi \\ \sin\theta \sin\phi \\ \cos\theta 
\end{array}
\right]$$
where $\theta$ is the colatitude and $\phi$ is the azimuthal angle.  The pullback metric in these coordinates is
$$
g(x)=D\iota^T D\iota = 
\left[
\begin{array}{cc}
\cos^2\theta \cos^2 \phi +\cos^2\theta \sin^2 \phi +\sin^2\theta& 0\\
0&\sin^2\theta
\end{array}
\right] = 
\left[
\begin{array}{cc}
1 & 0\\
0&\sin^2\theta
\end{array}
\right].
$$
The Laplacian, $\Delta$, on $\mathcal{M}$ in these coordinates is given by 
\begin{eqnarray*}
\Delta f &=& \frac{1}{\sqrt{|g|}} 
\left[ \dT\ \  \dP \right]
\left(\sqrt{|g|}\ \  g^{-1}
\left[
\begin{array}{c}
\frac{\partial f}{\partial \theta}\\
\frac{\partial f}{\partial \phi}
\end{array} 
\right] \right) = \cot \theta\ \frac{\partial f}{\partial\theta}+\frac{\partial^2 f}{\partial\theta^2}+\csc^2\theta\ \frac{\partial^2 f}{\partial\phi^2}.
\end{eqnarray*}
We can avoid blowup at $\theta=0$ by assuming a solution of the form $u(\theta,\phi) = \sin^2(\theta)\tilde u(\theta,\phi)$, and in this case we consider $u(\theta,\phi) = \sin^2(\theta)\sin(3\phi)/2$ which leads to 
\[ f = -\Delta u = (5/2+3\sin^2(\theta))\sin(3\phi). \]
Using this $f$ as the right-hand-side and using the true value of $u$ on the boundary as a Dirichlet boundary condition, $g=u$ on $\partial\mathcal{M}$, we then solve \eqref{eq:dir} using the estimator of the Laplacian ${\bf \hat L}$.  The resulting solution estimate $u_h$ is compared in Fig.~\ref{f:dirSphere}.
The error between the exact solution $u$ and it's approximation $u_h$ using our proposed 
method in $L^2$-norm is: $\|u-u_h\|_{L^2(\mathcal{M})} = 5.067884$e-03.
\begin{figure}[h!]
	\centering
	\includegraphics[width=0.4\textwidth]{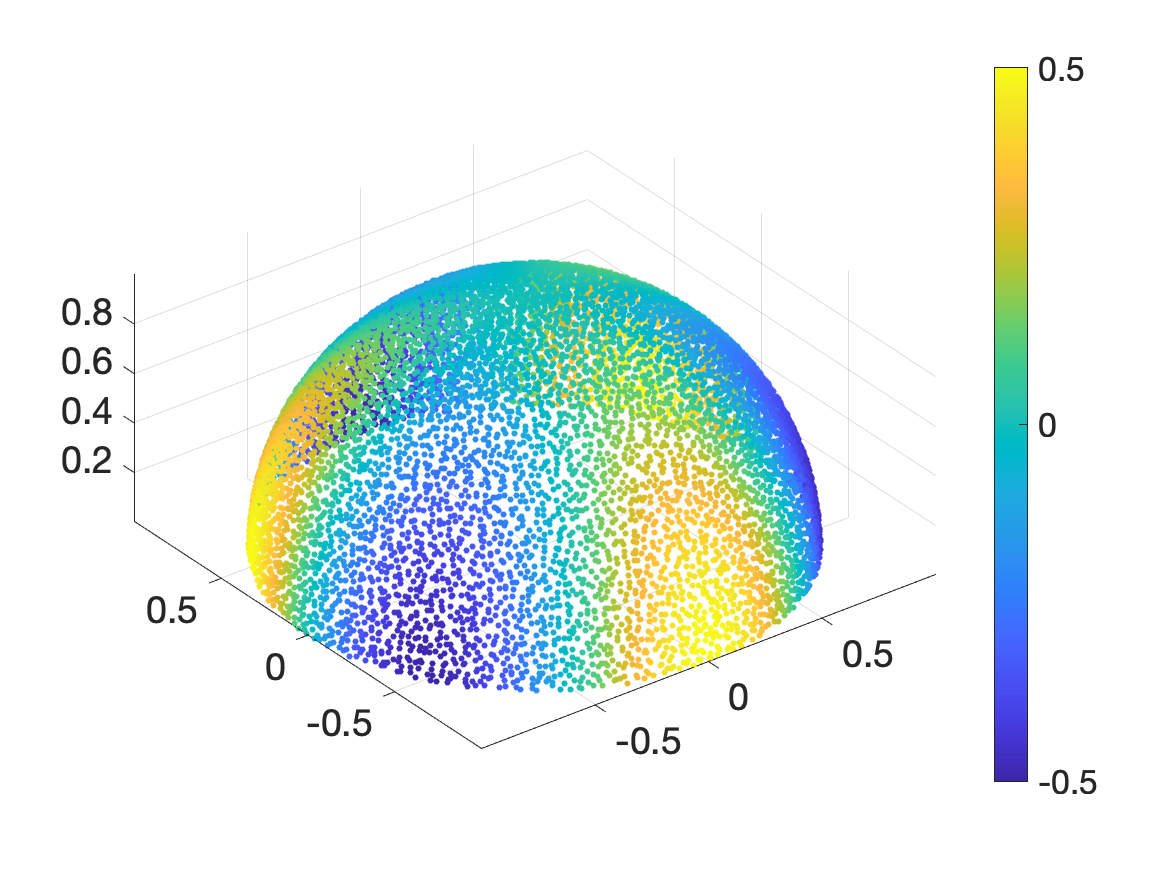}\includegraphics[width=0.4\textwidth]{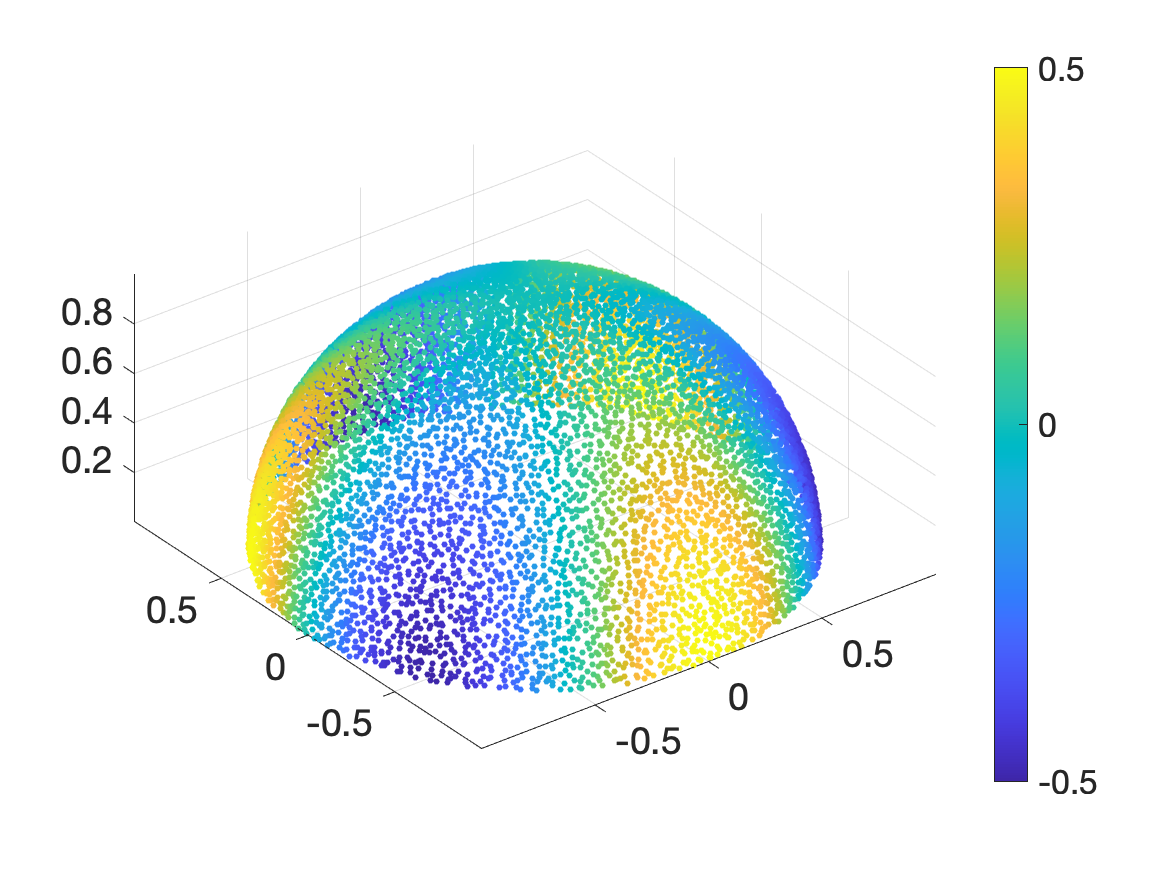}
	\caption{Left: Solution computed using our approach. Right: Exact solution.}
	\label{f:dirSphere}
\end{figure}
}
\end{example}

In all the above examples, we observe that the solutions computed using our approach are highly 
accurate.  We emphasize that the exact same code was used to solve the problem on the hemisphere as was used on the unit square.  This is the advantage of the these diffusion maps based approaches, all that is needed is points sampled on the manifold.

\section{Discussion}
In this work, we have provided an analysis of the bias of the kernel averaging operator associated to the diffusion maps graph Laplacian. One of the main contributions of this work is the use of semigeodesic coordinates for making the bias analysis for points near the boundary. Using these coordinates, we are able to obtain new asymptotic estimates of the kernel averaging operator as well as prove consistency of a boundary integral estimator. It should be noted that our convergence results only address the bias of the estimator, that is, as $\veps \rightarrow 0$. In order to discuss a full treatment of error rates involving the number of data points, a treatment of variance must also be done.

\section{Acknowledgements}
The first author is partially supported by NSF DMS-2153561. The second author is partially supported by NSF DMS-1854204 and NSF DMS-2006808. The last author is partially supported by NSF grants  DMS-1818772, DMS-1913004 and the Air Force Office of Scientific Research 
under Award NO: FA9550-19-1-0036.
\bibliographystyle{plain}
\bibliography{refs}

\appendix

\section{Proof of Lemma \ref{momentslemma}}\label{thm3proof}

In this section we prove Lemma \ref{momentslemma} connecting the integral of the first moment of a kernel near the boundary to the second moment of the kernel in the interior.
\begin{proof}[Proof of Lemma \ref{momentslemma}]
Since $m_1^\partial$ has fast decay, we can localize the integral to an $\epsilon^{\gamma}$ region for any $\gamma \in (0,1)$ by the same argument as in the proof of Lemma \ref{lemma1}.  Thus, we have
\begin{align*}
\int_0^\infty m^\partial_1(b_x) \ db_x  &= \int_0^{\veps^\gamma}m_1^\partial(b_x)  db_x +\mc{O}(\veps^z)\\
&= \int_0^{\veps^\gamma} \int_{-b_x}^{\veps^\gamma}\int_{\RR^{m-1}} u^m k(\norm{u}{\sem}^2)  du^1 ... du^{m-1}du^m db_x
\end{align*}
We notice that there is considerable symmetry in the domain of integration between the $b_x$ and $u^m$ variables. Since $u^mk(\norm{u}{\sem}^2)$ is an odd function with respect to $u^m$, we obtain cancellation of the domain in the area indicated in red in Figure \ref{momentsfig}.
\begin{figure}
\begin{center}
\includegraphics[scale=.2]{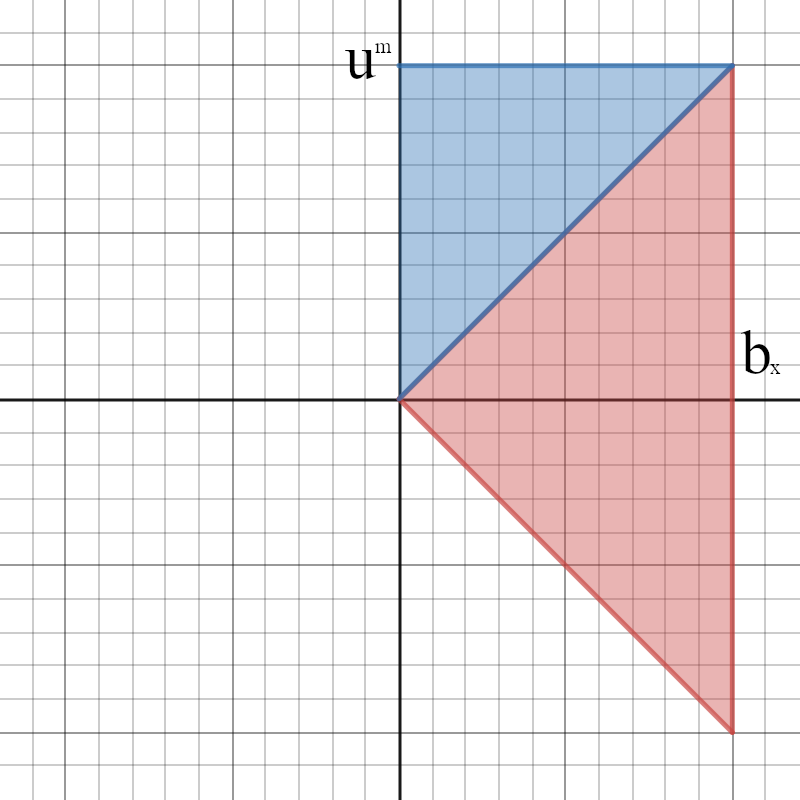}
\end{center}
\caption{\label{momentsfig} The domain of integration for $b_x$ and $u^m$. The red region indicates the domain which cancels due to symmetry. The remaining blue region is evaluated.}
\end{figure}
This leads to the following simplification:
\[
\int_0^{\veps^\gamma} \int_{-b_x}^{\veps^\gamma}\int_{\RR^{m-1}} u^m k(\norm{u}{\sem}^2)  du^1 ... du^{m-1}du^m db_x  = \int_0^{\veps^\gamma} \int_{b_x}^{\veps^\gamma}\int_{\RR^{m-1}} u^m k(\norm{u}{\sem}^2)  du^1 ... du^{m-1}du^m db_x
\]
We then apply Fubini's theorem to obtain,
\begin{align*}
&=\int_0^{\veps^\gamma} \int_0^{u^m}\int_{\RR^{m-1}} u^m k(\norm{u}{\sem}^2)  du^1 ... du^{m-1} db_x du^m\\
&=\int_0^{\veps^\gamma} \int_{\RR^{m-1}} \int_{0}^{u^m} u^m k(\norm{u}{\sem}^2)  db_x du^1 ... du^{m-1}  du^m\\
&= \int_0^{\veps^\gamma} \int_{\RR^{m-1}} (u^m-0) u^m k(\norm{u}{\sem}^2)  du^1 ... du^{m-1}  du^m\\
&= \int_0^{\veps^\gamma} \int_{\RR^{m-1}} (u^m)^2 k(\norm{u}{\sem}^2)du^1 ... du^m\\
&= \frac{1}{2}\int_{-\veps^\gamma}^{\veps^\gamma} \int_{\RR^{m-1}} (u^m)^2 k(\norm{u}{\sem}^2)du^1 ... du^m +\mc{O}(\veps^z) \\
&= \frac{m_2}{2} +\mc{O}(\veps^z).
\end{align*}
where in the last line we expand the integral to all of $\mathbb{R}^m$ and recover the second moment $m_2 = \int_{\mathbb{R}^m} (u^i)^2 k(||u||^2) \, du$ which is the same for any choice of $i$ due to the radial symmetry of the kernel. 
\end{proof}
\section{Estimating the normal vector field and distance to the boundary}\label{background}

In \cite{BERRY2017}, following the results of \cite{diffusion}, the authors extended the expansion \eqref{interiorexpansion} to manifolds with boundary as,
\begin{align}\label{boundaryexpansion} 
    \epsilon^{-m}\mathcal{K}f(x)  = m_0^\partial(x)q(x)f(x) + \mathcal{O}(\epsilon) + \epsilon^{-m}\textup{Error}_{\textup{Quad}}(N,f,q). \end{align}
where the coefficient $m_0^\partial$ is no longer constant but depends on the distance $b_x$ from $x$ to the boundary (defined as the infimum over smooth curves). 
As shown in \cite{BERRY2017}, the coefficients $m_\ell^\partial(x)$ appearing in \eqref{boundaryexpansion} for manifolds with boundary are given in \eqref{moments} (repeated here for clarity)
\begin{align*} 
m^{\partial}_\ell(x) = \int_{\{z  \in \mathbb{R}^m \, | \, z\cdot \eta_x < b_x/\epsilon\}} (z\cdot \eta_x)^\ell k(|z|^2)\, dz =  \int_{\mathbb{R}^{m-1}}\int_{-\infty}^{\dM_x/\epsilon} z_m^\ell k\left(|z|^2 \right) \, dz_m dz_1 \cdots dz_{m-1} \; . 
\end{align*}
The vector field $\eta_x$ is equal to the outward pointing normal when $x \in \partial \mathcal{M}$.  We can smoothly extend the vector field $\eta_x$ to a tubular neighborhood of the boundary called a \emph{normal collar} as discussed in Section \ref{prelims}.  It can easily be seen that when $b_x > \epsilon$ (meaning that $x$ is further from the boundary than $\epsilon$) these reduce to the formulas \eqref{coeff} up to higher order terms in $\epsilon$.

For the exponential kernel,
\begin{equation} \label{expker} k(z) = \exp\left(-z \right) \end{equation}
we can explicitly compute
\begin{align}\label{gaussianm} 
m^{\partial}_0(x) = \frac{\pi^{m/2}}{2}(1 + \textup{erf}(\dM_x/\epsilon)), \quad\quad 
m^{\partial}_1(x) =  -\frac{\pi^{(m-1)/2}}{2} \exp\left(-\frac{\dM_x^2}{\epsilon^2}\right). \nonumber 
\end{align}
These moments are used in the Section \ref{expansion} to extend the expansion \eqref{boundaryexpansion} to higher order terms.

The motivation for the expansion \eqref{boundaryexpansion} in \cite{BERRY2017} was to analyze the standard Kernel Density Estimator (KDE) for manifolds with boundary.  The standard KDE is (up to a constant) given by,
\[ q_{\epsilon,N}(x) \equiv \epsilon^{-m}\mathcal{K}1(x) =  \frac{1}{N \epsilon^{m}}\sum_{i=1}^N k\left(\frac{|x-X_i |^2}{\epsilon^2}\right) \]
and \eqref{boundaryexpansion} implies that
\begin{equation}\label{kdee} \mathbb{E}[q_{\epsilon,N}(x)] = m_0^\partial(x)q(x) + \mathcal{O}(\epsilon). \end{equation}
For manifolds without boundary, $q_{\epsilon,N}(x)$ can be made consistent after dividing by the normalization constant $m_0$ from \eqref{coeff}.  For manifolds with boundary, as a consequence of \eqref{boundaryexpansion} we see that $q_{\epsilon,N}(x)$ is not consistent at the boundary.  In \cite{BERRY2017} it has been shown how to fix this estimator by estimating the distance to the boundary $\dM_x$. We briefly summarize this method since it will be a key tool in constructing boundary integral estimators as defined in Section \ref{boundaryint}.

Since the standard estimator mixes information about the density and distance to the boundary, additional information is needed in order to estimate the density.  Thus, in \cite{BERRY2017} the Boundary Direction Estimator (BDE) was introduced, which is defined as,
\[ \mu_{\epsilon,N}(x) \equiv \frac{1}{N \epsilon^{m}}\sum_{i=1}^N k\left(\frac{|x-X_i |^2}{\epsilon^2}\right)\frac{(X_i-x)}{\epsilon}. \]
Notice that the BDE is a kernel weighted average of the vectors pointing from the specified point $x$ to all the other data points $\{X_i\}$.  The kernel weighting ensures that only the nearest neighbors of the point $x$ contribute significantly to the summation.  Moreover, for data points in the interior of the manifold, and for sufficiently small bandwidth parameter $\epsilon$, we expect the nearest neighbors to be evenly distributed in all the directions tangent to the manifold. The resulting cancellations imply that the summation should result in a relatively small value (it is shown to be order-$\epsilon$ in \cite{BERRY2017} for points further than $\epsilon$ from the boundary). On the other hand, when $x$ is on the boundary, if we look in the direction normal to the boundary we expect all the data points to be on one side of $x$, and thus the BDE will have a significant component in exactly the normal direction (inward pointing since we are averaging vectors pointing into the manifold). 
For points near the boundary (relative to the size of the bandwidth $\epsilon$) this effect will be diminished smoothly until the distance to the boundary becomes greater than the bandwidth and we return to the case of an interior point.  This intuitive description of the behavior of the BDE was made rigorous in \cite{BERRY2017} by showing that
\begin{equation}\label{bdee} \mathbb{E}[\mu_{\epsilon,N}(x)] = \eta_x q(x) m_1^{\partial}(x) + \mathcal{O}(\epsilon \nabla q(x), \epsilon q(x)) \end{equation}
where $\eta_x \in T_x\mathcal{M}$ is a unit vector pointing towards the closest boundary point (for $x\in \partial\mathcal{M}$, $\eta_x$ is the outward pointing normal). Moreover, $\mathbb{E}[\cdot]$ denotes the expected value.  Notice that since $m_1^{\partial}(x) < 0$ and $\eta_x$ is outward pointing, \eqref{bdee} implies that $\mu_{\epsilon,N}(x)$ points into the interior as expected from the above discussion.  

In \cite{BERRY2017} the authors combined the BDE with the classical density estimator. Indeed by dividing $\mu_{\epsilon,N}$ by 
$q_{\epsilon,N}$, the dependence on the true density $q(x)$ cancels and the result depends only on the distance to the boundary $b_x$, namely (dividing \eqref{bdee} by \eqref{kdee}),
\begin{align}\label{eq:expmuq}
\frac{{\mathbb E} \left[\mu_{\epsilon,N}(x) \right]}{\mathbb{E}\left[q_{\epsilon,N}(x)\right]} = \frac{\eta_x m_1^\partial(x) + \mathcal{O(\epsilon)}}{m_0^\partial(x) + \mathcal{O}(\epsilon)} = -\eta_x \frac{\pi^{-1/2}e^{-\dM_x^2/\epsilon^2}}{\left(1+\textup{erf}\left( \dM_x/\epsilon \right) \right)}+O(\epsilon). 
\end{align}  
A significant feature of this approach is that \eqref{eq:expmuq} can be easily estimated without any explicit dependence on the dimension $m$ of the manifold.  By combining the definitions of $\mu_{\epsilon,N}$ and $q_{\epsilon,N}$ above we find,
\[  
    \frac{\mu_{\epsilon,N}(x)}{q_{\epsilon,N}(x)} = \frac{\sum_{i=1}^N k\left(\frac{|x-X_i |^2}{\epsilon^2}\right)\frac{(X_i-x)}{\epsilon} }{\sum_{i=1}^N     k\left(\frac{|x-X_i |^2}{\epsilon^2}\right)} . 
\]

In order to compute the distance to the boundary, we compute the norm of the vector of the previous equation, and applying \eqref{eq:expmuq} we have 
\begin{equation}\label{eq:quadapp} 
   \mathbb{E}\left[ \left| \sqrt{\pi}\frac{\mu_{\epsilon,N}(x)}{q_{\epsilon,N}(x)} \right| \right] = \frac{e^{-\dM_x^2/\epsilon^2}}{\left(1+\textup{erf}\left( \dM_x/\epsilon \right) \right)} + \mathcal{O}(\epsilon). 
\end{equation}
This is now a scalar equation with a known quantity on the left-hand-side, so it remains only to invert the function on the right-hand-side in order to estimate the distance to the boundary $b_x$.  We note that this computation must be performed at each data point since we require an estimate of the distance to the boundary for each of our data points.  

While \cite{BERRY2017} used a Newton's method to solve \eqref{eq:quadapp} for $b_x$, we note that the right-hand-side is very well approximated by the following peicewise function
\[ \frac{e^{-\dM_x^2/\epsilon^2}}{\left(1+\textup{erf}\left( \dM_x/\epsilon \right) \right)} \approx  \left\{ \begin{array}{ll} 1 - 1.15 \frac{\dM_x}{\epsilon} + 0.35 \left(\frac{\dM_x}{\epsilon}\right)^2 & \dM_x < 1.4\epsilon \\ \frac{1}{2}\exp\left(\left(\frac{\dM_x}{\epsilon}\right)^2\right) & \dM_x \geq 1.4\epsilon \end{array} \right. \]
(The above quadratic approximation was derived by interpolating the function at $\frac{\dM_x}{\epsilon} \in \{0,1/2,1\}$ and for $b_x \geq 1.4\epsilon$ the denominator of \eqref{eq:quadapp} is very close to 2.)  Since the quadratic and the exponential are both explicitly invertible, this approximation avoids requiring a numerical inversion of the right-hand-side of \eqref{eq:quadapp}.

\begin{figure}[h]
\centering
\includegraphics[width=0.31\textwidth]{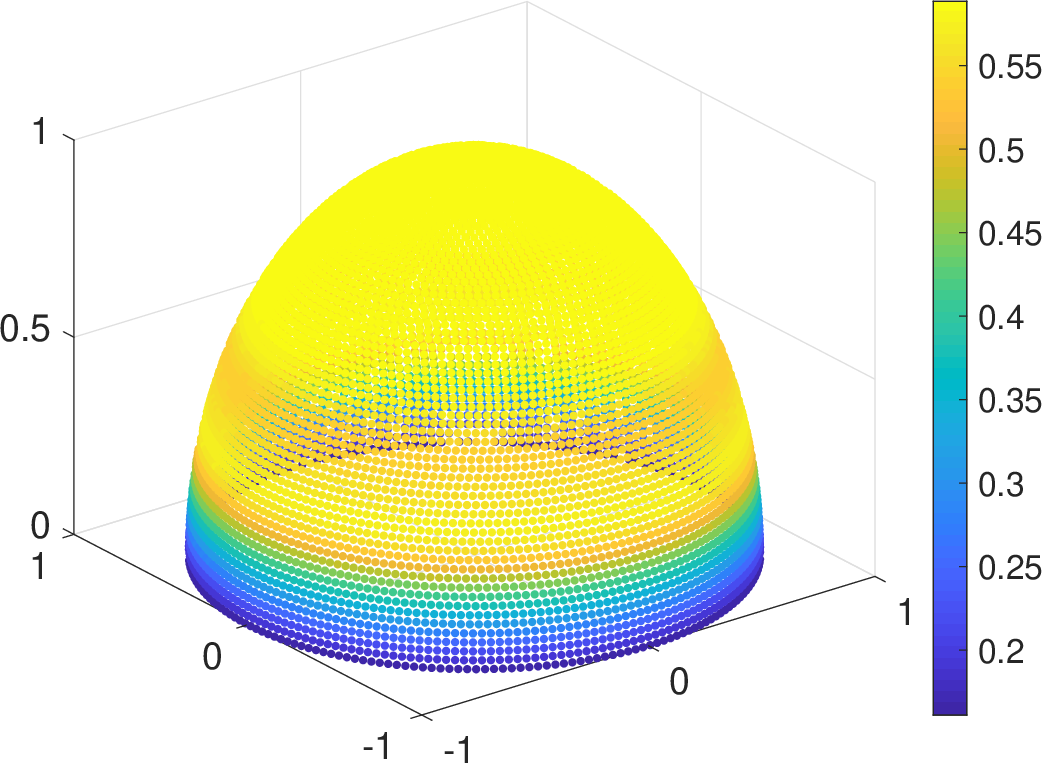}
\includegraphics[width=0.31\textwidth]{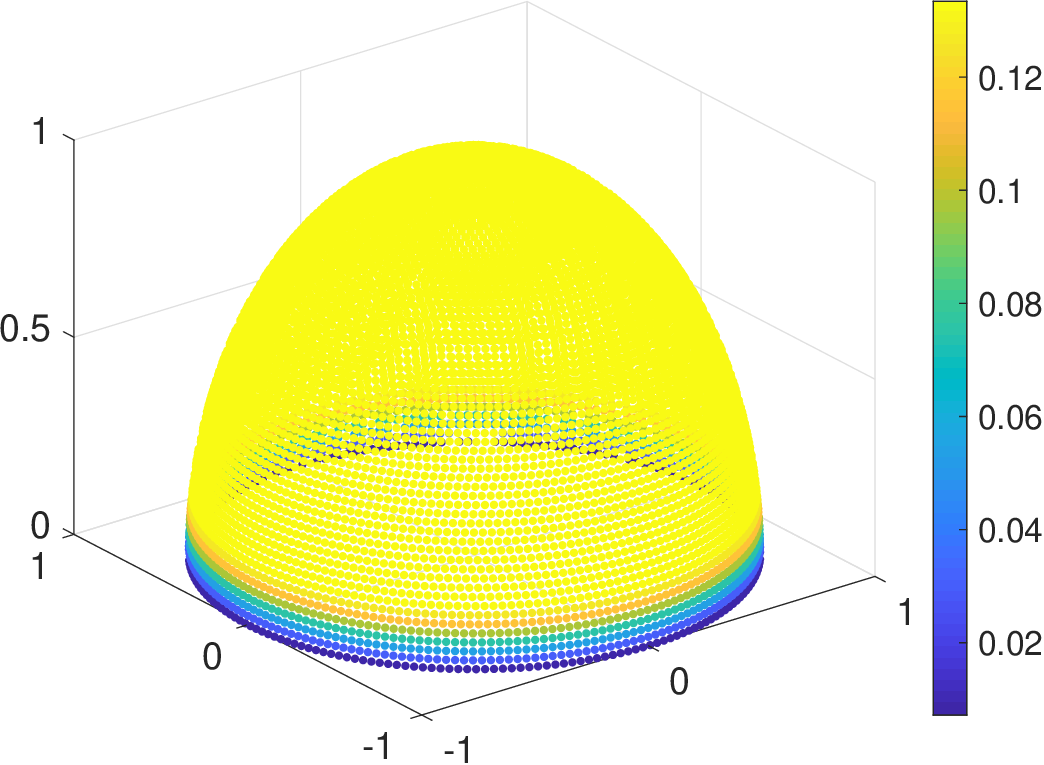} 
\includegraphics[width=0.31\textwidth]{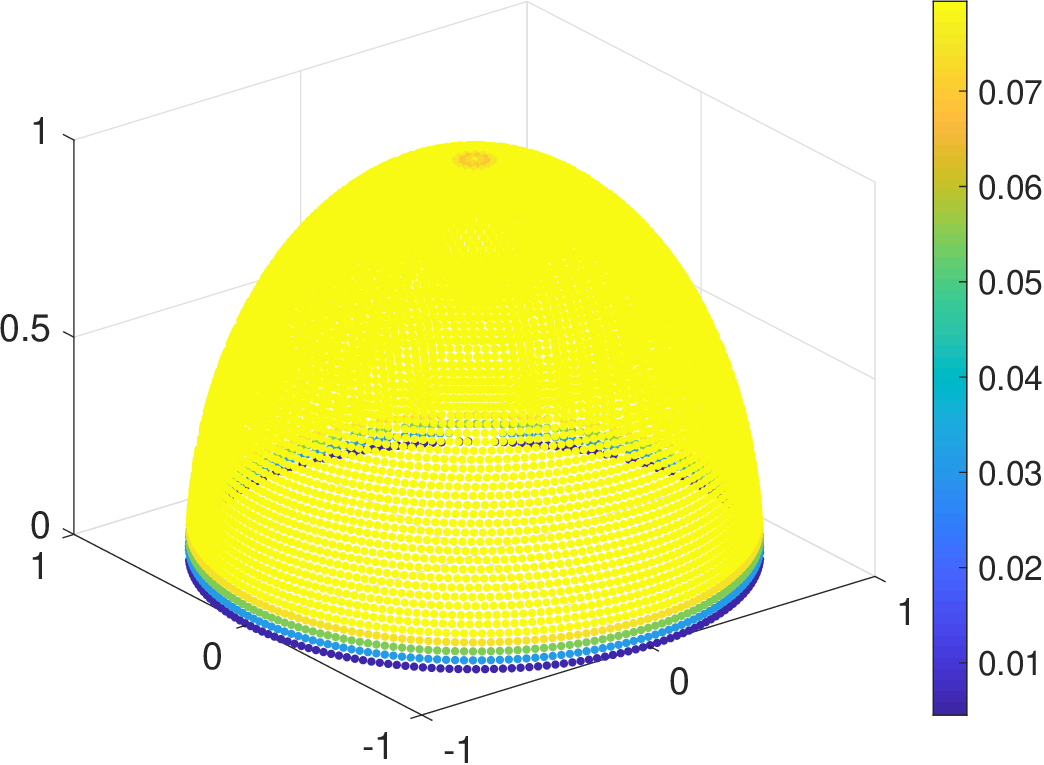} \\
\includegraphics[width=0.31\textwidth]{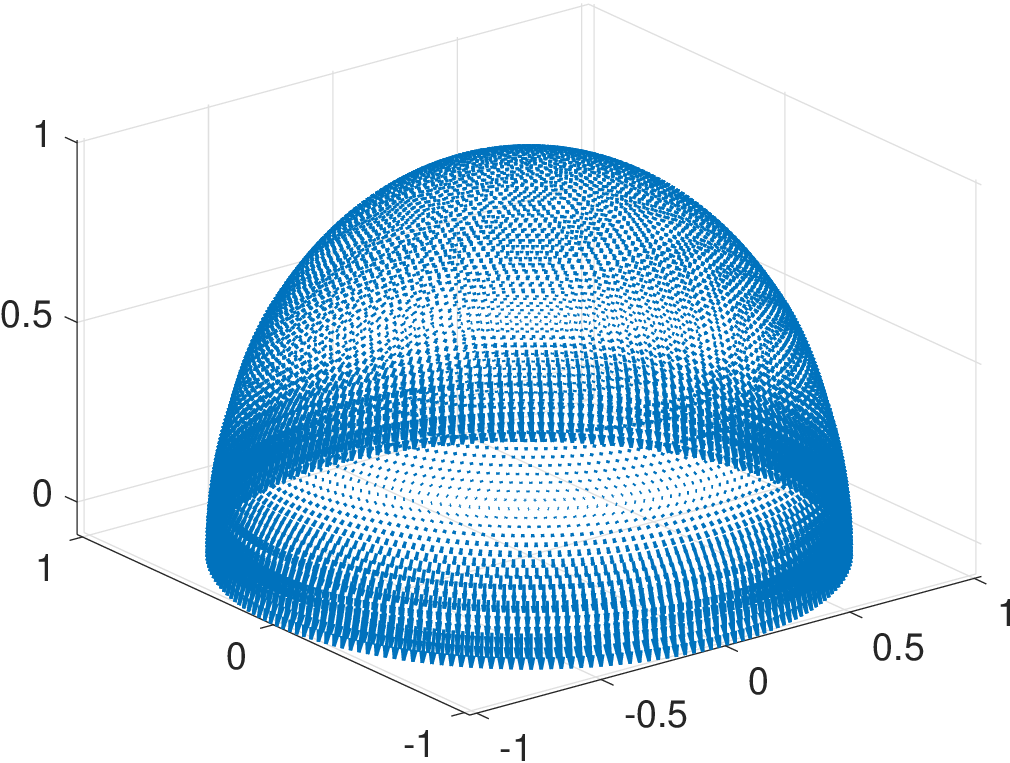}
\includegraphics[width=0.31\textwidth]{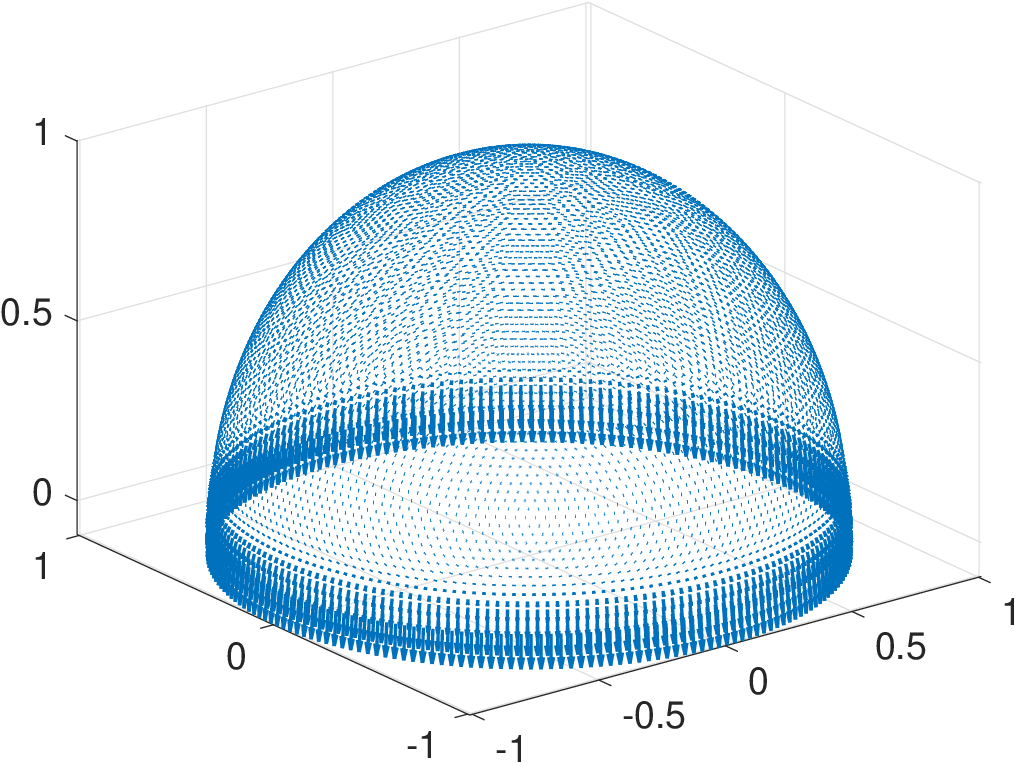} 
\includegraphics[width=0.31\textwidth]{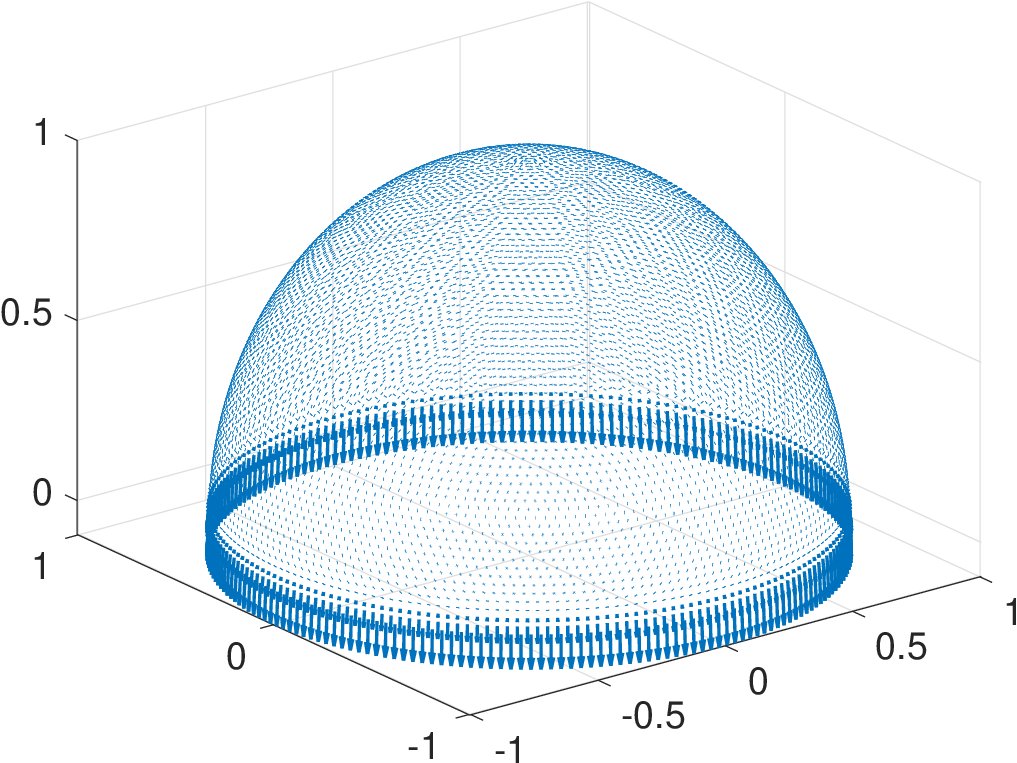}
\caption{Estimating the distance to the boundary (top row) and the normal vector field (bottom) for bandwidth parameters $\epsilon \in \{0.5,0.1,0.05\}$ (left to right).  The estimator is accurate up to a distance of approximately $1.5\epsilon$ from the boundary.}
\label{figure0}
\end{figure}  

We now have consistent estimators for both the direction of the boundary, $\eta_x$, and the distance to the boundary, $b_x$, and these will be essential in imposing the desired boundary conditions for our grid free solvers.

\end{document}